\documentclass[12pt]{amsart}
\usepackage{amssymb,latexsym,mathtools}
\usepackage[margin=1.25in]{geometry}
\usepackage{mathrsfs}
\usepackage{mathtools}
\usepackage{graphicx}
\usepackage{esint}
\usepackage{braket}
\usepackage{url, hyperref}
\newtheorem{theorem}{Theorem}
\newtheorem{lemma}{Lemma}
\newtheorem{definition}{Definition}

\newtheorem{prop}{Proposition}

\newtheorem{remark}{Remark}
\newtheorem{corollary}{Corollary}

\begin{document}

\title[Second Order Level-set PDE in Periodic Media]{On the Homogenization of Second Order Level Set PDE in Periodic Media}
\author[P.S.\ Morfe]{Peter S.\ Morfe}

\begin{abstract}This paper analyzes two classes of second order level set PDE in periodic media in the parabolic scaling.  First, we study fully nonlinear geometric operators under general assumptions in dimension $d = 2$ and prove that the associated equations homogenize in this case.  Next, we treat a class of quasi-linear geometric operators in arbitrary dimensions $d \geq 2$.  In this setting, by adapting arguments form the study of oscillating boundary value problems, we prove that the effective coefficients are generically discontinuous in all dimensions $d \geq 3$.  This necessitates a study of level set PDE driven by operators that are discontinuous at every rational direction on the sphere.  We prove that, in fact, the effective operators so obtained do have a comparison principle and, thus, homogenization occurs.  Finally, we investigate the connection between the effective mobility obtained in the quasi-linear case and linear response, drawing a connection between our results and those obtained in the hyperbolic scaling. \end{abstract}

\date{\today}
\maketitle

\section{Introduction}  

%
In this paper, we are interested in the behavior, as $\epsilon \to 0^{+}$, of the solutions of the parabolically scaled, level set PDE
	\begin{equation} \label{E: level set PDE}
		\left\{ \begin{array}{r l}
			u^{\epsilon}_{t} - F(Du^{\epsilon},D^{2}u^{\epsilon},\epsilon^{-1}x) = 0 & \text{in} \, \, \mathbb{R}^{d} \times (0,\infty), \\
			u^{\epsilon} = u_{0} & \text{on} \, \, \mathbb{R}^{d} \times \{0\}.
		\end{array} \right.
	\end{equation}
Here $F$ is a spatially periodic, geometric operator that has the same ellipticity properties as the mean curvature operator.  (See Section \ref{S: dimension two results} for precise assumptions).  


%

When $d = 2$ and $F$ is a general geometric operator, we prove that there is an effective nonlinearity $\overline{F} = \overline{F}(p,X)$ such that the solutions $(u^{\epsilon})_{\epsilon > 0}$ of \eqref{E: level set PDE} converge as $\epsilon \to 0^{+}$ to the solution $\bar{u}$ of the effective equation
	\begin{equation*}
		\left\{ \begin{array}{r l}
				\bar{u}_{t} - \overline{F}(D\bar{u}, D^{2} \bar{u}) = 0 & \text{in} \, \, \mathbb{R}^{2} \times (0,\infty), \\
				\bar{u} = u_{0} & \text{on} \, \, \mathbb{R}^{2} \times \{0\}.
			\end{array} \right.
	\end{equation*}  
In this case, due to the simpler geometry when $d = 2$, $\overline{F}$ is continuous in $(\mathbb{R}^{2} \setminus \{0\}) \times \mathcal{S}_{2}$.

In dimensions $d \geq 2$, we prove the homogenization of quasi-linear level set PDE with regular coefficients.  Specifically, these equations take the form
	\begin{equation} \label{E: quasilinear}
		\hspace{0.2in}\left\{ \begin{array}{r l}
			m(\epsilon^{-1} x, \widehat{Du^{\epsilon}}) u^{\epsilon}_{t} - \text{tr} \left( A(\epsilon^{-1}x,\widehat{Du^{\epsilon}}) D^{2} u^{\epsilon} \right) = 0 & \text{in} \, \, \mathbb{R}^{d} \times (0,\infty), \\
			u^{\epsilon} = u_{0} & \text{on} \, \, \mathbb{R}^{d} \times \{0\}.
		\end{array} \right. 
	\end{equation}
Here $m : \mathbb{T}^{d} \times S^{d-1} \to (0,\infty)$ is a mobility coefficient and $A$ is obtained from a uniformly elliptic matrix field $a : \mathbb{T}^{d} \times S^{d-1} \to \mathcal{S}_{d}$ through the following relation:
	\begin{equation} \label{E: structure condition}
		A(y,e) = (\text{Id} - e \otimes e) a(y,e) (\text{Id} - e \otimes e).
	\end{equation}  
Though the mobility $m$ can be incorporated into $A$, we will see below that it is illuminating to write the equation in this form.  

We prove that when $m$ and $a$ are regular enough, there are effective coefficients $\overline{m}$ and $\bar{a}$ such that if $\bar{A}$ relates to $\bar{a}$ as in \eqref{E: structure condition}, then the solutions $(u^{\epsilon})_{\epsilon > 0}$ converge as $\epsilon \to 0^{+}$ to the solution $\bar{u}$ of the  effective equation
	\begin{equation} \label{E: effective quasilinear}
		\overline{m}(\widehat{D\bar{u}}) \bar{u}_{t} - \text{tr} \left( \bar{A}(\widehat{D\bar{u}}) D^{2} \bar{u} \right) = 0 \quad \text{in} \, \, \mathbb{R}^{d} \times (0,\infty).
	\end{equation}
In dimensions $d \geq 3$, we prove that the effective coefficients $\bar{a}$ and $\bar{m}$ are generically discontinuous at normal directions $e \in \mathbb{R} \mathbb{Z}^{d}$, so-called \emph{rational directions}.  Therefore, in order to make sense of \eqref{E: effective quasilinear}, we need to study the well-posedness of second order level set PDE driven by operators that are discontinuous at countably many directions on the sphere.  We prove that the comparison principle does indeed extend to equations like \eqref{E: effective quasilinear} and use this to conclude that homogenization occurs.

Finally, we show that, in some cases, the effective mobility $\overline{m}$ describes the linear response of \eqref{E: quasilinear}, precisely as in the work of Spohn \cite{spohn paper} and Katsoulakis and Souganidis \cite{katsoulakis souganidis isotropic}, \cite{katsoulakis souganidis anisotropic}.  Further, when $d = 2$, this allows us to relate $\overline{m}$ to the front speeds obtained in the hyperbolic scaling by Caffarelli and Monneau \cite{caffarelli monneau}.  Again, there are pathologies in rational directions that complicate the picture.

We expect that these results can be generalized to \eqref{E: level set PDE} with fully nonlinear, geometric operators $F$ in arbitrary dimensions $d \geq 3$.  Toward that end, it will be necessary to overcome a number of difficulties that arise when the correctors used in the asymptotic analysis are no longer $C^{2}$.

\subsection{Literature review}  The level-set method for describing geometric flows was introduced by Ohta, Jasnow, and Kawasaki \cite{ohta jasnow kawasaki}, Sethian \cite{sethian thesis}, and Osher and Sethian \cite{osher sethian}.  The first rigorous description of the method was developed by Barles \cite{barles flame propagation} for a first order model of flame propagation and, in the second order setting, by Evans and Spruck \cite{evans spruck} and Chen, Giga, and Goto \cite{chen giga goto}.  See the papers of Barles, Soner, and Souganidis \cite{barles soner souganidis} and Barles and Souganidis \cite{barles souganidis} for the state of the art.  

Much of the literature on the homogenization of second order level set PDE treats the hyperbolic scaling, where the aim is to obtain a first order motion in the limit.  For periodic media, we refer the reader to the papers by Lions and Souganidis \cite{lions souganidis}, Cardaliaguet, Lions, and Souganidis \cite{cardaliaguet lions souganidis}, Caffarelli and Monneau \cite{caffarelli monneau}, and Gao and Kim \cite{gao kim} and the references therein.  The random setting has been treated by Armstrong and Cardaliaguet \cite{geometric motions} and Feldman \cite{feldman mean curvature}.

Fewer results are available in the parabolic scaling.  Novaga and Valdinoci \cite{novaga valdinoci} proved a pinning result for mean curvature with periodic forcing in this scaling, showing that pinning is generic for forcing terms in the $L^{1}(\mathbb{T}^{d})$ topology and relating this phenomenon to a prescribed curvature problem.  Barles, Cesaroni, and Novaga \cite{barles cesaroni novaga} prove homogenization of a similar problem when the forcing is laminar and the initial datum is a graph over a suitable direction.  

A number of works have treated parabolically scaled geometric motions subject to small forcings in the same spirit as equation \eqref{E: quasilinear forced} below.  Cesaroni, Novaga, and Valdinoci \cite{cesaroni novaga valdinoci} proved homogenization of plane-like solutions of a forced mean curvature equation closely related to \eqref{E: quasilinear forced}, obtaining homogenized speeds that are discontinuous with respect to the direction.  That built on the work of Craciun and Bhattacharya \cite{craciun bhattacharya}, who had initially studied the homogenized speeds in question, and is related to a paper by Chen and Lou \cite{chen lou}, who found planar and V-shaped solutions of equations like \eqref{E: quasilinear} and \eqref{E: quasilinear forced} in dimension two and with almost periodic forcings.

During the preparation of this paper, the author learned that a number of the ideas presented here were anticipated by Craciun and Bhattacharya in their study of periodically forced mean curvature flow.  Compare the proofs of Theorem \ref{T: cell problem ergodic constant} and Proposition \ref{P: rational traveling} below to the formal analysis of \cite[Proposition 6]{craciun bhattacharya}.

This paper draws on a number of technical advances in the theory of viscosity solutions.  First of all, our approach draws a novel connection between the homogenization of second order level set PDE and that of uniformly elliptic equations in nondivergence form.  Throughout the paper, we study the asymptotics of \eqref{E: level set PDE} using a degenerate elliptic cell problem in the torus $\mathbb{T}^{d}$.  As shown in Section \ref{S: approximate correctors} below, this cell problem leads to the analysis of a family of uniformly elliptic equations with almost periodic coefficients, which falls under the purview of the papers by Caffarelli, Souganidis, and Wang \cite{caffarelli souganidis wang} and Caffarelli and Souganidis \cite{caffarelli souganidis}.  In the study of the discontinuity properties of the homogenized coefficients, in particular, we make extensive use of the obstacle problem formulation of \cite{caffarelli souganidis wang}. 

Secondly, where the discontinuity of the effective coefficients is concerned, our approach is inspired by recent developments in the study of oscillating boundary value problems.  A non-exhaustive list of references in this area includes the papers by Barles, Da Lio, and Souganidis \cite{barles da lio souganidis}, Barles and Mironescu \cite{barles mironescu}, G\'{e}rard-Varet and Masmoudi \cite{homogenization and boundary layers}, Choi and Kim \cite{choi kim}, Feldman \cite{feldman}, Feldman and Kim \cite{feldman kim}, Feldman and Zhang \cite{feldman zhang}, and Feldman, Kim, and Souganidis \cite{feldman kim souganidis}.  This paper adapts to the setting of geometric flows the approach of \cite{feldman kim}.

Finally, the paper builds on studies of level set PDE with discontinuous coefficients by Gurtin, Soner, and Souganidis \cite{gurtin soner souganidis}, Ohnuma and Sato \cite{ohnuma sato}, and Ishii \cite{ishii level set}.  The major difference between our approach and theirs is the effective equations obtained in higher dimensions are generically discontinuous at a countably infinite set of directions on the sphere.  Nonetheless, by exploiting the particular properties of the effective operators, we prove comparison by adapting the arguments of \cite{ishii level set}.

\subsection{Organization of the Paper}  In the next section, the main results of the paper are stated and the proofs are sketched.  Section \ref{S: approximate correctors} is devoted to the analysis of the degenerate cell problem used throughout the paper.  In Section \ref{S: irrational directions}, these correctors are used to show that the homogenized motion has the desired behavior whenever its normal vector is irrational, and we explain how to extend this to arbitrary normal directions.  Section \ref{S: discontinuity} treats the continuity properties of the homogenized operator, including the proof that discontinuity in rational directions is generic.  Section \ref{S: comparison} extends the comparison principle to a class of level set PDE with discontinuous coefficients that includes the effective equations obtained in Section \ref{S: irrational directions}.  Finally, Section \ref{S: linear response} studies the effective mobility of the quasi-linear problem from the point of view of linear response.

A number of technical results needed in the paper are provided in the Appendix \ref{A: technical lemmata}, while Apppendix \ref{A: geometric construction} specifically treats a geometric construction used in the proof of comparison.

\subsection{Notation} \label{S: notation}  The Euclidean inner product in $\mathbb{R}^{d}$ is denoted by $\langle \cdot, \cdot \rangle$; $\|\cdot\|$ is its associated norm.  $S^{d-1}$ is the unit sphere in $\mathbb{R}^{d}$, that is, those vectors $e$ with $\|e\| = 1$.  We write $\langle e \rangle$ for the real span of $e$ and $\langle e \rangle^{\perp}$ for its orthogonal complement.

Given $p \in \mathbb{R}^{d} \setminus \{0\}$, we write $\hat{p} = \|p\|^{-1} p$.

$\mathcal{S}_{d}$ is the space of real $d \times d$ symmetric matrices.  Given $e \in S^{d-1}$, $e \otimes e$ is the orthogonal projection onto $\langle e \rangle$.  Given such an $e$ and an $X \in \mathcal{S}_{d}$, we define the matrix $\tilde{X}_{e}$ by $\tilde{X}_{e} = (\text{Id} - e \otimes e) X (\text{Id} - e \otimes e)$.  

$\mathbb{T}^{d}$ is the $d$-dimensional torus, that is, the quotient space obtained from $\mathbb{R}^{d}$ by identifying points that differ by an element of the integers $\mathbb{Z}^{d}$.  $\mathbb{R} \mathbb{Z}^{d}$ denotes the real span of $\mathbb{Z}^{d}$.

Given an $e \in S^{d-1}$, the differential operator $D^{2}_{e}$ is defined via its action on smooth functions $\varphi$ by $D^{2}_{e}\varphi = (\text{Id} - e \otimes e) D^{2} \varphi (\text{Id} - e \otimes e)$.

$\mathcal{L}^{d}$ and $\mathcal{H}^{d-1}$ are the Lebesgue measure and $(d-1)$-dimensional Hausdorff measure in $\mathbb{T}^{d}$, respectively, the latter normalized to coincide with surface area.

\section{Main Results} \label{S: main results}

\subsection{Homogenization for $d = 2$} \label{S: dimension two results}  In dimension two, we consider operators $F : (\mathbb{R}^{2} \setminus \{0\}) \times \mathcal{S}_{2} \times \mathbb{T}^{2} \to \mathbb{R}$ satisfying the following assumptions:
	\begin{itemize}
		\item[(i)] (Geometric)  If $\nu, \mu \in \mathbb{R}$ and $\nu > 0$, then
			\begin{equation} \label{A: geometric}
				F(\nu p, \nu X + \mu p \otimes p, y) = \nu F(p, X, y) \quad \text{for} \, \, (p,X,y) \in \mathbb{R}^{2} \times \mathcal{S}_{2} \times \mathbb{T}^{2}.
			\end{equation}
		\item[(ii)] (Stationary planes) For each $e \in S^{1}$, 
			\begin{equation} \label{E: stationary planes}
				F(e,0,y) = 0 \quad \text{for} \, \, y \in \mathbb{T}^{2}.
			\end{equation}
		\item[(iii)](Uniform degenerate ellipticity)  There are constants $\lambda, \Lambda > 0$ such that if $(e,X,y) \in S^{1} \times \mathcal{S}_{2} \times \mathbb{T}^{2}$, $Y \in \mathcal{S}_{2}$ satisfies $Y \geq 0$, and $\tilde{Y}_{e}$ is the matrix $\tilde{Y}_{e} = (\text{Id} - e \otimes e) Y (\text{Id} - e \otimes e)$, then 
			\begin{equation} \label{A: strong degenerate ellipticity}
				\lambda \|\tilde{Y}_{e}\| \leq F(e,X + Y,y) - F(e,X,y) \leq \Lambda \|\tilde{Y}_{e}\|.
			\end{equation}
		\item[(iv)](Regularity) $F$ is continuous in $(\mathbb{R}^{2} \setminus \{0\}) \times \mathcal{S}_{2} \times \mathbb{T}^{2}$.
		\item[(v)](Comparison) $F$ satisfies a technical assumption ensuring \eqref{E: level set PDE} is well-posed.  (This assumption is stated precisely in the appendix.)
	\end{itemize}


	\begin{theorem} \label{T: level set PDE} If $F$ satisfies assumptions (i)-(v) above, then there is a continuous function $\overline{F} : (\mathbb{R}^{2} \setminus \{0\}) \times \mathcal{S}_{2} \to \mathbb{R}$ satisfying the same assumptions such that if $u_{0} \in UC(\mathbb{R}^{2})$, $(u^{\epsilon})_{\epsilon >0}$ are the unique viscosity solutions of \eqref{E: level set PDE}, and $\bar{u} : \mathbb{R}^{d} \times [0,\infty) \to \mathbb{R}$ is the solution of the effective equation
		\begin{equation} \label{E: effective operator in dimension 2}
			\left\{ \begin{array}{r l}
				\bar{u}_{t} - \overline{F}\left(D\bar{u},D^{2} \bar{u} \right) = 0 & \text{in} \, \, \mathbb{R}^{2} \times (0,\infty), \\
				\bar{u} = u_{0} & \text{on} \, \, \mathbb{R}^{2} \times \{0\},
			\end{array} \right.
		 \end{equation}
	then $u^{\epsilon} \to \bar{u}$ locally uniformly in $\mathbb{R}^{2} \times [0,\infty)$.  \end{theorem}  
	
	Well-posedness of \eqref{E: level set PDE} is reviewed in Appendix \ref{A: wellposed} below.
	
\subsection{Homogenization of Quasi-linear Geometric Operators in $d \geq 2$}  We now state the results concerning the homogenization of \eqref{E: quasilinear}.  Here and henceforth we fix $\lambda, \Lambda > 0$ and define the space $\mathcal{S}_{d}(\lambda,\Lambda)$ by 
	\begin{equation*}
		\mathcal{S}_{d}(\lambda,\Lambda) = \left\{a_{0} \in \mathcal{S}_{d} \, \mid \, \lambda \text{Id} < a_{0} < \Lambda \text{Id} \right\}.
	\end{equation*}
The matrix field $a : \mathbb{T}^{d} \times S^{d - 1} \to \mathcal{S}_{d}$ is assumed to satisfy the following:
	\begin{align}
		a &: \mathbb{T}^{d} \times S^{d-1} \to \mathcal{S}_{d}(\lambda,\Lambda) \quad \text{is continuous}, \label{A: a continuous}\\
		&\|a(\cdot,e)\|_{C^{2,\alpha}(\mathbb{T}^{d})} < \infty \quad \text{if} \, \, e \in S^{d-1},  \label{A: a C2}\\
		\sup &\left\{ \frac{\|a(y,e) - a(y',e)\|}{\|y - y'\|} \, \mid \, (e,y), (e,y') \in S^{d-1} \times \mathbb{T}^{d}, \, \, y' \neq y \right\} < \infty. \label{A: a lip}
	\end{align}
Concerning the mobility $m : \mathbb{T}^{d} \times S^{d-1} \to (0,\infty)$, the assumptions are listed next:
	\begin{align}
		m \, \, &\text{is continuous}, \label{A: m continuous} \\
		\|m(\cdot,e)&\|_{C^{2,\alpha}(\mathbb{T}^{d})} < \infty \quad \text{if} \, \, e \in S^{d-1}, \label{A: m smooth} \\
		\sup \bigg \{ &\frac{|m(y,e) - m(y',e)|}{\|y - y'\|} \, \mid \, (y,e), (y',e) \in \mathbb{T}^{d} \times S^{d-1}, \, \, y \neq y' \bigg \} < \infty. \label{A: m lipschitz}
	\end{align}
	
The main result in this setting is:

	\begin{theorem} \label{T: quasilinear}  If $a$ satisfies assumptions \eqref{A: a continuous}, \eqref{A: a C2}, and \eqref{A: a lip} and $m$ satisfies \eqref{A: m continuous}, \eqref{A: m smooth}, and \eqref{A: m lipschitz}, then there are effective coefficients $\overline{m} : S^{d-1} \setminus \mathbb{R} \mathbb{Z}^{d} \to (0,\infty)$ and $\bar{a} : S^{d-1} \setminus \mathbb{R} \mathbb{Z}^{d} \to \mathcal{S}_{d}$ such that if  $\overline{F} : (\mathbb{R}^{d} \setminus \mathbb{R} \mathbb{Z}^{d}) \times \mathcal{S}_{d} \to \mathbb{R}$ is given by 
		\begin{equation} \label{E: homogenized coefficient linear}
			\overline{F}(p,X) = \overline{m}(\hat{p})^{-1} \text{tr} \left( \bar{a}(\hat{p}) \tilde{X}_{\hat{p}} \right),
		\end{equation}
	$u_{0} \in UC(\mathbb{R}^{d})$, and $(u^{\epsilon})_{\epsilon > 0}$ are the solutions of \eqref{E: quasilinear}, then:
		\begin{itemize}
			\item[(i)] There is a unique viscosity solution $\bar{u} : \mathbb{R}^{d} \times [0,\infty) \to \mathbb{R}$ of the equation
				\begin{equation} \label{E: effective equation}
					\left\{ \begin{array}{r l}
						\bar{u}_{t} - \overline{F}(D\bar{u}, D^{2}\bar{u}) = 0 & \text{in} \, \, \mathbb{R}^{d} \times (0,\infty), \\
						\bar{u} = u_{0} & \text{on} \, \, \mathbb{R}^{d} \times \{0\}.
					\end{array} \right.
				\end{equation}
			\item[(ii)] $u^{\epsilon} \to \bar{u}$ locally uniformly as $\epsilon \to 0^{+}$.
		\end{itemize}
	\end{theorem}  
	
	Well-posedness of \eqref{E: quasilinear} is reviewed in Appendix \ref{A: wellposed}.

A few remarks are in order:
	
	\begin{remark} (i) Even though we have imposed considerable regularity restrictions on the matrix field $a$, nonetheless we show below that $\overline{F}$ generically fails to admit a continuous extension to $(\mathbb{R}^{d} \setminus \{0\}) \times \mathcal{S}_{d}$.  We expect this is true more generally (i.e.\ fully non-linear operators) whenever $d \geq 3$.

(ii) In certain special cases, $\overline{F}$ \emph{does} extend to a continuous function.  Examples of this are discussed in Section \ref{S: anisotropic curvature flows} below.

(iii) The discontinuity of $\overline{F}$  means that the usual comparison principle does not apply to \eqref{E: effective equation}.  This is rectified in Section \ref{S: comparison}.
	\end{remark}

\begin{remark}  The smoothness assumptions on $m$ and $a$ are certainly restrictive.  As we will see below, these assumptions are natural in light of the degeneracy of the cell problems used in the asymptotic analysis.  At various places in the arguments, it is not clear how to proceed unless the corrector is $C^{2}$, and requiring the coefficients to be $C^{2}$ is a natural way to achieve this.  

Nonetheless, when $d =2$, we are able to obtain $C^{2}$ approximate correctors by regularizing the coefficients and exploiting the special structure of the cell problem in this dimension.  It is for this reason that Theorem \ref{T: level set PDE} is more general.

We expect that, using the same techniques, Theorem \ref{T: quasilinear} could be extended to convex or concave operators that are twice continuously differentiable in the $(X,y)$ variables, but this is not so interesting since the main examples we have in mind are one-homogeneous in the Hessian and, therefore, not differentiable at the origin.

\end{remark}  

\subsection{(Dis)continuity of the Homogenized Coefficients}  The effective operator $\overline{F}$ in Theorem \ref{T: quasilinear} is determined by the invariant probability measures associated with a certain family of diffusion processes on the torus.  Toward that end, we define the sets of invariant measures $\{\mathscr{I}^{a}_{e}\}_{e \in S^{d-1}} \subseteq \mathscr{P}(\mathbb{T}^{d})$ so that $\mu \in \mathscr{I}_{e}^{a}$ if and only if, for each $\varphi \in C^{\infty}(\mathbb{T}^{d})$, writing $D^{2}_{e} \varphi = (\text{Id} - e \otimes e) D^{2}\varphi (\text{Id} - e \otimes e)$, we have
	\begin{equation*}
		\int_{\mathbb{T}^{d}} \text{tr} \left( a(y,e) D^{2}_{e} \varphi(y) \right) \mu(dy) = 0.
	\end{equation*}
For a given $e \in S^{d-1}$, these are precisely the invariant probability measures of the diffusion process $X^{e}$ on $\mathbb{T}^{d}$ determined by the SDE
	\begin{equation} \label{E: SDE}
		dX^{e}_{t} = (\text{Id} - e \otimes e) \sqrt{a(X^{e}_{t},e)} (\text{Id} - e \otimes e) d B_{t}.
	\end{equation}

The basic structure of the sets $\{\mathscr{I}_{e}^{a}\}_{e \in S^{d-1}}$ and their relationship to the homogenized coefficients $\bar{a}$ and $\overline{m}$ is summed up in the following theorem.  In the rational case, the foliation of $\mathbb{T}^{d}$ into sub-tori normal to $e$ arises naturally.  That is, we will use the decomposition of $\mathbb{T}^{d}$ as
	\begin{equation*}
		\mathbb{T}^{d} = \bigcup_{r \in [0,r_{e})} \mathbb{T}^{d - 1}_{e}(r),
	\end{equation*}
where the sub-tori $\{\mathbb{T}^{d - 1}_{e}(r)\}_{r \in [0,r_{e})}$ and the period $r_{e} > 0$ are defined by 
	\begin{align}
		\mathbb{T}^{d-1}_{e}(r) &= \left\{y \in \mathbb{T}^{d} \, \mid \, \langle y,e \rangle = r + \langle k,e \rangle \, \, \text{for some} \, \, k \in \mathbb{Z}^{d} \right\}, \label{E: subtori} \\
		r_{e} &= \min \left\{ \langle k,e \rangle \, \mid \, k \in \mathbb{Z}^{d} \right\} \cap (0,\infty). \nonumber
	\end{align}
When $e$ is rational, the trajectories of the process $X^{e}$ of \eqref{E: SDE} remain confined to $\mathbb{T}^{d-1}_{e}(\langle X^{e}_{0},e \rangle)$ and this leads to a corresponding decomposition of $\mathscr{I}^{a}_{e}$.  

Now we are prepared to state the main result concerning the structure of the invariant measures and their relation to the effective coefficients $\bar{a}$ and $\overline{m}$:

	\begin{theorem} \label{T: invariant measures} Under the assumptions of Theorem \ref{T: quasilinear}, we have:
	\begin{itemize}
		\item[(i)] If $e \notin \mathbb{R} \mathbb{Z}^{d}$, then there is a unique probability measure $\bar{\mu}_{e}$ such that $\mathscr{I}^{a}_{e} = \{\bar{\mu}_{e}\}$.  Furthermore, $\bar{\mu}_{e} \ll \mathcal{L}^{d}$.  
	
	\item[(ii)] If $e \in \mathbb{R} \mathbb{Z}^{d}$, then there is an $r_{e}$-periodic function $\mu_{e} : \mathbb{R} \to \mathscr{P}(\mathbb{T}^{d})$, $\mu_{e} : s \mapsto \mu_{e}^{s}$, such that $\mathscr{I}^{a}_{e}$ equals the closed convex hull of $\{\mu_{e}^{s} \, \mid \, s \in \mathbb{R}\}$.  For each $s \in [0,r_{e})$, we have $\mu_{e}^{s} \ll \mathcal{H}^{d-1} \restriction_{\mathbb{T}^{d-1}_{e}(r)}$.  
	\item[(iii)] For each $e \in S^{d-1} \setminus \mathbb{R} \mathbb{Z}^{d}$, $\bar{a}$ and $\overline{m}$ are given by 
		\begin{equation} \label{E: effective a and m}
			\overline{a}(e) = \int_{\mathbb{T}^{d}} a(y,e) \, \bar{\mu}_{e}(dy), \quad \overline{m}(e) = \int_{\mathbb{T}^{d}} m(y,e) \, \bar{\mu}_{e}(dy).
		\end{equation}
	\end{itemize} 
	\end{theorem}  
	
In order to prove Theorem \ref{T: quasilinear}, we will study the continuity properties of the set-valued maps $e \mapsto \mathscr{I}^{a}_{e}$.  When $e \notin \mathbb{R} \mathbb{Z}^{d}$, continuity at $e$ is implied directly by uniqueness and compactness.  The continuity question is significantly more complicated when $e \in \mathbb{R} \mathbb{Z}^{d}$.  The next results address that issue.

	First, we need a digression on analysis in $S^{d-1}$ because it turns out that the limiting behavior of $e \mapsto \mathscr{I}^{a}_{e}$ at a rational direction depends on the direction of approach.  To make that precise, notice that if $(e_{n})_{n \in \mathbb{N}} \subseteq S^{d-1}$ converges to $e$ as $n \to \infty$, then there is necessarily a a subsequence $(n_{j})_{j \in \mathbb{N}} \subseteq \mathbb{N}$ and an $\eta \in S^{d-1} \cap \langle e \rangle^{\perp}$ such that 
	\begin{equation*}
		-\eta = \lim_{j \to \infty} \frac{e_{n_{j}} - e}{\|e_{n_{j}} - e\|}.
	\end{equation*}
	Geometrically, that means $e_{n_{j}}$ approximately approaches $e$ along the great circle parametrized by $\theta \mapsto \cos(\theta) e + \sin(\theta) \eta$, or, put simply, $e_{n_{j}} \to e$ along the $\eta$ direction.  We will see that it is necessary to take account of the direction $\eta$ when studying the continuity properties of $e \mapsto \mathscr{I}^{a}_{e}$.
	
	Finally, in the statement, it is convenient to use the metrizability of $\mathscr{P}(\mathbb{T}^{d})$.  Toward that end, we fix here and henceforth a metric $D : \mathscr{P}(\mathbb{T}^{d}) \times \mathscr{P}(\mathbb{T}^{d}) \to [0,\infty)$ inducing the weak-$*$ topology (e.g.\ Wasserstein distance).

	\begin{theorem}  \label{T: limiting measures} Under the assumptions of Theorem \ref{T: quasilinear}, for each $e \in \mathbb{R} \mathbb{Z}^{d}$, if we define the function $\tilde{\mu}_{e} : S^{d-1} \cap \langle e \rangle^{\perp} \to \mathscr{P}(\mathbb{T}^{d})$, $\tilde{\mu}_{e} : \eta \mapsto \tilde{\mu}_{e}^{\eta}$, by 
		\begin{align}
			\tilde{\mu}_{e}^{\eta} &= \left( \int_{0}^{r_{e}} \langle a^{\perp}_{e}(se) \eta, \eta \rangle^{-1} \, ds \right)^{-1} \int_{0}^{r_{e}} \langle a^{\perp}_{e}(se) \eta, \eta \rangle^{-1} \mu^{s}_{e} \, ds, \label{E: limiting measure} \\
			a_{e}^{\perp}(y) &= \int_{\mathbb{T}^{d}} a(y',e) \mu_{e}^{\langle y, e \rangle}(dy'), \label{E: oscillating tensor}
		\end{align}
	then this function describes the continuity properties of $e' \mapsto \mathscr{I}^{a}_{e'}$ at $e$ in the sense that
		\begin{equation*}
			\lim_{\delta \to 0^{+}} \sup \left\{ D(\mu,\tilde{\mu}_{e}^{\eta}) \, \mid \, \mu \in \mathscr{I}^{a}_{e'}, \, \, \|e' - e\| + \left\| \frac{e' - e}{\|e' - e\|} + \eta \right\| < \delta \right\} = 0.
		\end{equation*}
	\end{theorem}  
	
The corollary that follows shows that discontinuity is generic among coefficients $a$ satisfying the assumptions of Theorem \ref{T: quasilinear} if $d \geq 3$.

	\begin{corollary} \label{C: generic}  Assume $d \geq 3$.  There is a residual set $\mathscr{C}_{d} \subseteq C^{2,\alpha}(\mathbb{T}^{d}; \mathcal{S}_{d}(\lambda,\Lambda))$ such that if $a$ is independent of the $e$ variable and $a \in \mathscr{C}_{d}$, then the following statements hold:
		\begin{itemize}
			\item[(a)] If $\tilde{\mu}$ is the function defined in Theorem \ref{T: limiting measures}, then
			\begin{equation*}
				\forall e \in S^{d-1} \cap \mathbb{R} \mathbb{Z}^{d} \quad \#\{\tilde{\mu}_{e}^{\eta} \, \mid \, \eta \in S^{d-1} \cap \langle e \rangle^{\perp} \} = \infty.
			\end{equation*} 
			\item[(b)] The effective coefficient $\bar{a} : S^{d-1} \setminus \mathbb{R} \mathbb{Z}^{d} \to \mathcal{S}_{d}(\lambda,\Lambda)$ defined in Theorem \ref{T: invariant measures} has infinitely many distinct directional limits at each $e \in \mathbb{R} \mathbb{Z}^{d}$.
		\end{itemize}
	In particular, by taking $a \in \mathscr{C}_{d}$ and $m \equiv 1$ in Theorem \ref{T: quasilinear}, we find that $\overline{F}^{*}(e,\cdot) \neq \overline{F}_{*}(e,\cdot)$ for each $e \in S^{d-1} \cap \mathbb{R} \mathbb{Z}^{d}$.  
	\end{corollary}  
	

\subsection{Effective Mobility and Linear Response}  Finally, we discuss the relationship between the effective mobility $\overline{m}$ in \eqref{E: quasilinear} and linear response.  Specifically, given $e \in S^{d-1}$ and defining $A$ as in \eqref{E: structure condition}, we are interested in the forced motion
	\begin{equation} \label{E: quasilinear forced}
		\left\{ \begin{array}{r l}
			m(\epsilon^{-1} x, \widehat{Du_{e}^{\epsilon}}) u^{\epsilon}_{e,t} - \text{tr} \left( A(\epsilon^{-1} x, \widehat{Du_{e}^{\epsilon}}) D^{2} u_{e}^{\epsilon} \right) - \alpha \|Du_{e}^{\epsilon}\| = 0 & \text{in} \, \, \mathbb{R}^{d} \times (0,\infty), \\
			u_{e}^{\epsilon}(x,0) = \langle x, e \rangle & \text{if} \, \, x \in \mathbb{R}^{d}.
		\end{array} \right.
	\end{equation}

To start with, we prove homogenization of \eqref{E: quasilinear forced}:

	\begin{theorem} \label{T: homogenization forced planes} Fix $d \geq 2$.  If $a$ and $m$ satisfy the assumptions of Theorem \ref{T: quasilinear}, then there is an $\overline{m}_{\text{pl}} : S^{d-1} \to (0,\infty)$ such that if $e \in S^{d-1}$ and $(u_{e}^{\epsilon})_{\epsilon > 0}$ are the solutions of \eqref{E: quasilinear forced}, then 
		\begin{equation*}
			\lim_{\epsilon \to 0^{+}} u_{e}^{\epsilon}(x,t) = \langle x,e \rangle + \alpha \overline{m}_{\text{pl}}(e)^{-1} t \quad \text{locally uniformly in} \, \, \mathbb{R}^{d} \times [0,\infty).
		\end{equation*}
	\end{theorem}  
	
When $d = 2$, $\overline{m}_{\text{pl}}^{\, -1}$ is precisely the derivative of the front speeds obtained by Caffarelli and Monneau \cite{caffarelli monneau}.  Recall that if $e \in S^{1}$ and $\alpha \in \mathbb{R}$ and if $u^{e,\alpha}$ is the solution of the forced problem
	\begin{equation} \label{E: hyperbolic problem}
		\left\{ \begin{array}{r l}
			m(y,\widehat{Du^{e,\alpha}}) u^{e,\alpha}_{t}- \text{tr} \left(A(y,\widehat{Du^{e,\alpha}}) D^{2}u^{e,\alpha}\right) - \alpha \|Du^{e,\alpha}\| = 0 & \text{in} \, \, \mathbb{R}^{2} \times (0,\infty), \\
			u^{e,\alpha}(x) = \langle x,e \rangle & \text{if} \, \, x \in \mathbb{R}^{2},
		\end{array} \right.
	\end{equation}
then, by \cite{caffarelli monneau}, there is a $\lambda_{e}(\alpha) \in \mathbb{R}$ such that 
	\begin{equation*}
		\lim_{R \to \infty} R^{-1} u^{e,\alpha}(Rx,Rt) = \langle x,e \rangle + \lambda_{e}(\alpha) t \quad \text{locally uniformly in} \, \, \mathbb{R}^{2} \times [0,\infty).
	\end{equation*}
The derivative of this function is $\overline{m}_{\text{pl}}^{\, -1}$:

	\begin{corollary} \label{C: derivatives} If $d = 2$ and $e \in S^{1}$, then $\overline{m}_{\text{pl}}(e)^{-1} = \lim_{\alpha \to 0} \alpha^{-1} \lambda_{e}(\alpha)$.  \end{corollary}  

Notice that, up to a parabolic rescaling, the solution $u^{\epsilon}_{e}$ of \eqref{E: quasilinear forced} equals $u^{e,\alpha \epsilon}$.  As discussed by Spohn \cite{spohn paper} and Bellettini, Butt\`{a}, and Presutti \cite{bellettini butta presutti}, this is precisely the scaling where the forcing and the curvature are evenly matched, and the limiting behavior of $u^{\epsilon}$ should capture the linear response of \eqref{E: quasilinear} to small perturbations.  In light of previous works, it is natural to expect that $\overline{m}_{\text{pl}} = \overline{m}$.  

It turns out these functions only coincide in irrational directions, even when $d = 2$.  To start with, we obtain an expression for $\overline{m}_{\text{pl}}$:

	\begin{theorem} \label{T: linear response formula} Assume that $d \geq 2$ and $a$ and $m$ satisfy the assumptions of Theorem \ref{T: quasilinear}.  The following statements hold:
	\begin{itemize}
		\item[(i)]  If $e \in S^{d-1} \setminus \mathbb{R} \mathbb{Z}^{d}$, then $\overline{m}_{\text{pl}}(e) = \overline{m}(e)$.  
		\item[(ii)]  If $e \in S^{d-1} \cap \mathbb{R} \mathbb{Z}^{d}$ and $\mu_{e}$ is the function defined in Theorem \ref{T: invariant measures}, then 
		\begin{equation} \label{E: linear response}
			\overline{m}_{\text{pl}}(e) = r_{e}^{-1} \int_{0}^{r_{e}} \int_{\mathbb{T}^{d}} m(y,e) \, \mu^{s}_{e}(dy) \, ds.
		\end{equation}
	\end{itemize}
	\end{theorem}

In view of the formula for $\overline{m}_{\text{pl}}$ in rational directions, a proof similar to that of Corollary \ref{C: generic} yields the following:

	\begin{corollary} \label{C: bad mobility} If $d \geq 2$, then there is a residual set $\mathcal{A}_{d} \subseteq C^{2,\alpha}(\mathbb{T}^{d}; \mathcal{S}_{d}(\lambda,\Lambda)) \times C^{2,\alpha}(\mathbb{T}^{d} ; (0,\infty))$ such that if $m$ and $a$ are independent of $e$ and $(a,m) \in \mathcal{A}_{d}$, then, for each $e \in S^{d-1} \cap \mathbb{R} \mathbb{Z}^{d}$, the following hold:
		\begin{itemize}
			\item[(i)] There is an $\eta \in S^{d-1} \cap \langle e \rangle^{\perp}$ such that $\overline{m}_{\text{pl}}(e) \neq \tilde{m}_{e}^{\eta}$.  
			\item[(ii)] If $d \geq 3$, then $\overline{m} : S^{d-1} \setminus \mathbb{R} \mathbb{Z}^{d} \to (0,\infty)$ does not have a limit at $e$.
		\end{itemize}
	In particular, if $(a,m) \in \mathcal{A}_{d}$, then $\overline{m}_{\text{pl}}$ is not a continuous function.
	\end{corollary}    
	
\subsection{Strategy of Proof}  Let us briefly review the main elements of the proof of Theorems \ref{T: level set PDE}, \ref{T: quasilinear}, and \ref{T: homogenization forced planes}.  Inspired by the approach of Barles and Souganidis \cite{barles souganidis} and Barles, Cesaroni, and Novaga \cite{barles cesaroni novaga}, we employ the ansatz
	\begin{equation*}
		u^{\epsilon}(x,t) = \bar{u}(x,t) + \epsilon^{2} V(\epsilon^{-1}x) + \dots
	\end{equation*}
which leads to the following cell problems for correctors $V$ depending on the normal vector $e \in S^{d-1}$ and second fundamental form $X \in \mathcal{S}_{d}$:
	\begin{equation} \label{E: cell problem}
		- F(e, X + D^{2} V, y) = - \overline{F}(e,X) \quad \text{if} \, \, \mathbb{T}^{d}.
	\end{equation}
In view of \eqref{A: strong degenerate ellipticity}, this is a degenerate elliptic PDE.  The analysis below shows that it is ill-posed when $e \in \mathbb{R} \mathbb{Z}^{d}$, which presents a significant obstacle in what follows.

Nonetheless, in \emph{irrational directions}, that is, if $e \notin \mathbb{R} \mathbb{Z}^{d}$, it turns out that if we let $(V^{\delta})_{\delta > 0}$ be the solutions of the penalized cell problem 
	\begin{equation} \label{E: penalized cell problem}
		\delta V^{\delta} - F(e,X + D^{2} V^{\delta},y) = 0 \quad \text{in} \, \, \mathbb{T}^{d},
	\end{equation}
then there is an $\overline{F}(e,X) \in \mathbb{R}$ such that
	\begin{equation} \label{E: ergodic behavior}
		\overline{F}(e,X) = \lim_{\delta \to 0^{+}} \delta V^{\delta} \quad \text{uniformly in} \, \, \mathbb{T}^{d}.
	\end{equation}
As we show below, this can be seen by rewriting \eqref{E: penalized cell problem} as a one-parameter family of uniformly elliptic $(d-  1)$-dimensional problems and applying homogenization results for almost periodic, uniformly elliptic operators in nondivergence form.

Once \eqref{E: ergodic behavior} is proved, we construct approximate correctors of \eqref{E: penalized cell problem} following a well-worn approach in homogenization theory.  The approximate correctors are then used to show that the solutions $(u^{\epsilon})_{\epsilon > 0}$ behave appropriately wherever the limiting level set has an irrational normal.

That leaves the question of whether or not a function that solves a level set PDE at points where its normal vector is irrational is actually a genuine solution.  In this case, as long as we are studying the unforced problems \eqref{E: level set PDE} or \eqref{E: quasilinear}, the answer is yes and follows from the results of the companion paper \cite{allen cahn mobility}, as shown below. 

Two problems remain.  First, it is not at all clear that $\overline{F}$ extends to a continuous function.  Toward that end, we adapt ideas from the study of oscillating boundary value problems to characterize all possible limits of $\overline{F}$ at rational directions.  When $d = 2$, it \emph{is} continuous and, thus, the usual comparison principle for level set PDE applies and implies homogenization.  In the setting of Theorem \ref{T: quasilinear}, however, the operator is generically discontinuous in higher dimensions, necessitating an extension of the comparison principle to more exotic, discontinuous operators.

The remaining issue is the homogenization of the forced motion \eqref{E: quasilinear forced}.  Here additional ideas are needed to treat rational directions.  If $e \in \mathbb{R} \mathbb{Z}^{d}$ and we let $(V^{\delta})_{\delta > 0}$ be the solutions of 
	\begin{equation} \label{E: bad cell problem}
		m(y,e) + \delta V^{\delta} - \text{tr} \left( A(y,e) D^{2} V^{\delta} \right) = 0 \quad \text{in} \, \, \mathbb{T}^{d},
	\end{equation}
then there is an $m^{\perp}_{e} \in C(\mathbb{T}^{d})$ varying only in the $e$ direction such that
	\begin{equation*}
		m^{\perp}_{e} = \lim_{\delta \to 0^{+}} (-\delta V^{\delta}) \quad \text{uniformly in} \, \, \mathbb{T}^{d}.
	\end{equation*}
In general, $m^{\perp}_{e}$ is not a constant function.  The reason is that, at the level of \eqref{E: quasilinear forced}, sending $\delta \to 0^{+}$ in \eqref{E: bad cell problem} only averages the fluctuations of the front relative to a moving reference plane.  As we show below, the height of this plane oscillates around its mean and needs to be corrected through another, one-dimensional cell problem.

\section{Examples}  In this section, we discuss a few examples of interest that fit into the assumptions of either Theorem \ref{T: level set PDE} or Theorem \ref{T: quasilinear}.  

\subsection{Anisotropic Curvature Flows with Periodic Mobilities} \label{S: anisotropic curvature flows}  Let $\varphi : \mathbb{R}^{d} \to [0,\infty)$ be a $C^{2}$ Finsler norm, that is, a convex, positively one-homogeneous function that is positive and $C^{2}$ in $\mathbb{R}^{d} \setminus \{0\}$.  Given a mobility coefficient $m$ as above, the equation	
	\begin{equation} \label{E: anisotropic curvature flow}
		m(y,\widehat{Du}) u_{t} - \text{tr} \left(D^{2}\varphi(\widehat{Du}) D^{2} u \right) = 0 \quad \text{in} \, \, \mathbb{R}^{d} \times (0,\infty)
	\end{equation}
is referred to as an anisotropic curvature flow.  Formally, this is the gradient flow of the anisotropic perimeter determined by $\varphi$ with respect to the $L^{2}$-Riemannian metric determined by $m$ (cf.\ \cite{taylor cahn}).  The geometric flows associated with these equations are of interest in materials science.  

In this setting, our results show that, at large scales, the mobility gets averaged:
	
	\begin{prop} \label{P: anisotropic curvature} Suppose that $\varphi$ is a $C^{2}$ Finsler norm that is uniformly convex in the following sense:
		\begin{equation*}
			\lambda (\text{Id} - e \otimes e) \leq D^{2} \varphi(e) \leq \Lambda (\text{Id} - e \otimes e) \quad \text{if} \, \, e \in S^{d-1}.
		\end{equation*} 
	If $m$ satisfies the assumptions of Theorem \ref{T: quasilinear} and $A(y,e) = D^{2}\varphi(e)$ for all $e \in S^{d-1}$, then the solutions $(u^{\epsilon})_{\epsilon > 0}$ of \eqref{E: quasilinear} converge locally uniformly to the solution $\bar{u}$ of the averaged equation
		\begin{equation*}
			\left\{ \begin{array}{r l}
				\overline{m}(\widehat{D\bar{u}}) \bar{u}_{t} - \text{tr} \left( D^{2}\varphi(\widehat{D\bar{u}}) D^{2} \bar{u}\right) = 0 & \text{in} \, \, \mathbb{R}^{d} \times (0,\infty), \\
				\bar{u} = u_{0} & \text{on} \, \, \mathbb{R}^{d} \times \{0\}.
			\end{array} \right.
		\end{equation*}
	Here $\overline{m}(e) = \int_{\mathbb{T}^{d}} m(y,e) \, dy$.  Furthermore, in this case, $\overline{m}_{\text{pl}} = \overline{m}$ in $S^{d-1}$.  \end{prop}  
	
A few remarks are in order:
	
\begin{remark} It is interesting to note that the effective mobility $\overline{m}_{\text{pl}}$ is continuous in Proposition \ref{P: anisotropic curvature}.  By contrast, given a positive function $c \in C^{0,1}(\mathbb{T}^{d})$, the structurally similar forced problem
	\begin{equation*}
		\left\{ \begin{array}{r l}
			u^{\epsilon}_{t} - \text{tr} \left( D^{2} \varphi(\widehat{Du^{\epsilon}}) D^{2} u^{\epsilon} \right) - c(\epsilon^{-1} x) \|Du^{\epsilon}\| = 0 & \text{in} \, \, \mathbb{R}^{d} \times (0,\infty), \\
			u^{\epsilon}(x,0) = \langle x,e \rangle & \text{on} \, \, \mathbb{R}^{d} \times (0,\infty),
		\end{array} \right.
	\end{equation*}
also homogenizes with $u^{\epsilon}(x,t) \to \langle x,e \rangle + \bar{c}(e) t$ in the limit.  (This was proved when $d = 2$ and $\varphi = \|\cdot\|$ in \cite{cesaroni novaga valdinoci} and follows by arguing as in Section \ref{S: linear response} below in general.)  However, as in \cite{cesaroni novaga valdinoci}, one finds that $\bar{c}$ is generically discontinuous.  Put slightly differently, pathologies arise if we replace constant forcing by periodic forcing in \eqref{E: quasilinear forced}.
\end{remark}

\begin{remark} Let us observe that by exploiting the regularity theory of the Laplacian and using approximation arguments as in \cite{allen cahn mobility}, it is not hard to show that Proposition \ref{P: anisotropic curvature} still holds if \eqref{A: m smooth} is dropped. \end{remark}

\begin{remark}  Of course, it would also be interesting to analyze the case when $\varphi$ depends on $y$ as well as $e$.  In that case, to get an anisotropic curvature flow, \eqref{E: anisotropic curvature flow} should include another term involving $D\varphi$ and no longer fits into the framework of \eqref{E: quasilinear}.  The homogenization of such equations, which are the natural divergence form analogue of \eqref{E: quasilinear}, is largely open.  However, see \cite{barles cesaroni novaga} for results that apply to the case where $\varphi(y,e) = 1 + \langle \psi(y),e \rangle$ for some vector field $\psi$ with $\|\psi\|_{L^{\infty}(\mathbb{T}^{d})} < 1$.  \end{remark}

\begin{remark}  In general, the operator $\overline{F}$ obtained in Theorem \ref{T: quasilinear} extends continuously to $(\mathbb{R}^{d} \setminus \{0\}) \times \mathcal{S}_{d}$ whenever $a$ is independent of the spatial variable, as in the last proposition.  Another (actually equivalent) case is when $a(y,e) = \tilde{a}(y,e) \text{Id}$ for some positive function $\tilde{a} \in C^{0,1}(\mathbb{T}^{d})$.  In this case, $\bar{a}(e)$ equals the harmonic mean of $\tilde{a}(\cdot,e)$, consistent with \cite[Theorem C]{chen lou}.  \end{remark}  

\subsection{Fully Nonlinear Operators in Dimension Two}  Notice that if $F$ satisfies assumptions (i)-(iv) of Theorem \ref{T: level set PDE} and also is Lipschitz continuous in the following sense
	\begin{equation} \label{E: standard lipschitz}
		|F(p,X,y) - F(p,X,y')| \leq C \|\tilde{X}_{e}\| \|y - y'\| \quad \text{if} \, \, (p,X) \in (\mathbb{R}^{d} \setminus \{0\}) \times \mathcal{S}_{d}, \, \, y, y' \in \mathbb{T}^{d},
	\end{equation}
then we can argue as in \cite{ishii lions} to see that $F$ also satisfies (v).  

A natural class of fully nonlinear operators satisfying these assumptions take the following form:
	\begin{equation*}
		F(p,X,y) = \sup_{\alpha \in \mathcal{A}} \inf_{\beta \in \mathscr{B}} \text{tr} \left( \left(\text{Id} - \hat{p} \otimes \hat{p} \right) a_{\alpha,\beta}(y,\hat{p}) \left(\text{Id} - e \otimes e\right) X \right).	
	\end{equation*}
Here we assume that the matrix fields $\{a_{\alpha,\beta} \, \mid \, \alpha \in \mathcal{A}, \, \, \beta \in \mathscr{B} \}$ map $\mathbb{T}^{d} \times S^{d-1}$ continuously into $\mathcal{S}_{d}(\lambda,\Lambda)$ and that there is an $L > 0$ such that, for each $(\alpha,\beta) \in \mathcal{A} \times \mathscr{B}$, we have
	\begin{equation*}
		\sup \left\{ \frac{\|a_{\alpha,\beta}(y,e) - a_{\alpha,\beta}(y',e)\|}{\|y - y'\|} \, \mid \, (y,e), (y',e) \in \mathbb{T}^{d} \times S^{d-1}, \, \, y \neq y' \right\} \leq L.
	\end{equation*}

When $\mathcal{A}$ is a singleton so that $F$ is concave, such operators describe stochastic target problems, which are of interest in financial mathematics (see \cite{soner touzi}, \cite{cardaliaguet probabilistic}).

\section{Approximate Correctors}  \label{S: approximate correctors}

In this section, we construct approximate correctors that will be used in the analysis of the asymptotics of \eqref{E: level set PDE} and \eqref{E: quasilinear}.  The use of approximate correctors in homogenization is by now standard (cf.\ \cite{lions souganidis} and \cite{caffarelli souganidis wang}).  The main difficulty here is the operators of interest are degenerate elliptic and, therefore, as we will see below, the cell problem is not well behaved in rational directions.

The result we need concerning approximate correctors is stated next.  The majority of this section is devoted to its proof.

\begin{theorem} \label{T: approximate correctors} If $F$ satisfies the assumptions of Theorem \ref{T: level set PDE} or $F$ takes the form
	\begin{equation*}
		F(p,X,y) = m(y,\hat{p})^{-1} \text{tr}(a(y,\hat{p})\tilde{X}_{\hat{p}})
	\end{equation*} 
with $m$ and $a$ satisfying the assumptions of Theorem \ref{T: quasilinear}, then, for each $(e,X) \in S^{d-1} \times \mathcal{S}_{d}$, there is a unique continuous function $F^{\perp}_{e}(X,\cdot) : \mathbb{T}^{d} \to \mathbb{R}$ varying only in the $e$ direction such that, for each $\nu > 0$, there is a (non-unique) $V^{\nu} \in C^{2}(\mathbb{T}^{d})$ satisfying
	\begin{equation} \label{E: approximate correctors}
		- \nu \leq F^{\perp}_{e}(X,\langle y,e \rangle e) - F(e,\tilde{X}_{e} + D^{2}_{e} V^{\nu}, y) \leq \nu \quad \text{in} \, \, \mathbb{T}^{d}.
	\end{equation}

If $e \in S^{d-1} \setminus \mathbb{R} \mathbb{Z}^{d}$, then there is a constant $\overline{F}(e,X)$ such that $F_{e}^{\perp}(X,\cdot) \equiv \overline{F}(e,X)$ in $\mathbb{T}^{d}$.  Furthermore, $\overline{F}(e,\cdot)$ is continuous, it satisfies \eqref{A: strong degenerate ellipticity}, $|\overline{F}(e,X)| \leq \Lambda \|X\|$, and 
	\begin{equation*}
		\overline{F}(e,X) = \overline{F}(e,\tilde{X}_{e}) \quad \text{if} \, \, X \in  \mathcal{S}_{d}.
	\end{equation*}  
\end{theorem} 

%

Following the proof of Theorem \ref{T: approximate correctors}, we describe the relation between the functions $\{F_{e}^{\perp}\}_{e \in S^{d-1}}$ and $\overline{F}$ and the invariant measures of Theorem \ref{T: invariant measures}.  Finally, at the end of the section, we prove results specific to the rational case that are used in the proofs of Theorem \ref{T: homogenization forced planes} and Corollary \ref{C: generic}. 

In Appendix \ref{S: diophantine}, we discuss sufficient conditions for the existence of a corrector (i.e.\ a solution of \eqref{E: approximate correctors} with $\nu = 0$) in the case when $a$ is independent of the spatial variable.

In this section and elsewhere in the paper, we will extensively use the differential operator $D^{2}_{e}$ defined by its action on smooth functions $\varphi$ in $\mathbb{R}^{d}$ by
	\begin{equation} \label{E: tangential derivative}
		D^{2}_{e}\varphi = (\text{Id} - e \otimes e) D^{2} \varphi (\text{Id} - e \otimes e).
	\end{equation}

\subsection{Cell Problem and Ergodic Constant}  To start with, we establish the existence of penalized correctors and discuss their convergence as $\delta \to 0^{+}$.  

	\begin{theorem}  \label{T: cell problem ergodic constant} Given $e \in S^{d-1}$, $X \in \mathcal{S}_{d}$, and $\delta > 0$, there is a unique $V^{\delta} \in C(\mathbb{T}^{d})$ solving the penalized cell problem 
		\begin{equation} \label{E: penalized cell problem second instance}
			\delta V^{\delta} - F(e,X + D^{2}V^{\delta},y) = 0 \quad \text{in} \, \, \mathbb{T}^{d}.
		\end{equation}
	Furthermore, there is a function $F^{\perp}_{e}(X,\cdot) : \mathbb{T}^{d} \to \mathbb{R}$ varying only in the $e$ direction such that
		\begin{equation} \label{E: uniform convergence penalized correctors}
			\lim_{\delta \to 0^{+}} \sup \left\{ |\delta V^{\delta}(y) + F^{\perp}_{e}(X,y)| \, \mid \, y \in \mathbb{T}^{d} \right\} = 0.
		\end{equation}
	In particular, if $e \notin \mathbb{R} \mathbb{Z}^{d}$, then there is a constant $\overline{F}(e,X) \in \mathbb{R}$ such that $F^{\perp}_{e}(X,\cdot) \equiv \overline{F}(e,X)$ in $\mathbb{T}^{d}$.  
	\end{theorem}  
	
		\begin{proof}  The existence and uniqueness of $V^{\delta} \in C(\mathbb{T}^{d})$ is proved in Appendix \ref{A: technical lemmata}.
		
		Given $y \in \mathbb{T}^{d}$, define $\tilde{V}^{\delta}_{y} : \langle e \rangle^{\perp} \to \mathbb{R}$ by 
			\begin{equation} \label{E: slices}
				\tilde{V}^{\delta}_{y}(x') = V^{\delta}(y + x').
			\end{equation}
		A perturbation argument (cf.\ \cite[Appendix B]{rates on networks}) shows that, no matter the choice of $y$, the function $\tilde{V}^{\delta}_{y}$ is solution of
			\begin{equation}
				\delta \tilde{V}^{\delta}_{y} - F(e, \tilde{X}_{e} + D^{2}_{e}\tilde{V}^{\delta}_{y}, y + x') = 0 \quad \text{in} \, \, \langle e \rangle^{\perp}.  \quad \label{E: slicing equation}
			\end{equation}
		(Here $D^{2}_{e}$ is given by \eqref{E: tangential derivative}.)  This is the penalized cell problem associated with a uniformly elliptic operator with quasi-periodic coefficients.  Therefore, by \cite[Lemma 9.1]{caffarelli souganidis}, there is a constant $F_{e}^{\perp}(X,y) \in \mathbb{R}$ such that
			\begin{equation*}
				\lim_{\delta \to 0^{+}} \sup \left\{ |\tilde{V}^{\delta}_{y}(x') + F_{e}^{\perp}(X,y)| \, \mid \, x' \in \langle e \rangle^{\perp} \right\} = 0.
			\end{equation*}
			
		Notice that $y \mapsto F_{e}^{\perp}(X,y)$ is well-defined and only varies in the $e$ direction.  Indeed, if $k \in \mathbb{Z}^{d}$, then $\tilde{V}^{\delta}_{y + k} = \tilde{V}^{\delta}_{y}$ so $F_{e}^{\perp}(X,y + k) = F_{e}^{\perp}(X,y)$.  To see that $F_{e}^{\perp}(X,y) = F_{e}^{\perp}(X,\langle y,e \rangle e)$, observe that
			\begin{equation*}
				- F_{e}^{\perp}(X,y) = \lim_{\delta \to 0^{+}} \delta \tilde{V}^{\delta}_{y}(0) = \lim_{\delta \to 0^{+}} \delta \tilde{V}^{\delta}_{\langle y,e \rangle e}(y - \langle y,e \rangle e) = - F_{e}^{\perp}(X,\langle y, e \rangle e).
			\end{equation*}
		This proves $F_{e}^{\perp}(X,\cdot)$ varies only in the $e$ direction.
		
	It remains to prove the uniform convergence of $(\delta V^{\delta})_{\delta > 0}$.  If $e \in \mathbb{R} \mathbb{Z}^{d}$, then \eqref{E: slicing equation} is a periodic homogenization problem in $\langle e \rangle^{\perp}$.  (Here the coefficients are invariant under the group $M_{e} = \mathbb{Z}^{d} \cap \langle e \rangle^{\perp}$, which is a rank $(d - 1)$-subgroup of $\langle e \rangle^{\perp}$.)  Thus, arguing as in \cite[Section 8]{caffarelli souganidis}, we see that the rate of convergence of $(\delta \tilde{V}^{\delta}_{y})_{\delta > 0}$ is uniform with respect to $y$.  From this, it is immediate that $(\delta V^{\delta})_{\delta > 0}$ converges uniformly in $\mathbb{T}^{d}$.  
	
	On the other hand, if $e \notin \mathbb{R} \mathbb{Z}^{d}$, then the density of each leaf of the foliation $\{\mathbb{T}^{d-1}_{e}(r)\}_{r \in [0,r_{e})}$ implies that, for each $y' \in \mathbb{T}^{d}$,
		\begin{equation*}
			\sup \left\{ |\delta V^{\delta}(y) + F_{e}^{\perp}(X,y')| \, \mid \, y \in \mathbb{T}^{d} \right\} = \sup \left\{ |\delta \tilde{V}^{\delta}_{y'}(x') + F_{e}^{\perp}(X,y')| \, \mid \, x' \in \langle e \rangle^{\perp} \right\}.
		\end{equation*}	
	Thus, $\lim_{\delta \to 0^{+}} \delta V^{\delta}(y) = -F_{e}^{\perp}(X,y')$ uniformly in $\mathbb{T}^{d}$, no matter the choice of $y'$.  Note that this shows $F_{e}^{\perp}(X,\cdot)$ is a constant in this case.  
		\end{proof}  
		
Now we prove the properties of $\overline{F}(e,\cdot)$ claimed in Theorem \ref{T: approximate correctors}:

	\begin{corollary} \label{C: basic properties} $\overline{F}(e,\cdot) : \mathcal{S}_{d} \to \mathbb{R}$ is continuous, it satisfies \eqref{A: strong degenerate ellipticity}, $|\overline{F}(e,X)| \leq \Lambda \|X\|$, and 
		\begin{equation*}
			\overline{F}(e,X) = \overline{F}(e,\tilde{X}_{e}) \quad \text{if} \, \, X \in \mathcal{S}_{d}.
		\end{equation*}  
	\end{corollary}  
	
		\begin{proof}  By the proof of Theorem \ref{T: cell problem ergodic constant}, $\tilde{X}_{e} \mapsto \overline{F}(e,\tilde{X}_{e})$ is the homogenized operator associated with the (e.g.\ Dirichlet) homogenization problem
			\begin{equation*}
				\left\{ \begin{array}{r l}
					- F(e,D^{2}_{e} U^{\epsilon},\epsilon^{-1} x') = 0 & \text{in} \, \, \langle e \rangle^{\perp} \\
					U^{\epsilon} = G & \text{on} \, \, \langle e \rangle^{\perp}
				\end{array} \right.
			\end{equation*}
		Therefore, the claims follow directly from \cite{caffarelli souganidis wang}.\end{proof}  
		
Finally, we define the homogenized operator $\overline{F} : (\mathbb{R}^{d} \setminus \mathbb{R} \mathbb{Z}^{d}) \times \mathcal{S}_{d} \to \mathbb{R}$ by
	\begin{equation} \label{E: homogenized operator}
		\overline{F}(p,X) = \|p\| \overline{F}(\hat{p},\|p\|^{-1} X).
	\end{equation}
		
%
	
\subsection{Regularity Estimates}  We now turn to the regularity of the penalized correctors $(V^{\delta})_{\delta > 0}$ when the operator $F$ is sufficiently regular.  To start with, we consider the two dimensional case since it benefits from special structure.  We then turn to the operators of Theorem \ref{T: quasilinear}, which admit more-or-less smooth penalized correctors due to the strong assumptions imposed on the matrix field $a$.

Here is the result when $d =2$:

	\begin{prop}  Let $d = 2$ and assume that $F(e,\cdot) : \mathcal{S}_{d} \times \mathbb{T}^{d} \to \mathbb{R}$ is twice continuously differentiable.  Given $\delta > 0$, $e \in S^{1}$, and $X \in \mathcal{S}_{2}$, the penalized corrector $V^{\delta}$ solving \eqref{E: penalized cell problem second instance} is in $C^{2}(\mathbb{T}^{2})$.  \end{prop}
	
		\begin{proof}  Let $\{\tilde{V}^{\delta}_{y}\}_{y \in \mathbb{T}^{2}}$ be the functions defined in \eqref{E: slices}.  Since $|F(e,X,y)| \leq \Lambda \|X\|$, it follows that $\|V^{\delta}\|_{L^{\infty}(\mathbb{T}^{2})} \leq \Lambda \delta^{-1} \|X\|$.  From this, a straightforward computation involving \eqref{A: strong degenerate ellipticity} (and the fact that $\langle e \rangle^{\perp}$ is one-dimensional) shows that 
			\begin{equation*}
				-\frac{2 \Lambda \|X\|}{\lambda} \leq -D^{2}_{e} \tilde{V}_{y}^{\delta} \leq \frac{2 \Lambda \|X\|}{\lambda} \quad \text{in the viscosity sense in} \, \, \langle e \rangle^{\perp}.
			\end{equation*}
		Hence $\{\tilde{V}^{\delta}_{y}\}_{y \in \mathbb{T}^{2}}$ are bounded in $C^{1,1}$.  The equation uniquely and continuously determines the functions $\{D^{2}_{e}\tilde{V}^{\delta}_{y}\}_{y \in \mathbb{T}^{2}}$ by \eqref{A: strong degenerate ellipticity} so actually we are working with $C^{2}$ functions.  
		
		Manipulating difference quotients and invoking assumptions \eqref{A: strong degenerate ellipticity} and the $C^{2}$ assumption on $F$, we apply the Krylov-Safonov Theorem to find that $\{\tilde{V}^{\delta}_{y}\}_{y \in \mathbb{T}^{2}} \subseteq C^{2,\alpha}(\langle e \rangle^{\perp})$.  
		
		Next, we show that $y \mapsto \tilde{V}^{\delta}_{y}$ is $C^{2,\alpha}$.  Fix $y' \in \langle e \rangle^{\perp}$.  Using difference quotients again, we see that the function $\frac{\partial \tilde{V}^{\delta}_{y}}{\partial y_{e}}$ given by 
			\begin{equation*}
				\frac{\partial \tilde{V}^{\delta}_{y}}{\partial y_{e}}(x') = \lim_{h \to 0} \frac{\tilde{V}^{\delta}_{y + h e}(x') - \tilde{V}^{\delta}_{y}(x')}{h} \quad (x' \in \langle e \rangle^{\perp})
			\end{equation*}
		is solution of the uniformly elliptic linear PDE
			\begin{equation*}
				\delta \frac{\partial \tilde{V}^{\delta}_{y}}{\partial y_{e}} - D_{A}F(e,\tilde{X}_{e} + D^{2}_{e} \tilde{V}^{\delta}_{y},y + x') : D^{2}_{e} \left(\frac{\partial \tilde{V}^{\delta}_{y}}{\partial y_{e}}\right) - \langle D_{y} F(e,\tilde{X}_{e} + D^{2}_{e}\tilde{V}^{\delta}_{y}, y + \tilde{x}), e \rangle = 0 \quad \text{in} \, \, \langle e \rangle^{\perp}.
			\end{equation*}
		Thus, $\frac{\partial \tilde{V}^{\delta}_{y}}{\partial y_{e}}$ is $C^{2}$ in $\langle e \rangle^{\perp}$ and the uniform ellipticity gives a constant $C > 0$ such that
			\begin{align*}
				\sup \left\{ \delta \left\| \frac{\partial \tilde{V}^{\delta}_{y}}{\partial y_{e}} \right\|_{L^{\infty}(\langle e \rangle^{\perp})} + \left\| D^{2} \left( \frac{\partial \tilde{V}^{\delta}_{y}}{\partial y_{e}} \right) \right\|_{L^{\infty}(\langle e \rangle^{\perp})}  \, \mid \, y \in \mathbb{T}^{2} \right\}  \leq C.
			\end{align*}
		Differentiating again, we find that the second derivative $\frac{\partial^{2} \tilde{V}^{\delta}_{y}}{\partial y_{e}^{2}}$ (defined anlogously) is $C^{2}$ and satisfies
			\begin{equation*}
				\sup \left\{ \delta \left\| \frac{\partial^{2} \tilde{V}^{\delta}_{y}}{\partial y_{e}^{2}} \right\|_{C(\langle e \rangle^{\perp})} + \left\| D^{2} \left( \frac{\partial^{2} \tilde{V}^{\delta}_{y}}{\partial y_{e}^{2}} \right) \right\|_{L^{\infty}(\langle e \rangle^{\perp})}  \, \mid \, y \in \mathbb{T}^{2} \right\}  \leq C \delta^{-1}.
			\end{equation*}
			
		Finally, observe that if $y \in \mathbb{T}^{d}$ and $e' \in S^{d-1} \cap \langle e \rangle^{\perp}$, then
			\begin{align*}
				D^{2}_{e} V^{\delta}(y) = D^{2}_{e}\tilde{V}^{\delta}_{y}(0), &\quad \langle D^{2} V^{\delta}(y) e, e \rangle = \frac{\partial^{2} \tilde{V}^{\delta}_{y}}{\partial y_{e}^{2}}(0), \\ 
				\langle D^{2}V^{\delta}(y) e, e' \rangle &= \left \langle D \left(\frac{\partial \tilde{V}^{\delta}_{y}}{\partial y_{e}} \right)(0), e' \right \rangle.
			\end{align*}
		Therefore, the previous considerations show that $V^{\delta} \in C^{2}(\mathbb{T}^{d})$.  
		\end{proof}    
		
Finally, in the quasi-linear set-up, we have

	\begin{prop} \label{P: quasilinear regularity} If $a$ and $m$ satisfy the assumptions of Theorem \ref{T: quasilinear} and $F(p,X,y) = m(y,\hat{p})^{-1} \text{tr}(a(y,\hat{p})\tilde{X}_{\hat{p}})$, then, for each $\delta > 0$, the penalized corrector $V^{\delta}$ of \eqref{E: penalized cell problem second instance} satisfies $V^{\delta} \in C^{2,\alpha}(\mathbb{T}^{d})$.  \end{prop}
	
		\begin{proof}  The proof proceeds as in the last proposition.  To obtain $C^{2,\alpha}$ estimates, we use Schauder estimates for linear elliptic equations (cf.\ \cite[Chapter 6]{gilbarg trudinger} or \cite[Chapter 8]{caffarelli cabre}).  \end{proof}  
		
\subsection{Approximate Correctors}  This section is devoted to the proof of Theorem \ref{T: approximate correctors}.  In general, in the setting of Theorem \ref{T: level set PDE}, if $F$ is not twice differentiable, we approximate it by a family of regularized operators $(F^{\mu})_{\mu > 0}$ that are.  This is made precise in the next result:

	\begin{prop} \label{P: approximation} Given $e \in S^{1}$ and $F$ satisfying the assumptions of Theorem \ref{T: level set PDE}, there is a family $(F_{e}^{\mu})_{\mu > 0}$ of operators, twice continuously differentiable in $\mathcal{S}_{2} \times \mathbb{T}^{2}$, such that, for each $\mu > 0$,
		\begin{align}
			\lambda \|\tilde{N}^{+}_{e}\| - \Lambda \|\tilde{N}_{e}^{-}\| \leq F^{\mu}_{e}(M + N,x) - F^{\mu}_{e}(M,x) \leq \Lambda \|\tilde{N}_{e}^{+}\| - \lambda \|\tilde{N}_{e}^{-}\| \quad \text{if} \, \, M,N \in \mathcal{S}_{2}, \\
			|F_{e}^{\mu}(M,y') - F_{e}^{\mu}(M,y)| \leq C(1 + \mu + \|M\|)\|x' - x\| \quad \text{if} \, \, M \in \mathcal{S}_{2}, \, \, y, y' \in \mathbb{T}^{2}, \\
			|F_{e}^{\mu}(M,y) - F(e,M,y)| \leq \Lambda \mu + \omega^{e}_{1 + \|M\|}(\mu)\quad \text{if} \, \, (M,y) \in \mathcal{S}_{2} \times \mathbb{T}^{2}.
		\end{align}
	Here, for each $R > 0$, $\omega_{R}^{e}$ is the modulus of $F$ in $\{e\} \times B(0,R) \times \mathbb{T}^{d}$, that is,
		\begin{equation*}
			\omega_{R}^{e}(\delta) = \sup \left\{ |F(e,X,y) - F(e,X',y')| \, \mid \, (X,y), (X',y') \in B(0,R) \times \mathbb{T}^{d}, \, \, \|y - y'\| \leq \delta \right\}.
		\end{equation*}
	\end{prop}  
	
	Proposition \ref{P: approximation} can be proved by mollifying the operator $(X,y) \mapsto F(e,X,y)$.  The details are left to the interested reader.
	
	\begin{proof}[Proof of Theorem \ref{T: approximate correctors}]  We will assume that $F$ satisfies the assumptions of Theorem \ref{T: level set PDE}.  If instead $F$ satisfies the assumptions of Theorem \ref{T: quasilinear}, then Proposition \ref{P: quasilinear regularity} shows that the penalized correctors of Theorem \ref{T: cell problem ergodic constant} are $C^{2}$ so there is no need to regularize.
	
	Fix $\nu > 0$.  Let $\delta, \mu > 0$ be free variables and let $V^{\delta,\mu} \in C^{2}(\mathbb{T}^{2})$ be the solution of \eqref{E: penalized cell problem second instance} with $F$ replaced by $F^{\mu}$.  Recall that
		\begin{equation} \label{E: Hessian estimate}
			\|D^{2}_{e}V^{\delta}\|_{L^{\infty}(\mathbb{T}^{2})} \leq \frac{2 \Lambda \|X\|}{\lambda}.
		\end{equation}
	Thus, if $y \in \mathbb{T}^{2}$, then
		\begin{align*}
			|\overline{F}(e,X) - F(e,\tilde{X}_{e} + D^{2}_{e} V^{\delta},y)| &\leq (I) + (II) + (III),
		\end{align*}
	where
		\begin{align*}
			(I) &= \|\overline{F}^{\mu}(e,X) + \delta V^{\delta}\|_{L^{\infty}(\mathbb{T}^{2})}, \\
			(II) &= |F(e,\tilde{X}_{e} + D^{2}_{e}V^{\delta},y) - F^{\mu}_{e}(\tilde{X}_{e} + D^{2}_{e} V^{\delta},y)| \\
				&\leq \Lambda \mu + \omega^{e}_{1 + 2 \lambda^{-1} \Lambda \|X\|}(\mu),  \\
			(III) &= |\overline{F}(e,X) - \overline{F}^{\mu}(e,X)|.
		\end{align*}
	
	Using Lemma \ref{L: comparison ergodic constants} below and \eqref{E: Hessian estimate}, it is not hard to show that 
		\begin{equation*}
			|\overline{F}(e,X) - \overline{F}^{\mu}(e,X)| \leq  \|\overline{F}^{\mu}(e,X) + \delta V^{\delta}\|_{L^{\infty}(\mathbb{T}^{2})} + \Lambda \mu + \omega^{e}_{1 + (1 + 2 \lambda^{-1} \Lambda)\|X\|)}(\mu).
		\end{equation*}
	Indeed, the approximate corrector $V^{\delta,\mu}$ satisfies
		\begin{equation*}
			\delta V^{\delta,\mu} - F^{\mu}(e,\tilde{X}_{e} + D^{2}_{e} V^{\delta,\mu}) = 0 \quad \text{in} \, \, \mathbb{T}^{2}.
		\end{equation*}
	Thus, \eqref{E: Hessian estimate} and the choice of $F^{\mu}$ yield
		\begin{equation*}
			| -F(e,\tilde{X}_{e} + D^{2}_{e} V^{\delta,\mu}) + \overline{F}^{\mu}(e,X)| \leq \|\overline{F}^{\mu}(e,X) + \delta V^{\delta}\|_{L^{\infty}(\mathbb{T}^{2})} + \Lambda \mu + \omega^{e}_{1 + (1 + 2 \lambda^{-1} \Lambda)\|X\|}(\mu) \quad \text{in} \, \, \mathbb{T}^{2}.
		\end{equation*}
	Hence, by Lemma \ref{L: comparison ergodic constants}, we have
		\begin{equation*}
			(III) = |\overline{F}(e,X) - \overline{F}^{\mu}(e,X)| \leq \|\overline{F}^{\mu}(e,X) + \delta V^{\delta}\|_{L^{\infty}(\mathbb{T}^{2})} + \Lambda \mu + \omega^{e}_{1 + (1 + 2 \Lambda)\|X\|}(\mu).
		\end{equation*}

	We conclude that if we first choose $\mu > 0$ small enough that 
		\begin{equation*}
			\Lambda \mu + \omega^{e}_{1 + (1 + 2 \lambda^{-1} \Lambda)\|X\|)}(\mu) \leq \frac{\nu}{3}
		\end{equation*}
	and then choose $\delta > 0$ so small that $ \|\overline{F}^{\mu}(e,X) + \delta V^{\delta}\|_{L^{\infty}(\mathbb{T}^{2})} \leq \nu/3$, then we obtain
		\begin{equation*}
			- \nu \leq \overline{F}(e,X) - F(e,\tilde{X}_{e} + D^{2}V^{\delta}_{e}, y ) \leq \nu \quad \text{in} \, \, \mathbb{T}^{2}.
		\end{equation*}  
	\end{proof}

\subsection{Invariant Measures} \label{S: invariant measures} This section makes the connection between the homogenized operators $\overline{F}$ and oscillating functions $F^{\perp}_{e}$ obtained in Theorem \ref{T: approximate correctors} and the invariant measures of Theorem \ref{T: invariant measures}.  

To start with, we give the basic existence result for invariant measures:

	\begin{prop} \label{P: existence invariant measures}   Assume that $a$ satisfies the assumptions of Theorem \ref{T: quasilinear}.  For each $e \in S^{d-1}$, $\mathscr{I}^{a}_{e}$ is non-empty.  Furthermore, if $e \in \mathbb{R} \mathbb{Z}^{d}$, then there is an $r_{e}$-periodic function $\mu_{e} : \mathbb{R} \to \mathscr{I}^{a}_{e}$, $\mu_{e} : s \mapsto \mu_{e}^{s}$, such that, for each $r \in [0,r_{e})$, the support of $\mu_{e}$ equals $\mathbb{T}^{d-1}_{e}(r)$ and $\mu_{e}^{s} \ll \mathcal{H}^{d-1} \restriction_{\mathbb{T}^{d-1}_{e}(r)}$.
	\end{prop}  
	
		\begin{proof}  First, assume that $e \in \mathbb{R} \mathbb{Z}^{d}$.  Given $r \in [0,r_{e})$, the trajectories of the SDE \eqref{E: SDE} with $\langle X^{e}_{0}, e \rangle = r$ satisfy $\langle X^{e}_{t},e \rangle = r$ for all $t > 0$.  Hence, as long as we restrict to initial distributions concentrated in $\mathbb{T}^{d-1}_{e}(r)$, we can consider \eqref{E: SDE} as a process in $\mathbb{T}^{d-1}_{e}(r)$.  Restricted to $\mathbb{T}^{d-1}_{e}(r)$, \eqref{E: SDE} is uniformly non-degenerate.  Therefore, it has a unique invariant measure $\mu_{e}^{r}$ that has full support in $\mathbb{T}^{d-1}_{e}(r)$ and $\mu_{e}^{r} \ll \mathcal{H}^{d-1} \restriction_{\mathbb{T}^{d-1}_{e}(r)}$ (cf.\ \cite{asymptotic analysis} or \cite{varadhan}).
		
		If $e \notin \mathbb{R} \mathbb{Z}^{d}$, then there is a sequence $(e_{n})_{n \in \mathbb{N}} \subseteq S^{d-1} \cap \mathbb{R} \mathbb{Z}^{d}$ such that $e = \lim_{n \to \infty} e_{n}$.  By the Banach-Alaoglu Theorem, $(\mu_{e_{n}}^{0})_{n \in \mathbb{N}}$ has a weak-$*$ accumulation point $\mu$.  A quick argument shows that $\mu \in \mathscr{I}^{a}_{e}$ necessarily holds.  Further, Lemma \ref{L: equidistribution} in the appendix shows that $\mu \ll \mathcal{L}^{d}$.      \end{proof}  
		
We still need to understand the structure of the sets $\{\mathscr{I}^{a}_{e}\}_{e \in S^{d-1}}$.  Here and in Section \ref{S: linear response} it will be convenient to use the following variant of Theorem \ref{T: approximate correctors}:

	\begin{prop} \label{P: approximate correctors linear}  If $a$ satisfies the assumptions of Theorem \ref{T: quasilinear} and $f \in C(\mathbb{T}^{d})$, then, for each $e \in S^{d-1}$ and $\delta > 0$, there is a unique $V^{\delta} \in C(\mathbb{T}^{d})$ solving the degenerate elliptic PDE:
		\begin{equation} \label{E: penalized cell problem linear}
			\delta V^{\delta} - \text{tr} (a(y,e) D^{2}_{e} V^{\delta}) = f(y) \quad \text{in} \, \, \mathbb{T}^{d}.
		\end{equation}
	Furthermore, there is an $f^{\perp}_{e} \in C(\mathbb{T}^{d})$ varying only in the $e$ direction such that 
		\begin{equation*}
			\lim_{\delta \to 0^{+}} \|\delta V^{\delta} - f^{\perp}_{e}\|_{L^{\infty}(\mathbb{T}^{d})} = 0.
		\end{equation*}
	If $e \notin \mathbb{R} \mathbb{Z}^{d}$, then there is a constant $\bar{f}(e) \in \mathbb{R}$ such that $f^{\perp}_{e} \equiv \bar{f}(e)$.
	
	Finally, if $f \in C^{2,\alpha}(\mathbb{T}^{d})$, then $V^{\delta} \in C^{2,\alpha}(\mathbb{T}^{d})$.  
	\end{prop}  

This readily leads to the proof of Theorem \ref{T: invariant measures}.  In addition, we will show that the operator $F^{\perp}_{e}$ of Theorem \ref{T: approximate correctors} is determined in this case by the formula
	\begin{equation} \label{E: oscillating function linear}
		F^{\perp}_{e}(X,y) = m^{\perp}_{e}(y)^{-1} \text{tr} \left(a^{\perp}_{e}(y) \tilde{X}_{e}\right),
	\end{equation}
where $a^{\perp}_{e}$ is given by \eqref{E: oscillating tensor} and $m^{\perp}_{e}$ is defined analogously.

	\begin{proof}[Proof of Theorem \ref{T: invariant measures}]  We start with (i).  Assume that $e \notin \mathbb{R} \mathbb{Z}^{d}$.  By the Riesz Representation Theorem, we can define a Borel probability measure $\bar{\mu}_{e}$ in $\mathbb{T}^{d}$ by 
		\begin{equation*}
			\int_{\mathbb{T}^{d}} f(y) \bar{\mu}_{e}(dy) = \bar{f}(e),
		\end{equation*}
	where $\bar{f}(e)$ is the constant in Theorem \ref{P: approximate correctors linear}.  We claim that $\mathscr{I}^{a}_{e} = \{\bar{\mu}_{e}\}$.  
	
	Indeed, if $\mu \in \mathscr{I}^{a}_{e}$, $f \in C^{2,\alpha}(\mathbb{T}^{d})$, and $V^{\delta}$ is the associated solution of \eqref{E: penalized cell problem linear}, then
		\begin{equation*}
			\int_{\mathbb{T}^{d}} \delta V^{\delta}(y) \, \mu(dy) = \int_{\mathbb{T}^{d}} f(y) \, \mu(dy).
		\end{equation*}
	Sending $\delta \to 0^{+}$, we find $\int_{\mathbb{T}^{d}} f(y) \, \mu(dy) = \bar{f}(e)$.  This proves $\mu = \bar{\mu}_{e}$ by definition.  Since $\mathscr{I}^{a}_{e}$ is non-empty by Proposition \ref{P: existence invariant measures}, it also shows that $\bar{\mu}_{e} \in \mathscr{I}^{a}_{e}$.
	
	Next, we turn to (ii).  It remains to show that $\mathscr{I}^{a}_{e}$ is the closed convex hull of $\{\mu_{e}^{s} \, \mid \, s \in \mathbb{R}\}$.  By the Krein-Milman Theorem, we only have to prove that the latter equals the set of extreme points.  
		
		First, notice that if $f \in C(\mathbb{T}^{d})$, $(V^{\delta})_{\delta > 0}$ are the associated solutions of \eqref{E: penalized cell problem linear}, and $s \in [0,r_{e})$, then Proposition \ref{P: approximate correctors linear} and the definition of $\mathscr{I}^{a}_{e}$ give
			\begin{align*}
				\int_{\mathbb{T}^{d}} f(y) \, \mu_{e}^{s}(dy) &= \lim_{\delta \to 0^{+}} \int_{\mathbb{T}^{d}} \left(\delta V^{\delta}(y) - \text{tr} \left(A(y,e) D^{2}V^{\delta}(y)\right) \right) \, \mu_{e}^{s}(dy) \\
					&= \lim_{\delta \to 0^{+}} \int_{\mathbb{T}^{d}} \delta V^{\delta}(y) \, \mu_{e}^{s}(dy) = f_{e}^{\perp}(se).
			\end{align*}
		
		Now assume $\mu$ is an extreme point of $\mathcal{I}_{e}$.  Repeating the previous computation with $\mu^{s}_{e}$ replaced by $\mu$, we find, for each $f \in C(\mathbb{T}^{d})$, 
			\begin{equation*}
				\int_{\mathbb{T}^{d}} f(y) \, \mu(dy) = \int_{\mathbb{T}^{d}} f_{e}^{\perp}(\langle y,e \rangle e) \, \mu(dy) = \int_{\mathbb{T}^{d}} \left(\int_{\mathbb{T}^{d}} f(y') \, \mu_{e}^{\langle y,e \rangle}(dy') \right) \, \mu(dy).
			\end{equation*}
		In particular, letting $\nu \in \mathscr{P}(r_{e} \mathbb{T})$ be the push-forward of $\mu$ onto $r_{e} \mathbb{T}$ given by 
			\begin{equation*}
				\nu(A) = \mu(\{y \in \mathbb{T}^{d} \, \mid \, \langle y,e \rangle \in A\}),
			\end{equation*}
		we obtain $\mu = \int_{r_{e} \mathbb{T}} \mu_{e}^{s} \, \nu(ds)$.
		Notice that if $A \subseteq r_{e} \mathbb{T}$ and $\nu(A) > 0$, then 
			\begin{equation*}
				\mu = \nu(A) \fint_{A} \mu_{e}^{s} \, \nu(ds) + (1 - \nu(A)) \fint_{A^{c}} \mu_{e}^{s} \, \nu(ds).
			\end{equation*}
		Hence, since $\mu$ is an extreme point of $\mathcal{I}_{e}$, either $\nu(A) = 1$ or $\nu(A) = 0$.  From this, an elementary argument shows that $\nu = \delta_{s'}$ for some $s' \in [0,r_{e})$.  In particular, $\mu = \mu_{e}^{s'}$.  This proves $\{\mu_{e}^{s} \, \mid \, s \in \mathbb{R}\}$ equals the set of extreme points of $\mathscr{I}^{a}_{e}$.
		
		Finally, we prove (iii).  If $m \equiv 1$ in Theorem \ref{T: approximate correctors}, $\delta > 0$, and $X \in \mathcal{S}_{d}$, then the solution $V^{\delta}$ of \eqref{E: penalized cell problem second instance} solves \eqref{E: penalized cell problem linear} with $f(y) = \text{tr} (a(y,e) X)$.  Hence 
			\begin{align*}
				\overline{F}(e,X) &= \int_{\mathbb{T}^{d}} \text{tr} (a(y,e)X) \, \bar{\mu}_{e}(dy) \quad \text{if} \, \, e \notin \mathbb{R} \mathbb{Z}^{d}, \\
				F^{\perp}_{e}(X,y) &= \int_{\mathbb{T}^{d}} \text{tr}(a(y',e) X) \, \mu_{e}^{\langle y,e \rangle}(dy') \quad \text{if} \, \, e \in \mathbb{R} \mathbb{Z}^{d}.
			\end{align*}
		This shows that the operator $\overline{F}$ is given by \eqref{E: homogenized coefficient linear} with $\bar{a}$ and $a^{\perp}_{e}$ as in \eqref{E: effective a and m} and \eqref{E: oscillating tensor}.  
		
		When $m \not \equiv 1$, the solution $V^{\delta}$ of \eqref{E: penalized cell problem second instance} is also solution of
			\begin{equation*}
				\delta m(y,e) V^{\delta} - \text{tr} \left(a(y,e) D^{2}_{e} V^{\delta}\right) = \text{tr} \left(a(y,e) X \right) \quad \text{in} \, \, \mathbb{T}^{d}.
			\end{equation*}
		Given $\epsilon > 0$, for small enough $\delta$, this yields
			\begin{equation*}
				|F^{\perp}_{e}(X,y) m(y,e) - \text{tr} \left(a(y,e) D^{2}_{e} V^{\delta} \right) - \text{tr} \left(a(y,e) X \right)| \leq \epsilon \quad \text{in} \, \, \mathbb{T}^{d}.
			\end{equation*}
		When $e \in \mathbb{R} \mathbb{Z}^{d}$, we integrate with respect to $\mu_{e}^{s}$ for a given $s \in [0,r_{e})$ to find
			\begin{equation*}
				\left|F^{\perp}_{e}(X,se) \int_{\mathbb{T}^{d}} m(y',e) \, \mu_{e}^{s}(dy') - \text{tr} \left( a^{\perp}_{e}(se) X\right) \right| \leq \epsilon.
			\end{equation*}
		The arbitrariness of $\epsilon$ implies \eqref{E: oscillating function linear}.  A similar computation applies in the case that $e \notin \mathbb{R} \mathbb{Z}^{d}$, giving \eqref{E: homogenized coefficient linear}.    \end{proof}  
		
\subsection{Correctors in Rational Directions} \label{S: rational correctors}  Finally, we build correctors of the cell problem when $e \in \mathbb{R} \mathbb{Z}^{d}$ and study the regularity of the oscillating function $f_{e}^{\perp}$.  These will be convenient in the analysis of the forced problem \eqref{E: quasilinear forced} as well as the proof of Corollary \ref{C: generic} concerning generic discontinuities.

This section is devoted to the study of the following cell problem:
	\begin{equation} \label{E: exotic cell problem}
		- \text{tr} \left( a(y,e) D^{2}_{e} \tilde{V}_{e} \right) = f(y) - f_{e}^{\perp}(\langle y,e \rangle e) \quad \text{in} \, \, \mathbb{T}^{d}.
	\end{equation}
Notice that if a solution exists, then $f_{e}^{\perp}$ is necessarily as in Proposition \ref{P: approximate correctors linear}.  

	\begin{prop} \label{P: rational correctors}  If $a$ satisfies the assumptions of Theorem \ref{T: quasilinear} and $f \in C^{2,\alpha}(\mathbb{T}^{d})$, then there is a solution $V_{e} \in C^{2,\alpha}(\mathbb{T}^{d})$ of \eqref{E: exotic cell problem} and $f_{e}^{\perp} \in C^{2,\alpha}(\mathbb{T}^{d})$.  Further, if $\overline{V}_{e} \in C^{2,\alpha}(\mathbb{T}^{d})$ is any other solution of \eqref{E: exotic cell problem}, then $V_{e} - \overline{V}_{e}$ varies only in the $e$ direction.
	\end{prop}  
	
Note that this implies $a^{\perp}_{e}$ and $m^{\perp}_{e}$ of \eqref{E: oscillating function linear} are both $C^{2,\alpha}$.
	
	\begin{proof}[Proof of Proposition \ref{P: rational correctors}]  To start with, let $\varphi \in C^{2,\alpha}([0,r_{e}])$ be a function with $\varphi(0) = 0$.  For each $s \in [0,r_{e})$, let $\tilde{V}_{se} : \mathbb{T}^{d-1}_{e}(0) \to \mathbb{R}$ be the solution of the cell problem
			\begin{equation*}
				\left\{ \begin{array}{r l}
					 - \text{tr} \left( a(se + x') D^{2}_{e}\tilde{V}_{se} \right) = f(se + x') - f_{e}^{\perp}(se) & \text{in} \, \, \mathbb{T}^{d-1}_{e}(0) \\
					\tilde{V}_{se}(0) = \varphi(s)
				\end{array} \right.
			\end{equation*}
		Notice that $\tilde{V}_{se}$ is unique.
		
		\textbf{Step 1: Preliminary regularity of $f_{e}^{\perp}$} 
		
		We claim that there is a constant $C > 0$ such that
			\begin{equation*}
				|f_{e}^{\perp}((s + h)e) - f_{e}^{\perp}(se)| \leq C |h| \quad \text{if} \, \, s, h \in \mathbb{R}.
			\end{equation*}	
		
		Indeed, first, recall (cf.\ \cite[Lemma 3.1]{evans} or \cite[Lemma 2.1]{camilli marchi}) that there is a $C_{*} > 0$ depending only on $\lambda, \Lambda$, and $\|f\|_{L^{\infty}(\mathbb{T}^{d})}$ such that 
			\begin{equation*}
				\sup \left\{ \|\tilde{V}_{se} - \tilde{V}_{se}(0)\|_{L^{\infty}(\mathbb{T}^{d-1}_{e}(0))} \, \mid \, s \in [0,r_{e}) \right\} \leq C_{*}.
			\end{equation*}
		Therefore, by Schauder estimates,
			\begin{equation*}
				C_{0} := \sup \left\{ \|D^{2}_{e} \tilde{V}_{se}\|_{L^{\infty}(\mathbb{T}^{d-1}_{e}(0))} \, \mid \, s \in [0,r_{e}) \right\} < \infty.
			\end{equation*}
		Therefore, by \eqref{A: a C2}, given $s \in [0,r_{e})$, the function $\tilde{V}_{(s + h)e}$ satisfies
			\begin{align*}
				-f(se + x') - \text{tr} \left( a(se + x') D^{2}_{e} \tilde{V}_{(s + h)e} \right) &\leq -f_{e}^{\perp}((s + h)e) + (C_{0} \|a(\cdot,e)\|_{C^{1,\alpha}(\mathbb{T}^{d})} + \|f\|_{C^{1,\alpha}(\mathbb{T}^{d})}) |h| \\
				-f(se + x') - \text{tr} \left( a(se + x') D^{2}_{e} \tilde{V}_{(s + h)e} \right) &\geq -f_{e}^{\perp}((s + h)e) - \left( C_{0} \|a(\cdot,e)\|_{C^{1,\alpha}(\mathbb{T}^{d})} + \|f\|_{C^{1,\alpha}(\mathbb{T}^{d})} \right) |h|
			\end{align*}
		From this and Lemma \ref{L: comparison ergodic constants}, we find
			\begin{equation*}
				|f_{e}^{\perp}((s + h)e) - f_{e}^{\perp}(se)| \leq \left(C_{0} \|a(\cdot,e)\|_{C^{1,\alpha}(\mathbb{T}^{d})} + \|f\|_{C^{1,\alpha}(\mathbb{T}^{d})} \right) |h|
			\end{equation*}
			
		\textbf{Step 2: Convenient extension of $\{\tilde{V}_{se}\}_{s \in [0,r_{e})}$}  
		
			At this stage, it is more-or-less inevitable to extend the function $s \mapsto \tilde{V}_{se}$ from $[0,r_{e})$ to $\mathbb{R}$.  Given $r \in \mathbb{R}$, define $\tilde{V}_{re} \in C(\mathbb{T}^{d-1}_{e}(0))$ by 
				\begin{equation*}
					\tilde{V}_{re}(x') = \tilde{V}_{(r - Nr_{e})e}(x' + Nr_{e}e) \quad \text{if} \, \, 0 \leq r - Nr_{e} < r_{e} \, \, \text{and} \, \, N \in \mathbb{Z}.
				\end{equation*}  
		
		We refine the choice of $\varphi$ to make $r \mapsto \tilde{V}_{re}$ continuous.  Since it coincides with $\varphi$, $r \mapsto \tilde{V}_{re}(0)$ is differentiable at each $r \in (0,r_{e})$.  To get a derivative at $0$ or $r_{e}$, we constrain $\varphi$ so that the following identity holds:
			\begin{equation*}
				\varphi(r_{e}) = \tilde{V}_{0}(r_{e}e).
			\end{equation*}
		Note this isn't circular since $\tilde{V}_{0}$ is determined by $\varphi(0)$ alone.  We leave it to the reader to check that now $r \mapsto \tilde{V}_{re}$ takes $\mathbb{R}$ continuously into $C(\mathbb{T}^{d-1}_{e}(0))$.
			
		\textbf{Step 3: Differentiability in $e$ direction}
		
		To get a derivative of $r \mapsto \tilde{V}_{re}$, we further refine the choice of $\varphi$.  First of all, notice that, with the Lipschitz estimate on $f_{e}^{\perp}$ in hand, we deduce that there is a $C_{1} > 0$ such that, for each $s \in [0,r_{e})$ and $h > 0$ small enough, the function $\tilde{W}^{h}_{se} = \tilde{V}_{(s + h)e} - \tilde{V}_{se}$ satisfies
			\begin{align*}
				\left\{ \begin{array}{r l}
					-C_{1} |h| \leq - \text{tr} \left( a(se + x') D^{2}_{e} \tilde{W}^{h}_{se}\right) \leq C_{1} |h| & \text{in} \, \, \mathbb{T}^{d - 1}_{e}(0) \\
					\tilde{W}^{h}_{se}(0) = \varphi(s + h) - \varphi(s)
				\end{array} \right.
			\end{align*}
		From this, the compactness of $\mathbb{T}^{d-1}_{e}(0)$, and the Krylov-Safonov Theorem, it follows that there is a $B > 0$ depending only on $C_{1}$ such that
			\begin{equation*}
				\|\tilde{W}^{h}_{se}\|_{L^{\infty}(\mathbb{T}^{d - 1}_{e}(0))} \leq B|h|.
			\end{equation*}
		
		We caim that $\tilde{W}_{se} = \lim_{h \to 0^{+}} h^{-1}\tilde{W}^{h}_{se}$ exists in $C(\mathbb{T}^{d})$ for each $s \in [0,r_{e})$.  Indeed, suppose $(h_{n})_{n \in \mathbb{N}} \subseteq (0,\infty)$ is a sequence such that $\lim_{n \to \infty} h_{n} = 0$ and $h_{n}^{-1} \tilde{W}^{h_{n}}_{se} \to \tilde{W}$ uniformly in $\mathbb{T}^{d-1}_{e}(0)$ for some $\tilde{W} \in C(\mathbb{T}^{d-1}_{e}(0))$.  Passing to a subsequence, if necessary, we can assume that there is a $\tilde{c} \in \mathbb{R}$ such that
		\begin{equation*}
			\tilde{c} = \lim_{n \to \infty} \frac{f_{e}^{\perp}((s + h_{n})e) - f_{e}^{\perp}(se)}{h_{n}}.
		\end{equation*}  
	It follows that $\tilde{W}$ is a viscosity solution of
			\begin{equation*}
				-\langle Df(se + x'), e \rangle - \text{tr} \left(a(se + x',e) D^{2}_{e}\tilde{W} \right) = -\tilde{c} + \text{tr} \left( \langle D_{y}a(se + x',e), e \rangle D^{2}_{e}\tilde{V}_{se} \right) \quad \text{in} \, \, \mathbb{T}^{d - 1}_{e}(0).
			\end{equation*}
		By Lemma \ref{L: comparison ergodic constants}, $\tilde{c}$ does not depend on the sequence $(h_{n})_{n \in \mathbb{N}}$, and it follows from the normalization $\tilde{W}(0) = \varphi'(s)$ that $\tilde{W}$ is also unique.  In particular, there is a function $\tilde{W}_{se}$ such that, for each $s \in [0,r_{e})$,
			\begin{equation*}
				h^{-1}(\tilde{V}_{(s + h)e} - \tilde{V}_{se}) \to \tilde{W}_{se} \quad \text{uniformly in} \, \, \mathbb{T}^{d-1}_{e}(0) \, \, \text{as} \, \, h \to 0^{+}.
			\end{equation*}   Note that $\tilde{W}_{se}$ is the unique viscosity solution of
			\begin{equation*}
				-\langle Df(se + x'),e \rangle - \text{tr} \left( a(se + x') D^{2}_{e} \tilde{W}_{se} \right) = -\langle Df_{e}^{\perp}(se),e \rangle + \text{tr} \left( \langle D_{y}a(se + x'), e \rangle D^{2}_{e}\tilde{V}_{se} \right) \quad \text{in} \, \, \mathbb{T}^{d-1}_{e}(0).
			\end{equation*}
			
		Note that above we considered one-sided derivatives.  It is straightforward to see that, for $s \in (0,r_{e})$, the restriction $h > 0$ was unnecessary.  To get a geniune (two-sided) derivative at $0$ or $r_{e}$, however, we need to add a compatibility condition to $\varphi$.  We require the following one:
			\begin{equation} \label{E: first derivative compatibility condition}
				\varphi'(r_{e}) = \tilde{W}_{0}(r_{e} e).
			\end{equation}
		We leave it to the reader to verify that \eqref{E: first derivative compatibility condition} implies that $s \mapsto \tilde{V}_{se}$ is differentiable at each $s \in [0,r_{e}]$.  In fact, it implies that $s \mapsto \tilde{V}_{se}$ is differentiable in $\mathbb{R}$ with
			\begin{equation*}
				\lim_{h \to 0} \frac{\tilde{V}_{(s + h)e} - \tilde{V}_{se}}{h} = \tilde{W}_{(s - Nr_{e})e}(\cdot + Nr_{e} e) \quad \text{uniformly in} \, \, \mathbb{T}^{d} \quad \text{if} \, \, 0 \leq s - Nr_{e} < r_{e}.
			\end{equation*}
		Furthermore, $\|\tilde{W}_{se}\|_{L^{\infty}(\mathbb{T}^{d - 1}_{e}(0))} \leq B$ independent of $s \in \mathbb{R}$.   
		
		Finally, notice that if we define $V_{e} : \mathbb{T}^{d} \to \mathbb{R}$ by
			\begin{equation*}
				V_{e}(x) = \tilde{V}_{\langle x,e \rangle e}(x - \langle x,e \rangle e)
			\end{equation*}
		then $V_{e}$ solves \eqref{E: exotic cell problem} and 
			\begin{equation*}
				\langle DV_{e}(x), e \rangle = \tilde{W}_{\langle x,e \rangle e}(x - \langle x,e \rangle e) \quad \text{if} \, \, x \in \mathbb{T}^{d}.
			\end{equation*}
		In particular, $V_{e} \in C^{1}(\mathbb{T}^{d})$.  
		
		\textbf{Step 4: Second derivative in $e$ direction}
			Since $\tilde{W}_{se}(0) = \varphi'(s)$ for all $s \in [0,r_{e})$ by construction, by imposing compatibility conditions on $\varphi''$, we can repeat the previous arguments to show that $V_{e} \in C^{2,\alpha}(\mathbb{T}^{d})$.  (The H\"{o}lder continuity of $D^{2}f_{e}^{\perp}$ follows by arguing as in the proof that it is Lipschitz; from this, the linearity of the equation can be used to show $\langle D^{2}V_{e} e, e \rangle$ is H\"{o}lder continuous with respect to the $e$ variable.)
\end{proof}

\section{Homogenization in Irrational Directions}    \label{S: irrational directions}

In this section, we undertake the main step in the proofs of Theorems \ref{T: level set PDE} and \ref{T: quasilinear}.  We prove that the solutions $(u^{\epsilon})_{\epsilon > 0}$ are described in the limit $\epsilon \to 0^{+}$ by the homogenized equation at any contact point where the level set of the normal vector is irrational.  It turns out that once this is proved, the remaining viscosity inequalities follow directly.  

More precisely, in this section, we show that if $\bar{u}^{*} = \limsup^{*} u^{\epsilon}$ and $\bar{u}_{*} = \liminf_{*} u^{\epsilon}$, then $\bar{u}^{*}$ and $\bar{u}_{*}$ are respectively sub- and supersolutions of \eqref{E: effective operator in dimension 2} or \eqref{E: effective equation} depending on the context.

\subsection{Solutions in Irrational Directions}  Due to the difficulty analyzing \eqref{E: level set PDE} at contact points with rational normals, we are led to consider an a priori weaker notion of viscosity solution.  In what follows, we are interested in functions $\tilde{u}$ that satisfy one or more of the viscosity inequalities
	\begin{equation} \label{E: effective sub or supersolution}
		\tilde{u}_{t} - \overline{F}^{*}(D\tilde{u}, D^{2} \tilde{u}) \leq 0, \quad \tilde{u}_{t} - \overline{F}_{*}(D\tilde{u},D^{2}\tilde{u}) \geq 0 \quad \text{in} \, \, \mathbb{R}^{d} \times (0,\infty).
	\end{equation}
Here, as usual in the theory of viscosity solutions, the operators $\overline{F}^{*},\overline{F}_{*} : \mathbb{R}^{d} \times \mathcal{S}_{d} \to \mathbb{R}$ are obtained from the effective operator $\overline{F}$ of \eqref{E: homogenized operator} by the formulas
	\begin{align*}
		\overline{F}^{*}(p,X) &= \lim_{\delta \to 0^{+}} \sup \left\{ \overline{F}(p',X') \, \mid \, (p',X') \in (\mathbb{R}^{d} \setminus \mathbb{R} \mathbb{Z}^{d}) \times \mathcal{S}_{d}, \, \, \|p' - p\| + \|X' - X\| < \delta \right\}, \\
		\overline{F}_{*}(p,X) &= \lim_{\delta \to 0^{+}} \inf \left\{ \overline{F}(p',X') \, \mid \, (p',X') \in (\mathbb{R}^{d} \setminus \mathbb{R} \mathbb{Z}^{d}) \times \mathcal{S}_{d}, \, \, \|p' - p\| + \|X' - X\| < \delta \right\}.
	\end{align*}
	
The difficulties encountered in rational directions motivate the following definition:

\begin{definition} \label{D: irrational directions} Given an open set $U \subseteq \mathbb{R}^{d} \times (0,\infty)$, we say that a locally bounded, upper semi-continuous function $\tilde{u} : \mathbb{R}^{d} \times (0,\infty) \to \mathbb{R}$ is a \emph{subsolution of \eqref{E: effective equation} in irrational directions in $U$} if there is a $K > 0$ such that, given a smooth function $\varphi : \mathbb{R}^{d} \times (0,\infty) \to \mathbb{R}$ such that $\tilde{u} - \varphi$ has a strict local maximum at $(x_{0},t_{0}) \in U$, the following conditions are met:
	\begin{itemize}
		\item[(a)] If $D\varphi(x_{0},t_{0}) \in \mathbb{R}^{d} \setminus \mathbb{R} \mathbb{Z}^{d}$, then 
			\begin{equation*}
				\varphi_{t}(x_{0},t_{0}) - \overline{F}^{*}(D\varphi(x_{0},t_{0}),D^{2}\varphi(x_{0},t_{0})) \leq 0.
			\end{equation*}
		\item[(b)] If $D\varphi(x_{0},t_{0}) \in \mathbb{R} \mathbb{Z}^{d} \setminus \{0\}$ and 
			\begin{equation*}
				\left\| \left(\text{Id} - \widehat{D\varphi}(x_{0},t_{0}) \otimes \widehat{D\varphi}(x_{0},t_{0})\right) D^{2} \varphi(x_{0},t_{0}) \right\| \leq \delta \|D\varphi(x_{0},t_{0})\|,
			\end{equation*}
		then 
			\begin{equation*}
				\varphi_{t}(x_{0},t_{0}) \leq K \delta \|D\varphi(x_{0},t_{0})\|.
			\end{equation*}
		\item[(c)] If $\|D\varphi(x_{0},t_{0})\| = \|D^{2}\varphi(x_{0},t_{0})\| = 0$, then
			\begin{equation*}
				\varphi_{t}(x_{0},t_{0}) \leq 0.
			\end{equation*}
	\end{itemize}
	
Similarly, a locally bounded, lower semi-continuous function $\tilde{v} : \mathbb{R}^{d} \times (0,\infty) \to \mathbb{R}$ is a \emph{supersolution of \eqref{E: effective equation} in irrational directions in $U$} if $-\tilde{v}$ satisfies the definition of subsolution in irrational directions with $\overline{F}_{*}$ replacing $\overline{F}^{*}$ in (a).  \end{definition}

Combining this definition with the definition of viscosity solution of Barles and Georgelin \cite{barles georgelin}, we can prove that a sub- or supersolution of \eqref{E: effective equation} in irrational directions is actually a sub- or supersolution in the usual sense.  Let us first recall the definition of \cite{barles georgelin}:

\begin{definition} \label{D: barles georgelin}  Given an open set $U \subseteq \mathbb{R}^{d} \times (0,\infty)$, we say that a locally bounded, upper semi-continuous function $\tilde{u} : \mathbb{R}^{d} \times (0,\infty) \to \mathbb{R}$ is a viscosity subsolution of \eqref{E: effective equation} in $U$ if, for each smooth function $\varphi : \mathbb{R}^{d} \times (0,\infty) \to \mathbb{R}$ and each point $(x_{0},t_{0}) \in U$ such that $\tilde{u} - \varphi$ has a strict local maximum at $(x_{0},t_{0})$, the following inequalities hold:
	\begin{itemize}
		\item[(i)] If $D\varphi(x_{0},t_{0}) \neq 0$, then
			\begin{equation*}
				\varphi_{t}(x_{0},t_{0}) - \overline{F}^{*}(D\varphi(x_{0},t_{0}), D^{2} \varphi(x_{0},t_{0})) \leq 0.
			\end{equation*}
		\item[(ii)] If $\|D\varphi(x_{0},t_{0})\| = \|D^{2} \varphi(x_{0},t_{0})\| = 0$, then
			\begin{equation*}
				\varphi_{t}(x_{0},t_{0}) \leq 0.
			\end{equation*}
	\end{itemize}

Analogously, we say that a locally bounded, lower semi-continuous function $\tilde{v} : \mathbb{R}^{d} \times (0,\infty) \to \mathbb{R}$ is a viscosity supersolution of \eqref{E: effective equation} in $U$ if $-\tilde{v}$ is a viscosity subsolution in $U$ with $\overline{F}_{*}$ replacing $\overline{F}^{*}$.  A locally bounded, continuous function in $\mathbb{R}^{d} \times (0,\infty)$ is a viscosity solution of \eqref{E: effective equation} if it is both a sub- and supersolution.  \end{definition}

The fact that this definition is equivalent to the usual one is proved arguing exactly as in \cite[Proposition 2.2]{barles georgelin}.

Using Definition \ref{D: barles georgelin} and the comparison machinery for viscosity solutions of second order parabolic equations, we can prove the following:  

	\begin{theorem} \label{T: irrational directions theorem}  Given $U \subseteq \mathbb{R}^{d} \times (0,\infty)$, if $\tilde{u} : \mathbb{R}^{d} \times (0,\infty) \to \mathbb{R}$ is a subsolution (resp.\ supersolution) of \eqref{E: effective equation} in irrational directions in $U$, then it is a viscosity subsolution (resp.\ supersolution) in $U$.  \end{theorem}  
	
		\begin{proof}  Modulo cosmetic changes, this is precisely \cite[Theorem 2]{allen cahn mobility}.  \end{proof}  
		
\subsection{Homogenization}  In view of the previous theorem, our first step in the direction of Theorems \ref{T: level set PDE} and \ref{T: quasilinear} is to show that the half-relaxed limits $\bar{u}^{*}$ and $\bar{u}_{*}$ are sub- and supersolution, respectively, of \eqref{E: effective sub or supersolution} in irrational directions.  In particular, this eliminates any non-trivial analysis of the behavior of those functions at contact points where the normal is rational.  

At contact points where the normal vector is irrational, we will use the approximate correctors of the previous section and Evans's perturbed test function method \cite{evans}.

Where rational normals are concerned, we only have to check that conditions (b) and (c) are satisfied.  As we will see, this follows directly from the assumptions on the operator $F$.  

The preceding discussion leads to the following result:

	\begin{prop} \label{P: homogenization irrational directions} If $(u^{\epsilon})_{\epsilon > 0}$ are the functions in Theorem \ref{T: level set PDE} or \ref{T: quasilinear} and $\bar{u}^{*} = \limsup^{*} u^{\epsilon}$ and $\bar{u}_{*} = \liminf_{*} u^{\epsilon}$ are the half-relaxed limits defined by 
		\begin{align*}
			\bar{u}^{*}(x,t) &= \lim_{\delta \to 0^{+}} \sup \left\{ u^{\epsilon}(y,s) \, \mid \, \epsilon + \|y - x\| + |s - t| \leq \delta \right\}, \\
			\bar{u}_{*}(x,t) &= \lim_{\delta \to 0^{+}} \inf \left\{ u^{\epsilon}(y,s) \, \mid \,  \epsilon + \|y - x\| + |s - t| \leq \delta \right\},
		\end{align*}
	then $\bar{u}^{*}$ is a subsolution of \eqref{E: effective sub or supersolution} in irrational directions in $\mathbb{R}^{d} \times (0,\infty)$ and $\bar{u}_{*}$ is a supersolution of \eqref{E: effective sub or supersolution} in irrational directions in $\mathbb{R}^{d} \times (0,\infty)$. \end{prop}  
	
		\begin{proof}  We break the proof down into three steps, one corresponding to each condition in Definition \ref{D: irrational directions}.  Since the argument for $\bar{u}^{*}$ is analogous, we restrict attention to $\bar{u}_{*}$.  
		
		Assume that $\varphi : \mathbb{R}^{d} \times (0,\infty) \to \mathbb{R}$ is a smooth function and $(x_{0},t_{0}) \in \mathbb{R}^{d} \times (0,\infty)$ is a point where $\bar{u}_{*} - \varphi$ attains a strict local minimum.
		
		\textbf{Step 1: Irrational Normal}  
		
		Assume that $D\varphi(x_{0},t_{0}) \notin \mathbb{R} \mathbb{Z}^{d}$.  Let us argue by contradiction.  Assume that there is a $\nu > 0$ such that
	\begin{equation} \label{E: contradict homogenization}
		\varphi_{t}(x_{0},t_{0}) - \overline{F}(D\varphi(x_{0},t_{0}),D^{2}\varphi(x_{0},t_{0})) \leq -2\nu \|D\varphi(x_{0},t_{0})\|.
	\end{equation}
	
Define $e = \widehat{D\varphi}(x_{0},t_{0})$ and $X = \|D\varphi(x_{0},t_{0})\|^{-1} D^{2}\varphi(x_{0},t_{0})$.  By Theorem \ref{T: approximate correctors}, we can let $V^{\nu} \in C^{2}(\mathbb{T}^{d})$ be an approximate corrector satisfying
	\begin{equation*}
		- \nu \leq \overline{F}(e,X) - F(e,\tilde{X}_{e} + D^{2}_{e} V^{\nu},y) \leq \nu \quad \text{in} \, \, \mathbb{T}^{d}.
	\end{equation*}
	
For each $\epsilon > 0$, define $\varphi_{\epsilon}$ by 
	\begin{equation*}
		\varphi_{\epsilon}(x,t) = \varphi(x,t) + \epsilon^{2} \|D\varphi(x_{0},t_{0})\| V^{\nu}(\epsilon^{-1} x).
	\end{equation*}
Let $(x_{\epsilon},t_{\epsilon})$ be a local minimum of $u^{\epsilon} - \varphi_{\epsilon}$ close to $(x_{0},t_{0})$.  Since $(x_{0},t_{0})$ is a strict minimum, we can fix a sequence $(\epsilon_{n})_{n \in \mathbb{N}}$ such that 
	\begin{equation*}
		\lim_{n \to \infty} \epsilon_{n} = 0, \quad \lim_{n \to \infty} (x_{\epsilon_{n}},t_{\epsilon_{n}}) = (x_{0},t_{0}), \quad \inf \left\{ \|D\varphi(x_{\epsilon_{n}},t_{\epsilon_{n}})\| \, \mid \, n \in \mathbb{N} \right\} > 0.
	\end{equation*}
Invoking the equation satisfied by $u^{\epsilon_{n}}$ and using the continuity of $F$, we find, for sufficiently large $n$,
	\begin{equation*}
		\varphi_{t}(x_{\epsilon_{n}},t_{\epsilon_{n}}) - F(D\varphi(x_{\epsilon_{n}},t_{\epsilon_{n}}), D^{2}\varphi(x_{\epsilon_{n}},t_{\epsilon_{n}}) + \|D\varphi(x_{0},t_{0})\| D^{2}V^{\nu}(\epsilon_{n}^{-1} x_{\epsilon_{n}}), \epsilon_{n}^{-1} x_{\epsilon_{n}}) \geq o(1).
	\end{equation*}
The continuity and geometric properties of $F$ then yield
	\begin{align*}
		0 &\leq \varphi_{t}(x_{\epsilon_{n}},t_{\epsilon_{n}}) - F(D\varphi(x_{0},t_{0}), D^{2}_{e}\varphi(x_{\epsilon_{n}},t_{\epsilon_{n}}) + \|D\varphi(x_{0},t_{0})\| D^{2}_{e}V^{\nu}(\epsilon_{n}^{-1} x_{\epsilon_{n}}),\epsilon_{n}^{-1} x_{\epsilon_{n}}) + o(1) \\
			&\leq \varphi_{t}(x_{\epsilon_{n}},t_{\epsilon_{n}}) - \|D\varphi(x_{0},t_{0})\| F(e, \tilde{X}_{e} + D^{2}_{e} V^{\nu}(\epsilon_{n}^{-1} x_{\epsilon_{n}}), \epsilon_{n}^{-1} x_{\epsilon_{n}}) + o(1) \\
			&\leq \varphi_{t}(x_{0},t_{0}) - \|D\varphi(x_{0},t_{0})\| \overline{F}(e,\tilde{X}_{e}) + \nu \|D\varphi(x_{0},t_{0})\| + o(1).
	\end{align*}
Sending $\epsilon \to 0^{+}$ and recalling the definition of $\overline{F}$ in \eqref{E: homogenized operator}, we contradict \eqref{E: contradict homogenization}.  Since $\varphi$ was arbitrary, we conclude that $\bar{u}_{*}$ satisfies condition (a) of Definition \ref{D: irrational directions}.

	\textbf{Step 2: Rational Normal, Flat Level Set}

	Next, we assume that $D\varphi(x_{0},t_{0}) \neq 0$ and define $\delta$ by
	\begin{equation*}
		\delta = \|D\varphi(x_{0},t_{0})\|^{-1} \left\| \left( \text{Id} - \widehat{D\varphi}(x_{0},t_{0}) \otimes \widehat{D\varphi}(x_{0},t_{0}) \right) D^{2} \varphi(x_{0},t_{0}) \right\|.
	\end{equation*}
	
Using the equation directly and invoking assumptions \eqref{E: stationary planes} and \eqref{A: strong degenerate ellipticity} or \eqref{E: structure condition} and \eqref{A: a continuous}, we find $\epsilon_{n} \to 0^{+}$ and $(x_{\epsilon_{n}},t_{\epsilon_{n}}) \to (x_{0},t_{0})$ so that
	\begin{align*}
		0 &\leq \varphi_{t}(x_{\epsilon_{n}},t_{\epsilon_{n}}) - F(D\varphi(x_{\epsilon_{n}},t_{\epsilon_{n}}), D^{2}\varphi(x_{\epsilon_{n}},t_{\epsilon_{n}}), \epsilon^{-1}_{n} x_{\epsilon_{n}}), \\
			&\leq \varphi_{t}(x_{\epsilon_{n}},t_{\epsilon_{n}}) + \Lambda \|D^{2}_{e_{n}}\varphi(x_{\epsilon_{n}},t_{\epsilon_{n}})\|,
	\end{align*}
where $e_{n} = \widehat{D\varphi}(x_{\epsilon_{n}},t_{\epsilon_{n}})$.  Sending $n \to \infty$, we recover
	\begin{equation*}
		\varphi_{t}(x_{0},t_{0}) \geq -\Lambda \|D^{2}_{e}\varphi(x_{0},t_{0})\| \geq -\Lambda \delta \|D\varphi(x_{0},t_{0})\|.
	\end{equation*}
We conclude that $\bar{u}_{*}$ satisfies condition (b) in Definition \ref{D: irrational directions} with $K = \Lambda$.  

\textbf{Step 3: Vanishing Normal}

Finally, if $\|D\varphi(x_{0},t_{0})\| = \|D^{2} \varphi(x_{0},t_{0})\| = 0$, then we can find $\epsilon_{n} \to 0^{+}$ and $(x_{\epsilon_{n}},t_{\epsilon_{n}}) \to (x_{0},t_{0})$ such that
	\begin{align*}
		0 &\leq \varphi_{t}(x_{\epsilon_{n}},t_{\epsilon_{n}}) - F_{*}(D\varphi(x_{\epsilon_{n}},t_{\epsilon_{n}}), D^{2} \varphi(x_{\epsilon_{n}},t_{\epsilon_{n}}),\epsilon_{n}^{-1} x_{n}).
	\end{align*}
In the limit $n \to \infty$, we use \eqref{A: strong degenerate ellipticity} to recover $\varphi_{t}(x_{0},t_{0}) \geq 0.$
This proves $\bar{u}_{*}$ satisfies condition (c) in Definition \ref{D: irrational directions}.\end{proof}

\section{Continuity and Discontinuity of Homogenized Coefficients} \label{S: discontinuity}

In this section, we study the continuity properties of the homogenized operator $\overline{F}$ of Theorem \ref{T: level set PDE} and the effective coefficients $\overline{m}$ and $\bar{a}$ of Theorem \ref{T: quasilinear}.  The main results are Theorem \ref{T: limiting measures} and Corollary \ref{C: generic} concerning the (disc)continuity properties of $\overline{m}$ and $\overline{a}$ in the quasi-linear setting.  

In dimension two, it turns out that the homogenized operator obtained in Theorem \ref{T: level set PDE} is always continuous, a result that is stated next.  

\begin{prop} \label{P: continuity in dimension 2}  Under the assumptions of Theorem \ref{T: level set PDE}, the homogenized operator $\overline{F} : (\mathbb{R}^{2} \setminus \mathbb{R} \mathbb{Z}^{2}) \times \mathcal{S}_{2} \to \mathbb{R}$ (see \eqref{E: homogenized operator}) extends continuously to $(\mathbb{R}^{2} \setminus \{0\}) \times \mathcal{S}_{2}$.  \end{prop}

For the sake of completeness, we show how Proposition \ref{P: continuity in dimension 2} completes the proof of Theorem \ref{T: level set PDE}:

\begin{proof}[Proof of Theorem \ref{T: level set PDE}]  Let $\bar{u}^{*}$ and $\bar{u}_{*}$ be the half-relaxed limits defined in Proposition \ref{P: homogenization irrational directions}.  That proposition shows that $\bar{u}^{*}$ and $\bar{u}_{*}$ are respectively sub- and supersolution of \eqref{E: effective equation} in $\mathbb{R}^{d} \times (0,\infty)$.  Furthermore, using \eqref{A: strong degenerate ellipticity} and the assumption $u_{0} \in UC(\mathbb{R}^{d})$, we can build sub- and supersolutions  that show $\bar{u}^{*} = \bar{u}_{*} = u_{0}$ on $\mathbb{R}^{d} \times \{0\}$.  Since $\overline{F}^{*} = \overline{F}_{*}$ in $(\mathbb{R}^{2} \setminus \{0\}) \times \mathcal{S}_{2}$, the comparison principle implies $\bar{u}^{*} \leq \bar{u}_{*}$  (cf.\ Appendix \ref{A: technical lemmata}).  Therefore, $\bar{u}^{*} = \bar{u}_{*} = \bar{u}$ and we conclude that $u^{\epsilon} \to \bar{u}$ locally uniformly in $\mathbb{R}^{d} \times [0,\infty)$. \end{proof}

By contrast, in higher dimensions, the effective coefficients $\bar{a}$ and $\bar{m}$ are generically discontinuous at each rational direction on $S^{d-1}$.  Thus, the final step in the proof of Theorem \ref{T: quasilinear} will be deferred until Section \ref{S: comparison}, where we extend comparison to \eqref{E: effective equation}.  

\subsection{Strategy of Proof}  Before entering into the details, let us briefly give a heuristic explanation of the strategy of the proof and the core technical issues that arise.  The main idea of the proof and the discussion that follows are inspired by \cite{feldman kim}.

Suppose $e \in S^{d-1} \cap \mathbb{R} \mathbb{Z}^{d}$.  In the proof of Theorem \ref{T: limiting measures}, we proceed by studying the behavior of $(\bar{\mu}_{e_{n}})_{n \in \mathbb{N}}$ along a sequence of irrational directions $(e_{n})_{n \in \mathbb{N}}$ with $e_{n} \to e$ and $\frac{e_{n} - e}{\|e_{n} - e\|} \to -\eta$.

$\bar{\mu}_{e_{n}}$ captures the behavior of the diffusion $X^{e_{n}}$ of \eqref{E: SDE} after long times.  Therefore, it is natural to pass to the diffusions $X^{e_{n}}$ and consider their behavior.  Further, in light of the structure of $\mathscr{I}^{a}_{e}$, the only question is the $e$ marginal of any limit point of $(\bar{\mu}_{e_{n}})_{n \in \mathbb{N}}$, that is, we only need to study $f(\langle X^{e_{n}}_{t},e \rangle)$, where $f$ is an $r_{e}$-periodic function of one variable and $t > 0$ is large. 

We may as well assume that $\langle X^{e_{n}}_{0},e_{n} \rangle = 0$, which implies that $\langle X^{e_{n}}_{t},e_{n} \rangle = 0$ for all $t \geq 0$.  Therefore, if we write $e_{n} = \cos(\theta_{n}) e - \sin(\theta_{n}) \eta_{n}$ for some $\theta_{n} \in (-\pi,\pi]$ and $\eta_{n} \in S^{d-1} \cap \langle e \rangle^{\perp}$, we have
	\begin{equation*}
		\langle X^{e_{n}}_{t},e \rangle = \langle X^{e_{n}}_{t}, e - e_{n} \rangle = (1 - \cos(\theta_{n}) \langle X^{e_{n}}_{t}, e \rangle + O(\theta_{n}^{2}).
	\end{equation*} 
Hence to recover anything meaningful from $f(\langle X^{e_{n}}_{t},e \rangle)$, we need to wait until $\|X^{e_{n}}_{t}\| \approx \theta_{n}^{-1}$.  Since this takes a time proportional to $ \theta_{n}^{-2}$, we should study $f(\langle X^{e_{n}}_{\theta_{n}^{-2}T},e \rangle)$ in the simultaneous limit $n, T \to \infty$.

Put another way, at the level of the PDE, we would like to understand the behavior as $n \to \infty$ and $\delta \to 0^{+}$ of the penalized correctors $(V^{\delta}_{n})_{(\delta,n) \in (0,\infty) \times \mathbb{N}}$ solving
	\begin{equation*}
		\delta V^{\delta}_{n} - \text{tr} \left(a(\theta_{n}^{-1}y,e_{n}) D^{2}_{e} V^{\delta}_{n} \right) = f(\theta_{n}^{-1} \langle x,e \rangle) \quad \text{in} \, \, \langle e_{n} \rangle^{\perp}.
	\end{equation*}  
A simple homogenization argument shows that, for a fixed $\delta > 0$, $V^{\delta}_{n} \to \tilde{V}^{\delta}$ as $n \to \infty$, where $\tilde{V}^{\delta}$ is the solution of the problem
	\begin{equation*}
		\delta \tilde{V}^{\delta} - \text{tr} \left( a^{\perp}_{e}(\langle \eta, x \rangle e) D^{2}_{e} \tilde{V}^{\delta} \right) = f(\langle \eta, x \rangle e) \quad \text{in} \, \, \langle e \rangle^{\perp}.
	\end{equation*}
It turns out that extracting the limit of $\delta \tilde{V}^{\delta}$ as $\delta \to 0^{+}$ leads to the correct limit of $\int_{\mathbb{T}^{d}} f(\langle y,e \rangle) \, \bar{\mu}_{e_{n}}(dy)$.  The question is simply how to show that the limits $n \to \infty$ and $\delta \to 0^{+}$ commute.  

The argument below shows how to do this working at the level of the obstacle problems introduced in \cite{caffarelli souganidis wang} rather than the penalized correctors.  Working with penalized correctors is difficult since it requires understanding the rate at which $\delta V^{\delta}_{n}$ converges as $\delta \to 0^{+}$, independently of $n$.  We circumvent this by passing to the obstacle problem approach of \cite{caffarelli souganidis wang} and leveraging an upper semi-continuity property proved there.  In this way, it is possible to first send $n \to \infty$ and then $\delta \to 0^{+}$ without explicitly quantifying the rates of convergence of the associated almost periodic homogenization problems. 

\subsection{Preliminaries}  We start by introducing some notation that will be used later.

Henceforth let $F$ be the operator in \eqref{E: level set PDE} satisfying either the assumptions of Theorem \ref{T: level set PDE} or Theorem \ref{T: quasilinear} with $F(p,X,y) = m(y,\hat{p})^{-1} \text{tr} (a(y,\hat{p}) \tilde{X}_{\hat{p}})$.  It will be convenient to define the family of shifted operators $\{F^{x}\}_{x \in \mathbb{T}^{d}}$ by 
	\begin{equation*}
		F^{x}(p,X,y) = F(p,X,y+x).
	\end{equation*}
The reader familiar with stochastic homogenization can think of the shift $x$ like an element $\omega$ of a probability space $(\Omega,\mathbb{P})$.  In our case, $\Omega = \mathbb{T}^{d}$ and $\mathbb{P} = \mathcal{L}^{d}$, although we will also have occasion to work with the surface area measures on the sub-tori $\{\mathbb{T}^{d - 1}_{e}(0)\}_{e \in S^{d-1} \cap \mathbb{R} \mathbb{Z}^{d}}$.  

To prove Proposition \ref{P: continuity in dimension 2} and Theorem \ref{T: limiting measures}, we start with the following preliminary result:

	\begin{prop} \label{P: preliminary but main}  Fix $e \in S^{d-1} \cap \mathbb{R} \mathbb{Z}^{d}$ and $\eta \in S^{d-1} \cap \langle e \rangle^{\perp}$.  There is a function $\gamma_{*}^{e,\eta} : \mathcal{S}_{d} \to \mathbb{R}$ such that if $(e_{n})_{n \in \mathbb{N}} \subseteq S^{d-1}$ satisfies
		\begin{equation*}
			\lim_{n \to \infty}\left( \|e_{n} - e\| + \left\| \frac{e_{n} - e}{\|e_{n} - e\|} + \eta \right\| \right)= 0.
		\end{equation*}
	then
	\begin{equation*}
		\lim_{n \to \infty} \sup \left\{ |F^{\perp}_{e_{n}}(X,s) + \gamma_{*}^{e,\eta}(X)| \, \mid \, s \in \mathbb{R} \right\} = 0.
	\end{equation*}  
	\end{prop}  

The main thrust of the section is the proof of this result.  For the rest of the section, fix such an $e$, $\eta$, and $(e_{n})_{n \in \mathbb{N}}$

By assumption, we can fix $(\eta_{n})_{n \in \mathbb{N}} \subseteq S^{d-1}$ and $(\theta_{n})_{n \in \mathbb{N}} \subseteq (-\pi,\pi]$ such that
	\begin{equation*}
		e_{n} = \cos(\theta_{n}) e - \sin(\theta_{n}) \eta_{n}.
	\end{equation*}
We will assume without loss of generality that $(\theta_{n})_{n \in \mathbb{N}} \subseteq (-\pi,\pi) \setminus \{0\}$.  

For each $n \in \mathbb{N}$, let $O_{n} : \mathbb{R}^{d} \to \mathbb{R}^{d}$ be the orthogonal transformation such that
	\begin{equation*}
		O_{n}(e_{n}) = e, \quad O_{n}(\sin(\theta_{n})e + \cos(\theta_{n})\eta_{n}) = \eta_{n}, \quad O_{n} \restriction_{\langle e,e_{n} \rangle^{\perp}} = \text{Id} \restriction_{\langle e, e_{n} \rangle^{\perp}}
	\end{equation*}
Notice that $\lim_{n \to \infty} O_{n} = \text{Id}$ and $O_{n}(\langle e_{n} \rangle^{\perp}) = \langle e \rangle^{\perp}$.  

Finally, we let $\{v_{1},\dots,v_{d-1}\}$ be an orthonormal basis of $\langle e \rangle^{\perp}$ and define, for each $R > 0$, the cube $Q^{*}_{R} \subseteq \langle e \rangle^{\perp}$ by 
	\begin{equation*}
		Q^{*}_{R} = \left\{ y \in \langle e \rangle^{\perp} \, \mid \, |\langle y,v_{1} \rangle| \vee \dots \vee |\langle y, v_{d-1}\rangle| < R/2 \right\}.
	\end{equation*}
Given $n \in \mathbb{N}$, we define the analogous cubes $\{Q^{(n)}_{R} \, \mid \, R > 0\}$ in $\langle e_{n} \rangle^{\perp}$ by $Q^{(n)}_{R} = O_{n}^{-1}(Q^{*}_{R})$.

\subsection{Obstacle Problems}  It is technically very convenient to replace the penalized correctors $(V^{\delta}_{n})_{(\delta,n)}$ of the previous discussion by solutions of a related family of obstacle problems.  In this section, we describe the set-up.

Given $n \in \mathbb{N}$, $\gamma \in \mathbb{R}$, $R > 0$, $X \in \mathcal{S}_{d}$, and $x \in \mathbb{T}^{d}$, we set $\nu = (n,\gamma,R,X)$ and define the obstacle sub- and supersolutions $u^{\nu,x}$ and $u_{\nu,x}$ as the solutions of the equations
	\begin{align}
		\left\{ \begin{array}{r l} \label{E: obstacle problem subsolution}
			\max \left\{ - F^{x}(e_{n}, \tilde{X}_{e_{n}} + D^{2}_{e_{n}} u^{\nu,x},\theta_{n}^{-1}y) - \gamma, u^{\nu,x} \right\} = 0 & \text{in} \, \, Q^{(n)}_{R}, \\
			u^{\nu,x} = 0 & \text{on} \, \, \partial Q^{(n)}_{R},
		\end{array} \right. \\
		\left\{ \begin{array}{r l} \label{E: obstacle problem supersolution}
			\min \left\{ - F^{x}(e_{n}, \tilde{X}_{e_{n}} + D^{2}_{e_{n}} u_{\nu,x},\theta_{n}^{-1}y) - \gamma, u_{\nu,x} \right\} = 0 & \text{in} \, \, Q^{(n)}_{R}, \\
			u_{\nu,x} = 0 & \text{on} \, \, \partial Q^{(n)}_{R}.
		\end{array} \right.
	\end{align}
Existence and uniqueness for these problems can be proved using a comparison principle or through a penalization procedure as in \cite{caffarelli souganidis wang}.  

When $\nu = (*,\gamma,R,X)$, $u^{\nu,x}$ and $u_{\nu,x}$ are the solutions of the homogenized problems
	\begin{align}
		\left\{ \begin{array}{r l} \label{E: obstacle problem subsolution homogenized}
			\max \left\{ - F_{e}^{\perp}(\tilde{X}_{e} + D^{2}_{e} u^{\nu,x}, \langle y, \eta \rangle e + \langle x, e \rangle e) - \gamma, u^{\nu,x} \right\} = 0 & \text{in} \, \, Q^{*}_{R}, \\
			u^{\nu,x} = 0 & \text{on} \, \, \partial Q^{*}_{R},
		\end{array} \right. \\
		\left\{ \begin{array}{r l} \label{E: obstacle problem supersolution homogenized}
			\min \left\{ - F_{e}^{\perp}(\tilde{X}_{e} + D^{2}_{e} u^{\nu,x}, \langle y, \eta \rangle e + \langle x, e \rangle e) - \gamma, u^{\nu,x} \right\} = 0 & \text{in} \, \, Q^{*}_{R}, \\
			u^{\nu,x} = 0 & \text{on} \, \, \partial Q^{*}_{R}.
		\end{array} \right.
	\end{align}
For these equations, well-posedness is clear in the context of Theorem \ref{T: level set PDE} since $\langle e \rangle^{\perp}$ is one-dimensional in this case.  Where Theorem \ref{T: quasilinear} is concerned, the regularity of $a^{\perp}_{e}$ and $m^{\perp}_{e}$ proved in Proposition \ref{P: rational correctors} is more than enough to guarantee well-posedness.

Next, for $n \in \mathbb{N} \cup \{*\}$, we define the contact sets $K^{\nu,x}$ and $K_{\nu,x}$ by 
	\begin{equation*}
		K^{\nu,x} = \{y \in Q^{(n)}_{R} \, \mid \, u^{\nu,x}(y) = 0\}, \quad K_{\nu,x} = \{y \in Q^{(n)}_{R} \, \mid \, u_{\nu,x}(y) = 0\}.
	\end{equation*}
	
By the results of \cite{caffarelli souganidis wang}, for each $n \in \mathbb{N} \cup \{*\}$, there are functions $\bar{\ell}^{\perp}_{n}, \underline{\ell}^{\perp}_{n} : \mathbb{R} \times \mathcal{S}_{d} \times \mathbb{T}^{d} \to [0,1]$ such that, for each $(\gamma,X) \in \mathbb{R} \times \mathcal{S}_{d}$, $\nu(R) = (n,\gamma,R,X)$, and $x \in \mathbb{T}^{d}$, we have
	\begin{align*}
		\bar{\ell}^{\perp}_{n}(\gamma,X,x) &= \lim_{R \to \infty} R^{1 -d} \mathcal{H}^{d-1}(K^{\nu_{n}(R),x}) \\
		\underline{\ell}^{\perp}_{n}(\gamma,X,x) &= \lim_{R \to \infty} R^{1-d} \mathcal{H}^{d-1}(K_{\nu_{n}(R),x}).
	\end{align*}
Notice that, by analogy with Theorem \ref{T: approximate correctors}, for each $n \in \mathbb{N}$, $\bar{\ell}^{\perp}_{n}(\gamma,X,\cdot)$ and $\underline{\ell}^{\perp}_{n}(\gamma,X,\cdot)$ vary only in the $e_{n}$ direction.

Next, note that these functions satisfy the following variational principles: if $e_{n} \in S^{d-1} \setminus \mathbb{R} \mathbb{Z}^{d}$, then the unique ergodicity of the associated group of translations (cf.\ \cite[Theorem 11]{pulsating einstein}) implies that, for each $x \in \mathbb{T}^{d}$,
	\begin{align}
		\bar{\ell}^{\perp}_{n}(\gamma,X,x) &= \inf \left\{ R^{1-d} \int_{\mathbb{T}^{d}} \mathcal{H}^{d-1}(K^{\nu_{n}(R),y}) \, dy \, \mid \, R > 0 \right\}, \label{E: variational principle subsolutions} \\
		\underline{\ell}_{n}(\gamma,X,x) &= \inf \left\{ R^{1-d} \int_{\mathbb{T}^{d}} \mathcal{H}^{d-1}(K_{\nu_{n}(R),y}) \, dy \, \mid \, R > 0 \right\}, \nonumber
	\end{align}
while the case $e_{n} \in \mathbb{R} \mathbb{Z}^{d}$ (cf.\ \cite[Proposition 53]{pulsating einstein}), for each $s \in [0,r_{e_{n}})$,
	\begin{align} 
		\overline{\ell}_{n}(\gamma,X, s e_{n}) &= \inf \left\{ R^{1-d} \fint_{\mathbb{T}^{d-1}_{e_{n}}(s)} \mathcal{H}^{d-1}(K^{\nu_{n}(R),\xi}) \, \mathcal{H}^{d-1}(d \xi) \, \mid \, R > 0 \right\}, \label{E: variational principle subsolutions rational} \\
		\underline{\ell}_{n}(\gamma,X, s e_{n}) &= \inf \left\{ R^{1-d} \fint_{\mathbb{T}^{d-1}_{e_{n}}(s)} \mathcal{H}^{d-1}(K_{\nu_{n}(R),\xi}) \, \mathcal{H}^{d-1}(d\xi) \, \mid \, R > 0 \right\}. \nonumber
	\end{align}
All of this follows from the sub-additive ergodic theorem (cf.\ \cite[Appendix A]{caffarelli souganidis wang}).

Finally, notice that, when $n = *$, we can argue as in the rational case.  Since the coefficients of the associated obstacle problems vary only in the $\eta$ direction, it follows that the integrals over $\mathbb{T}^{d-1}_{e}(s)$ so obtained do not depend on $s$.  Therefore, as in the irrational case, we find, for each $x \in \mathbb{T}^{d}$,
	\begin{align}
		\bar{\ell}^{\perp}_{*}(\gamma,X,x) &= \inf \left\{ R^{1-d} \int_{\mathbb{T}^{d}} \mathcal{H}^{d-1}(K^{\nu_{*}(R),y}) \, dy \, \mid \, R > 0 \right\}, \label{E: variational principle subsolutions *} \\
		\underline{\ell}^{\perp}_{*}(\gamma,X,x) &= \inf \left\{ R^{1-d} \int_{\mathbb{T}^{d}} \mathcal{H}^{d-1}(K_{\nu_{*}(R),y}) \, dy \, \mid \, R > 0 \right\}. \nonumber
	\end{align}	
\subsection{Homogenization of Obstacle Problems}  By a direct analogy to the discussion of the penalized correctors above, we now show that the obstacle problems \eqref{E: obstacle problem subsolution} and \eqref{E: obstacle problem supersolution} homogenize to the problems \eqref{E: obstacle problem subsolution homogenized} and \eqref{E: obstacle problem supersolution homogenized} as $n \to \infty$.  

	\begin{prop} \label{P: homogenization obstacle problem} Given $(\gamma,X,R) \in \mathbb{R} \times \mathcal{S}_{d} \times (0,\infty)$ and $x \in \mathbb{T}^{d}$, if $\nu_{n} = (n,\gamma,X,R)$ and $\nu_{*} = (*,\gamma,X,R)$, and if $(x_{n})_{n \in \mathbb{N}} \subseteq \mathbb{T}^{d}$ is such that $x= \lim_{n \to \infty} x_{n}$, then  
		\begin{align*}
			0 = \lim_{\delta \to 0^{+}} \sup \left\{ |u^{\nu_{n},x_{n}}(y_{1}) - u^{\nu_{*},x}(y_{2})| \, \mid \, (y_{1},y_{2}) \in Q^{(n)}_{R} \times Q^{*}_{R}, \, \, n^{-1} + \|y_{1} - y_{2}\| < \delta \right\}, \\
			0 = \lim_{\delta \to 0^{+}} \sup \left\{ |u_{\nu_{n},x_{n}}(y_{1}) - u_{\nu_{*},x}(y_{2})| \, \mid \, (y_{1},y_{2}) \in Q^{(n)}_{R} \times Q^{*}_{R}, \, \, n^{-1} + \|y_{1} - y_{2}\| < \delta \right\}.
		\end{align*}
	\end{prop}  
	
		\begin{proof}  To prove this, we will work with half-relaxed limits.  The proof for supersolutions follows by analogous arguments so we will restrict attention to subsolutions.
		
		Define $\bar{u}^{*}$ and $\bar{u}_{*}$ in $\overline{Q^{*}_{R}}$ by 
			\begin{align*}
				\bar{u}^{*}(y) &= \lim_{\delta \to 0^{+}} \sup \left\{ u^{\nu_{n},x_{n}}(y') \, \mid \, y' \in Q^{(n)}_{R}, \, \, n^{-1} + \|y' - y\| < \delta \right\}, \\
				\bar{u}_{*}(y) &= \lim_{\delta \to 0^{+}} \inf \left\{ u^{\nu_{n},x_{n}}(y') \, \mid \, y' \in Q^{(n)}_{R}, \, \, n^{-1} + \|y' - y\| < \delta \right\}.
			\end{align*}
		It suffices to show that $\bar{u}^{*} = \bar{u}_{*} = u^{\nu_{*},x}$.  To do this, we only need to prove that $\bar{u}^{*}$ and $\bar{u}_{*}$ are sub- and supersolution of \eqref{E: obstacle problem subsolution homogenized} and apply the comparison principle.  Since the proof for $\bar{u}^{*}$ is almost identical, we will restrict attention to $\bar{u}_{*}$.  
		
		Assume that $\varphi : \mathbb{R}^{d} \to \mathbb{R}$ is a smooth function and $y_{0} \in Q^{*}_{R}$ is a point where $\bar{u}_{*} - \varphi$ has a strict local minimum.  We claim that
			\begin{equation} \label{E: need to prove homogenization}
				\max \left\{ - F^{\perp}_{e}(\tilde{X}_{e} + D^{2}_{e} \varphi(y_{0}), - \langle y_{0}, \eta \rangle e + \langle x,e \rangle e) - \gamma, \bar{u}_{*}(y_{0}) \right\} \geq 0.
			\end{equation}
		Clearly, we can assume that $\bar{u}_{*}(y_{0}) < 0$.  
		
		Let us argue by contradiction.  If \eqref{E: need to prove homogenization} fails to hold, then there is a $\zeta > 0$ such that
			\begin{equation} \label{E: homogenization by contradiction obstacle}
				- F^{\perp}_{e}(\tilde{X}_{e} + D^{2}_{e} \varphi(y_{0}), \langle y_{0}, \eta \rangle e + \langle x,e \rangle e) - \gamma < - \zeta.
			\end{equation}
		
		Let $V \in C^{2}(\mathbb{T}^{d})$ be an approximate corrector satisfying the equation
			\begin{equation*}
				- \zeta \leq F^{\perp}_{e}(\tilde{X}_{e} + D^{2}_{e} \varphi(y_{0}), \langle y,e \rangle e + \langle x,e \rangle e) - F^{x}(e,\tilde{X}_{e} + D^{2}_{e} \varphi(y_{0}) + D^{2}_{e} V, y) \leq \zeta \quad \text{in} \, \, \mathbb{T}^{d}.
			\end{equation*}
		For each $n \in \mathbb{N}$, let $y_{n} \in \overline{Q^{(n)}_{R}}$ be a point where the function $y \mapsto u^{\nu_{n},x_{n}}(y) - \varphi(y) - \theta_{n}^{2} V(\theta_{n}^{-1} y)$ attains its minimum in $\overline{Q_{R}^{(n)}}$.  By classical arguments, we can pass to a subsequence $(y_{n_{j}})_{j \in \mathbb{N}}$ such that 
			\begin{equation*}
				\lim_{j \to \infty} y_{n_{j}} = y_{0}, \quad \lim_{j \to \infty} u^{\nu_{n_{j}},x_{n_{j}}}(y_{n_{j}}) = \bar{u}_{*}(y_{0}) < 0.
			\end{equation*}
		
		There is no loss of generality in assuming that both $y_{n_{j}} \in Q^{(n_{j})}_{R}$ and $u^{\nu_{n_{j}},x_{n_{j}}}(y_{n_{j}}) < 0$ for all $j \in \mathbb{N}$.  Thus, the equation satisfied by $u^{\nu_{n_{j}},x_{n_{j}}}$ gives
			\begin{align*}
				0 &\leq - F^{x_{n_{j}}}(e_{n_{j}}, \tilde{X}_{e_{n_{j}}} + D^{2}_{e_{n_{j}}} \varphi(y_{n_{j}}) + D^{2}_{e_{n_{j}}} V(\theta_{n_{j}}^{-1} y_{n_{j}}), \theta_{n_{j}}^{-1} y_{n_{j}}) - \gamma \\
				&= - F^{x}(e, \tilde{X}_{e} + D^{2}_{e} \varphi(y_{0}) + D^{2}_{e} V(\theta_{n_{j}}^{-1} y_{n_{j}}), \theta_{n_{j}}^{-1} y_{n_{j}}) - \gamma + o(1) \\
				&\leq - F^{\perp}_{e}(\tilde{X}_{e} + D^{2}_{e} \varphi(y_{0}), \theta_{n_{j}}^{-1} \langle y_{n_{j}}, e \rangle e + \langle x,e \rangle e) - \gamma + \zeta +  o(1).
			\end{align*}
		Since $y_{n_{j}} \in \langle e_{n_{j}} \rangle^{\perp}$, we have
			\begin{align*}
				\theta_{n_{j}}^{-1} \langle y_{n_{j}}, e \rangle &= \theta_{n_{j}}^{-1} \langle y_{n_{j}}, e - e_{n_{j}} \rangle = \left(\frac{1 - \cos(\theta_{n_{j}})}{\theta_{n_{j}}}\right) \langle e, y_{n_{j}} \rangle + \left(\frac{\sin(\theta_{n_{j}})}{\theta_{n_{j}}}\right) \langle \eta_{n_{j}}, y_{n_{j}} \rangle.
			\end{align*}
		Thus, $\theta_{n_{j}}^{-1} \langle y_{n_{j}}, e \rangle \to \langle y_{0}, \eta \rangle$ and we find, in the limit $n \to \infty$,
			\begin{equation*}
				-\zeta \leq - F^{\perp}_{e}(\tilde{X}_{e} + D^{2}_{e} \varphi(y_{0}), \langle y_{0}, \eta \rangle e + \langle x,e \rangle e) - \gamma.
			\end{equation*}
		This directly contradicts \eqref{E: homogenization by contradiction obstacle}. 
		
		We deduce that $\bar{u}_{*}$ is a supersolution of \eqref{E: obstacle problem subsolution homogenized} in the interior of $Q^{*}_{R}$.  Using barriers, it is not hard to show that $\bar{u}_{*} \geq 0$ in $\partial Q^{*}_{R}$.  Thus, $\bar{u}_{*}$ is a supersolution.  \end{proof}  

\subsection{Densities of Contact Sets}  Once we state some properties of the densities $\overline{\ell}_{n}^{\perp}$ and $\underline{\ell}_{n}^{\perp}$, we will have all the tools necessary to prove Proposition \ref{P: continuity in dimension 2} and Theorem \ref{T: limiting measures}.  

The following result follows by arguing as in \cite{caffarelli souganidis wang} (also see \cite{armstrong smart}).  In the statement, we write $e_{*} = e$ and assume we are in case (i).  

	\begin{prop} \label{P: densities} For each $(X,x) \in \mathcal{S}_{d} \times \mathbb{T}^{d}$, there is a sequence $(\gamma_{n}(X,x))_{n \in \mathbb{N} \cup \{*\}} \subseteq \mathbb{R}$ such that, for each $n \in \mathbb{N} \cup \{*\}$, the following statements hold:
		\begin{itemize}
			\item[(i)] $\bar{\ell}^{\perp}_{n}(\gamma,X,x) = 0$ if $\gamma < \gamma_{n}(X,x)$ and $\bar{\ell}^{\perp}_{n}(\gamma,X,x) \geq c |\gamma - \gamma_{n}|^{d - 1}$ if $\gamma > \gamma_{n}(X,x)$.
			\item[(ii)] $\underline{\ell}^{\perp}_{n}(\gamma,X,x) = 0$ if $\gamma > \gamma_{n}(X,x)$ and $\underline{\ell}^{\perp}_{n}(\gamma,X,x) \geq c|\gamma - \gamma_{n}|^{d - 1}$ if $\gamma < \gamma_{n}(X,x)$.
			\item[(iii)] If $n \in \mathbb{N} \cup \{*\}$ and, for each $\nu = (n,\mu,X,R)$, $v^{\nu,x}$ is the solution of the Dirichlet problem
				\begin{equation*}
					\left\{ \begin{array}{r l}
							- F^{x}(e_{n}, \tilde{X}_{e_{n}} + D^{2}_{e_{n}} v^{\nu,x}, y) = \gamma_{n}(X,x) & \text{in} \, \, Q^{(n)}_{R}, \\
							v^{\nu,x} = 0 & \text{on} \, \, \partial Q^{(n)}_{R},
						\end{array} \right.
				\end{equation*} 
			then 
				\begin{equation*}
					\lim_{R \to \infty} \sup \left\{ R^{-2} |v^{\nu_{n}(R),x}(y)| \, \mid \, y \in Q^{(n)}_{R} \right\} = 0.
				\end{equation*}
		\end{itemize}
	(The constant $c > 0$ depends on $\lambda$, $\Lambda$, and $d$, but not on $n$.) 		\end{prop}  
	
%

In addition, we will need the following fact, adapted from \cite{caffarelli souganidis wang}, that follows from the homogenization result above:

	\begin{prop} \label{P: upper semicontinuous} For each $(\gamma,X,x) \in \mathbb{R} \times \mathcal{S}_{d} \times \mathbb{T}^{d}$, we have
		\begin{align*}
			\bar{\ell}^{\perp}_{*}(\gamma,X,x) &\geq \lim_{n \to \infty} \sup \left\{ \bar{\ell}_{n}^{\perp}(\gamma,X,x') \, \mid \, x' \in \mathbb{T}^{d} \right\}, \\
			\underline{\ell}_{*}^{\perp}(\gamma,X,x) &\geq \lim_{n \to \infty} \sup \left\{ \underline{\ell}^{\perp}_{n}(\mu,X,x') \, \mid \, x' \in \mathbb{T}^{d} \right\}.
		\end{align*}
	\end{prop}  
	
The proof uses the upper semi-continuity properties of the contact sets $\{K^{\nu,x}\}$.  When $(e_{n})_{n \in \mathbb{N}} \subseteq S^{d-1} \setminus \mathbb{R} \mathbb{Z}^{d}$, this can be combined with Fatou's Lemma by adapting the idea of \cite{caffarelli souganidis wang} directly to our setting using Proposition \ref{P: homogenization obstacle problem}.  

When $(e_{n})_{n \in \mathbb{N}} \subseteq \mathbb{R} \mathbb{Z}^{d}$, the situation is more delicate since the probability measures in the variational principle \eqref{E: variational principle subsolutions rational} depend on $n$.  As pointed out by W.M. Feldman, the same argument still applies if we replace Fatou's Lemma by the generalization due to Feinberg, Kasyanov, and Zadoianchuk \cite{fatou general}.  
	
		\begin{proof}  We give the details for $(\bar{\ell}^{\perp}_{n})_{n \in \mathbb{N}}$; the same basic idea also applies to $(\underline{\ell}^{\perp}_{n})_{n \in \mathbb{N}}$.  
		
		\textbf{Step 1: Property of Contact Sets}  
		
		Given $(\mu,X,R) \in \mathbb{R} \times \mathcal{S}_{d} \times (0,\infty)$ and $x \in \mathbb{T}^{d}$, let $\nu_{n} = (n,\mu,X,R)$ and $\nu_{*} = (*,\mu,X,R)$, and assume that $(x_{n})_{n \in \mathbb{N}} \subseteq \mathbb{T}^{d}$ satisfies $\lim_{n \to \infty} x_{n} = x$.  We claim that
			\begin{equation} \label{E: lower semicontinuity contact sets}
				\mathcal{H}^{d-1}(K^{\nu_{*},x}) \geq \limsup_{n \to \infty} \mathcal{H}^{d-1}(K^{\nu_{n},x_{n}}).
			\end{equation}
		To see this, first, define $\tilde{K} \subseteq Q^{*}_{R}$ by 
			\begin{equation*}
				\tilde{K} = \bigcap_{n = 1}^{\infty} \bigcup_{m = n}^{\infty} \left\{y \in Q^{*}_{R} \, \mid \, u^{\nu_{m},x_{m}}(O_{m}^{-1}(y)) = 0 \right\}.
			\end{equation*}
		Notice that Proposition \ref{P: homogenization obstacle problem} implies that $\tilde{K} \subseteq K^{\nu_{*},x}$.  Therefore, using the measure preserving property of the orthogonal transformations $\{O_{m}\}_{m \in \mathbb{N}}$, we find
			\begin{align*}
				\mathcal{H}^{d-1}(K^{\nu_{*},x}) &\geq \mathcal{H}^{d-1}(\tilde{K}) \\
					&= \lim_{n \to \infty} \mathcal{H}^{d-1} \left( \bigcup_{m = n}^{\infty} \left\{y \in Q^{*}_{R} \, \mid \, u^{\nu_{m},x_{m}}(O_{m}^{-1}(y)) = 0 \right\} \right) \\
					&\geq \limsup_{n \to \infty} \mathcal{H}^{d-1}(K^{\nu_{n},x_{n}}).
			\end{align*}
			
		\textbf{Step 2: Convenient Reformulation}  
		
		The result of the previous step can be reformulated slightly.  Define Borel functions $\{f_{n}\}_{n \in \mathbb{N}}$ and $f_{*}$ in $\mathbb{T}^{d}$ by 
			\begin{equation*}
				f_{n}(x) = \mathcal{H}^{d-1}(K^{\nu_{n},x}), \quad f_{*}(x) = \mathcal{H}^{d-1}(K^{\nu_{*},x}).
			\end{equation*}
		By Step 1, for each $x \in \mathbb{T}^{d}$, we have
			\begin{equation} \label{E: pure upper semicontinuity}
				\lim_{\delta \to 0^{+}} \sup \left\{ f_{n}(y) \, \mid \, n^{-1} + |x - y| < \delta \right\} \leq f_{*}(x).
			\end{equation}
			
		\textbf{Step 2: Conclusion}  
		
		We conclude the proof by considering two cases: (i) $(e_{n})_{n \in \mathbb{N}} \subseteq S^{d-1} \setminus \mathbb{R} \mathbb{Z}^{d}$ and (ii) $(e_{n})_{n \in \mathbb{N}} \subseteq \mathbb{R} \mathbb{Z}^{d}$.  Note that there is no loss of generality in assuming that either (i) or (ii) holds since we can always pass to subsequences if necessary.  
		
		In either case, we start by fixing an $\epsilon > 0$.  By invoking \eqref{E: variational principle subsolutions *}, we can fix an $R > 0$ such that
			\begin{equation*}
				\bar{\ell}^{\perp}_{*}(\gamma,X,x) \geq R^{1-d} \int_{\mathbb{T}^{d}} \mathcal{H}^{d-1}(K^{\nu_{*},x}) \, dx - \epsilon.
			\end{equation*}
		Now we consider cases (i) and (ii) separately.  
		
		In case (i), Fatou's Lemma, the conclusion of Step 1, and \eqref{E: variational principle subsolutions} combine to give
			\begin{align*}
				\int_{\mathbb{T}^{d}} \mathcal{H}^{d-1}(K^{\nu_{*},y}) \, dy \geq \limsup_{n \to \infty} \int_{\mathbb{T}^{d}} \mathcal{H}^{d-1}(K^{\nu_{n},y}) \, dy \geq \limsup_{n \to \infty} R^{d-1} \sup \left\{ \overline{\ell}_{n}(\gamma,X,x') \, \mid \, x' \in \mathbb{T}^{d} \right\}.
			\end{align*}
		Thus,
			\begin{equation*}
				\overline{\ell}_{*}(\gamma,X) \geq \limsup_{n \to \infty} \sup \left\{ \overline{\ell}_{n}(\gamma,X,x') \, \mid \, x' \in \mathbb{T}^{d} \right\} - \epsilon.
			\end{equation*}
		
		In case (ii), we can combine \eqref{E: variational principle subsolutions rational}, \eqref{E: pure upper semicontinuity}, Lemma \ref{L: equidistribution}, and the generalization of Fatou's Lemma in \cite[Theorem 1.1]{fatou general} to find, for each $(s_{n})_{n \in \mathbb{N}} \subseteq [0,\infty)$,
			\begin{align*}
				\int_{\mathbb{T}^{d}} \mathcal{H}^{d-1}(K^{\nu_{*},y}) \, dy &\geq \limsup_{n \to \infty} \fint_{\mathbb{T}^{d - 1}_{e_{n}}(s_{n})} \mathcal{H}^{d-1}(K^{\nu_{n},\xi}) \, \mathcal{H}^{d-1}(d \xi) \\
				&\geq \limsup_{n \to \infty} R^{d-1} \overline{\ell}_{n}(\gamma,X,s_{n}e_{n}).
			\end{align*}
		Since $(s_{n})_{n \in \mathbb{N}}$ was arbitrary, we conclude
			\begin{equation*}
				\overline{\ell}_{*}(\gamma,X) \geq \limsup_{n \to \infty} \sup \left\{ \overline{\ell}_{n}(\gamma,X,x') \, \mid \, x' \in \mathbb{T}^{d} \right\} - \epsilon.
			\end{equation*}
			
		In any case, the arbitrariness of $\epsilon > 0$ gives the desired result. 
		\end{proof}  
		
Finally, combining Propositions \ref{P: densities} and \ref{P: upper semicontinuous}, we obtain
	
	\begin{prop} \label{P: limit of mu} For each $(X,x) \in \mathcal{S}_{d} \times \mathbb{T}^{d}$, we have
		\begin{equation*}
			\lim_{n \to \infty} \sup \left\{ |\gamma_{n}(X,x') - \gamma_{*}(X,x)| \, \mid \, x' \in \mathbb{T}^{d} \right\} = 0.
		\end{equation*}  \end{prop} 
	
		\begin{proof}  Note that it is enough to prove the following two inequalities:
			\begin{align*}
				\gamma_{*}(X,x) &\geq \limsup_{n \to \infty} \sup \left\{ \mu_{n}(X,x') \, \mid \, x' \in \mathbb{T}^{d} \right\}, \\
				\gamma_{*}(X,x) &\leq \liminf_{n \to \infty} \inf \left\{ \mu_{n}(X,x') \, \mid \, x' \in \mathbb{T}^{d} \right\}.
			\end{align*}
		We will only prove the first inequality since the second one follows by a similar argument in which $(\bar{\ell}^{\perp}_{n})_{n \in \mathbb{N}}$ replaces $(\underline{\ell}^{\perp}_{n})_{n \in \mathbb{N}}$.
		
		Choose a sequence $(x_{n})_{n \in \mathbb{N}} \subseteq \mathbb{T}^{d}$ such that
			\begin{equation*}
				\limsup_{n \to \infty} \sup \left\{ \gamma_{n}(X,x') \, \mid \, x \in \mathbb{T}^{d} \right\} = \lim_{n \to \infty} \gamma_{n}(X,x_{n}).
			\end{equation*}
		To obtain the desired result, we will show that if $\gamma < \lim_{n \to \infty} \gamma_{n}(X,x_{n})$, then $\gamma < \gamma_{*}(X,x)$.

		To see this, suppose that $\gamma < \limsup_{n \to \infty} \gamma_{n}(X,x_{n})$.  By Proposition \ref{P: densities},
			\begin{equation*}
				\underline{\ell}^{\perp}_{*}(\gamma,X,x) \geq \limsup_{n \to \infty} \underline{\ell}^{\perp}_{n}(\gamma,X,x_{n}) \geq c \limsup_{n \to \infty} (\gamma_{n} - \gamma)^{d - 1} > 0.
			\end{equation*}
		From this, the same proposition with $n = *$ yields $\gamma < \gamma_{*}(X,x)$. \end{proof}   
		
\subsection{Proof of Proposition \ref{P: preliminary but main}}  The results of the previous section directly imply Proposition \ref{P: preliminary but main}, as we now show. 

	To start with, we note that the functions $(\gamma_{n})_{n \in \mathbb{N}}$ of Proposition \ref{P: densities} are precisely the oscillating functions $(F^{\perp}_{e_{n}})_{n \in \mathbb{N}}$.  

	\begin{lemma} \label{L: homogenization and density}  For each $n \in \mathbb{N}$ and $(X,x) \in \mathcal{S}_{d} \times \mathbb{T}^{d}$, the following identity holds:
		\begin{equation*}
			\gamma_{n}(X,x) = - F^{\perp}_{e_{n}}(X,x).
		\end{equation*}
	\end{lemma}  
	
		\begin{proof}  Fix $n \in \mathbb{N}$ and $(X,x) \in \mathcal{S}_{d} \times \mathbb{T}^{d}$.  By Proposition \ref{P: densities}, if we let $(v_{\epsilon})_{\epsilon > 0}$ be the solutions of the Dirichlet problem
			\begin{equation*}
				\left\{ \begin{array}{r l}
					- F^{x}(e_{n},\tilde{X}_{e_{n}} + D^{2}_{e_{n}} v_{\epsilon}, \epsilon^{-1} x') = \mu_{n}(X) & \text{in} \, \, Q^{(n)}_{1}, \\
					v_{\epsilon} = 0 & \text{on} \, \, \partial Q^{(n)}_{1},
				\end{array} \right.
			\end{equation*}
		then $v_{\epsilon} \to 0$ uniformly in $Q^{(n)}_{1}$ as $\epsilon \to 0^{+}$.  
		
		At the same time, this is an almost periodic elliptic homogenization problem in $\langle e_{n} \rangle^{\perp}$ and we know (e.g.\ using the approximate correctors of Theorem \ref{T: approximate correctors}) that $v_{\epsilon} \to \bar{v}$, where $\bar{v}$ is the solution of the constant coefficient equation
			\begin{equation*}
				\left\{ \begin{array}{r l}
					- F^{\perp}_{e_{n}}(\tilde{X}_{e_{n}} + D^{2}_{e_{n}} \bar{v}, x) = \gamma_{n}(X,x) & \text{in} \, \, Q^{(n)}_{1} \\
					\bar{v} = 0 & \text{on} \, \, \partial Q^{(n)}_{1}
				\end{array} \right.
			\end{equation*}
		The previous paragraph says that $\bar{v} \equiv 0$ in $Q^{(n)}_{1}$.  Therefore, 
			\begin{equation*}
				\gamma_{n}(X,x) = - F^{\perp}_{e_{n}}(\tilde{X}_{e_{n}},x).
			\end{equation*}
		\end{proof}  
		
	Now Proposition \ref{P: preliminary but main} follows by combining Lemma \ref{L: homogenization and density} with Proposition \ref{P: limit of mu}.
	
	Before moving on to the proofs of Proposition \ref{P: continuity in dimension 2} and Theorem \ref{T: limiting measures}, let us identify the function $\gamma_{*}$.  We begin with the setting of Theorem \ref{T: level set PDE}:
		
		\begin{lemma} \label{L: two d resolution}  If $d = 2$, then, for each $x \in \mathbb{T}^{d}$, the function $-\gamma_{*}(\cdot,x) : \mathcal{S}_{d} \to \mathbb{R}$ is the homogenized coefficient associated with the (e.g.\ Dirichlet) elliptic homogenization problem:
			\begin{equation} \label{E: one dimensional problem}
				\left\{ \begin{array}{r l}
					U - F^{\perp}_{e}(D^{2}_{e} U, \epsilon^{-1} \langle x',e^{\perp} \rangle e) = 0 & \text{in} \, \, Q^{*}_{1} \\ 
					U = G & \text{on} \, \, \partial Q^{*}_{1}
				\end{array} \right.
			\end{equation}
		Here $e^{\perp}$ is any unit vector with $\langle e^{\perp}, e \rangle = 0$.  In particular, $\gamma_{*}$ is independent of $\eta$.  \end{lemma}

			\begin{proof}  Fix $e^{\perp} \in S^{1} \cap \langle e \rangle^{\perp}$ and note that $S^{1} \cap \langle e \rangle^{\perp} = \{e^{\perp},-e^{\perp}\}$.
			
			Arguing as in Lemma \ref{L: homogenization and density} and replacing $e^{\perp}$ by $-e^{\perp}$ if necessary in \eqref{E: one dimensional problem}, we see that $-\mu_{*}(X)$ is the homogenized coefficient associated with \eqref{E: one dimensional problem}.  At the same time, a direct manipulation of \eqref{E: one dimensional problem} shows this coefficient is unchanged if $e^{\perp}$ is replaced by $-e^{\perp}$.  Therefore, $-\mu_{*}(X)$ does not depend on $\eta$. \end{proof}  
			
The quasi-linear problem can be treated in the same way.  However, as soon as $d \geq 3$, the limiting coefficient $\mu_{*}$ has a non-trivial dependence on the direction $\eta$.
		
\begin{lemma} \label{L: limiting measure equation} If $m$ and $a$ satisfy the assumptions of Theorem \ref{T: quasilinear} and $F(p,X,y) = m(y,\hat{p})^{-1} \text{tr}(a(y,\hat{p}) \tilde{X}_{\hat{p}})$, then, for each $x \in \mathbb{T}^{d}$,
	\begin{equation} \label{E: limiting coefficient}
		- \gamma_{*}(X,x) = (\tilde{m}_{e}^{\eta})^{-1} \text{tr}(\tilde{a}^{\eta}_{e} \tilde{X}_{e}),
	\end{equation}
where $\tilde{a}^{\eta}_{e}$ and $\tilde{m}^{\eta}_{e}$ are given by 
	\begin{equation*}
		\tilde{a}^{\eta}_{e} = \int_{\mathbb{T}^{d}} a(y,e) \, \tilde{\mu}_{e}^{\eta}(dy),\quad \tilde{m}^{\eta}_{e} = \int_{\mathbb{T}^{d}} m(y,e) \, \tilde{\mu}_{e}^{\eta}(dy).
	\end{equation*}
(Here $\tilde{\mu}^{\eta}_{e}$ is the measure defined in Theorem \ref{T: limiting measures}.)  \end{lemma}  

	\begin{proof}  As in Lemma \ref{L: two d resolution}, $-\gamma_{*}(X,x)$ is the homogenized coefficient associated with the homogenization problem
		\begin{equation*}
			\left\{ \begin{array}{r l}
				m^{\perp}_{e}(\epsilon^{-1}\langle x', \eta \rangle e) - \text{tr}\left(a^{\perp}_{e}(\epsilon^{-1} \langle x', \eta \rangle e) D^{2}_{e} U\right) = 0 & \text{in} \, \, Q^{*}_{1}, \\
				U = G & \text{on} \, \, \partial Q^{*}_{1}.
			\end{array} \right.
		\end{equation*}
	Since the coefficients only vary in the $\eta$ direction, the homogenized coefficients are the same as if we consider the one-dimensional problem.  In particular, by a well-known computation, $-\gamma_{*}(X,x)$ is given by \eqref{E: limiting coefficient}, where $\tilde{a}^{\eta}_{e}$ and $\tilde{m}^{\eta}_{e}$ are defined by 
		\begin{align*}
			\tilde{a}^{\eta}_{e} &= \left( \int_{0}^{r_{e}} \langle a_{e}^{\perp}(se) \eta, \eta \rangle^{-1} \, ds \right)^{-1} \int_{0}^{r_{e}} a_{e}^{\perp}(se) \langle a_{e}^{\perp}(se) \eta, \eta \rangle^{-1} \, ds, \\
			\tilde{m}^{\eta}_{e} &= \left( \int_{0}^{r_{e}} \langle a_{e}^{\perp}(se) \eta, \eta \rangle^{-1} \, ds \right)^{-1} \int_{0}^{r_{e}} m_{e}^{\perp}(se) \langle a_{e}^{\perp}(se) \eta, \eta \rangle^{-1} \, ds.
		\end{align*}
	We conclude by recalling the definitions of $a^{\perp}_{e}$ and $m^{\perp}_{e}$ (see \eqref{E: oscillating tensor} and \eqref{E: oscillating function linear}).
	\end{proof}  
	
\subsection{Proofs of Propositions \ref{P: continuity in dimension 2} and Theorem \ref{T: limiting measures}}  By applying Proposition \ref{P: preliminary but main}, we readily deduce the continuity of $\overline{F}$ in $d = 2$ and the limiting behavior of the set-valued maps $e \in \mathscr{I}^{a}_{e}$ in the quasi-linear setting.  

\begin{proof}[Proof of Proposition \ref{P: preliminary but main}]  Since Proposition \ref{P: preliminary but main} applies to an arbitrary sequence and $\gamma_{*}^{e,\eta}$ is independent of $\eta$ by Lemma \ref{L: two d resolution}, we deduce that, for each $e \in S^{1} \cap \mathbb{R} \mathbb{Z}^{2}$,
	\begin{equation*}
		\lim_{\delta \to 0^{+}} \sup \left\{ |\overline{F}(e',X) + \gamma_{*}^{e,e^{\perp}}(X)| \, \mid \, e' \in S^{1} \setminus \mathbb{R} \mathbb{Z}^{2}, \, \, \|e' - e\| < \delta \right\} = 0.
	\end{equation*}
Hence $\overline{F}$ extends continuously to $(\mathbb{R}^{2} \setminus \{0\}) \times \mathcal{S}_{2}$.  \end{proof}  

In light of the form of $\{F^{\perp}_{e}\}$ in the quasi-linear case, a similar argument implies Theorem \ref{T: limiting measures}, as we now show:

\begin{proof}[Proof of Theorem \ref{T: limiting measures}]  Notice that to obtain the conclusion of Theorem \ref{T: limiting measures}, it suffices to show that if $(e_{n})_{n \in \mathbb{N}} \subseteq S^{d-1}$ is such that $e_{n} \to e$ and $\frac{e_{n} - e}{\|e_{n} - e\|} \to - \eta$ as $n \to \infty$, then, for each positive $m \in C^{\infty}(\mathbb{T}^{d})$ and each $(s_{n})_{n \in \mathbb{N}} \subseteq \mathbb{R}$, we have
	\begin{equation*}
		\lim_{n \to \infty} \int_{\mathbb{T}^{d}} m(y) \mu_{e_{n}}^{s_{n}}(dy) = \int_{\mathbb{T}^{d}} m(y) \tilde{\mu}^{\eta}_{e}(dy).
	\end{equation*}
(Here define $\mu^{s}_{e'} = \bar{\mu}_{e'}$ if $s \in \mathbb{R}$ and $e' \in S^{d-1} \setminus \mathbb{R} \mathbb{Z}^{d}$.) This follows directly from Proposition \ref{P: preliminary but main}, Lemmas \ref{L: homogenization and density} and \ref{L: limiting measure equation}, and \eqref{E: oscillating function linear} by varying $m$ while $a$ remains fixed. \end{proof}  
	
\subsection{Generic Discontinuities}  In light of the formulas obtained for the limiting measures in the previous section, it is natural to expect that $\bar{a}$ and $\bar{m}$ are generically discontinuous at some rational directions.  In this section, we prove that, in fact, $\bar{a}$ and $\bar{m}$ are generically discontinuous at every rational direction when $d \geq 3$.

\begin{proof}[Proof of Corollary \ref{C: generic}]  To start with, since $d \geq 3$, we can fix $e \in S^{d-1}$ and let $(\eta_{n})_{n \in \mathbb{N}}$ be a sequence of points in $S^{d-1} \cap \langle e \rangle^{\perp}$ with $\eta_{n} \notin \{\eta_{m}, - \eta_{m}\}$ for all $n \neq m$.  The goal is to prove that, for each $n, m \in \mathbb{N}$ with $n \neq m$, the following sets are open and dense in $C^{2,\alpha}(\mathbb{T}^{d}; \mathcal{S}_{d}(\lambda,\Lambda))$ in the $C^{2,\alpha}$ norm topology:
	\begin{align*}
		\mathcal{U}_{e}(n,m) = \left\{ a \in C^{2,\alpha}(\mathbb{T}^{d}; \mathcal{S}_{d}(\lambda,\Lambda)) \, \mid \, \tilde{\mu}_{e}^{\eta_{n}} \neq \tilde{\mu}_{e}^{\eta_{m}} \right\}, \\
		\mathcal{V}_{e}(n,m) = \left\{ a \in C^{2,\alpha}(\mathbb{T}^{d}; \mathcal{S}_{d}(\lambda,\Lambda)) \, \mid \, \tilde{a}_{e}^{\eta_{n}} \neq \tilde{a}_{e}^{\eta_{m}} \right\}.
	\end{align*}
That these sets are open is immediate.  It only remains to show they are dense.  


We start with $\mathcal{U}_{e}(n,m)$.  It turns out that we only need to understand the derivative of the map $a \mapsto a_{e}^{\perp}$.  Toward that end, for each $n \in \mathbb{N}$, define $u_{n} \in C^{2,\alpha}(\mathbb{T}^{d})$ to be the solution of the cell problem \eqref{E: exotic cell problem} with $f(y) = \langle a(y) \eta_{n}, \eta_{n} \rangle$.  Notice that the oscillating function $f_{e}^{\perp}$ equals $\langle a_{e}^{\perp} \eta_{n}, \eta_{n} \rangle$ in this case.

\textbf{Step 1: Perturb so that $D^{2}_{e}u_{n} \neq D^{2}_{e}u_{m}$}

To start with, we claim that there is no loss of generality in assuming that $D^{2}_{e}u_{n} \neq D^{2}_{e}u_{m}$.  Indeed, if $D^{2}_{e}u_{n} = D^{2}_{e}u_{m}$, then \eqref{E: exotic cell problem} implies $\langle [a - a^{\perp}_{e}] \eta_{n}, \eta_{n} \rangle = \langle [a - a^{\perp}_{e}] \eta_{m}, \eta_{m} \rangle$.  That is, 
	\begin{equation*}
		\langle a(y) \eta_{n}, \eta_{n} \rangle - \langle a(y) \eta_{m}, \eta_{m} \rangle = \langle a_{e}^{\perp}(\langle y,e \rangle e) \eta_{n}, \eta_{n} \rangle - \langle a_{e}^{\perp}(\langle y, e \rangle e) \eta_{m}, \eta_{m} \rangle \quad \text{if} \, \, y \in \mathbb{T}^{d}.
	\end{equation*}
In particular, the left-hand side varies only in the $e$ direction.  This symmetry is easily broken, for instance, by replacing $a$ by $y \mapsto a(y) + \nu \cos(2 \pi \langle k, y \rangle) \eta_{n} \otimes \eta_{n}$ for some $k \in \mathbb{Z}^{d} \setminus \langle e \rangle$ and $\nu \in \mathbb{R}$ sufficiently small.  

\textbf{Step 2: Restrict attention to $a_{e}^{\perp}$}

First, observe that $\tilde{\mu}^{\eta_{n}}_{e} = \tilde{\mu}^{\eta_{m}}_{e}$ if and only if there is a $C > 0$ such that $\langle a_{e}^{\perp} \eta_{n}, \eta_{n} \rangle = C \langle a_{e}^{\perp} \eta_{m}, \eta_{m} \rangle$ in $\mathbb{T}^{d}$.  This is a direct consequence of the formula \eqref{E: limiting measure} proved in the previous section.  We claim that if $\langle a_{e}^{\perp} \eta_{n}, \eta_{n} \rangle = C \langle a_{e}^{\perp} \eta_{m}, \eta_{m} \rangle$ for some $C > 0$, then this symmetry is broken by some arbitrarily small perturbation of $a$.

To see this, we will differentiate the function $a \mapsto a_{e}^{\perp}$.  Given $a_{*} \in C^{2,\alpha}(\mathbb{T}^{d}; \mathcal{S}_{d})$ and $h \in \mathbb{R}$ small enough, define $a_{h} = a + h a_{*}$ and let $(a_{h})_{e}^{\perp}$ be the associated averaged tensor.

Employing arguments similar to those in the proof of Proposition \ref{P: rational correctors}, we see that, for each $r \in [0,r_{e})$, $h \in \mathbb{R}$, and $j \in \{n,m\}$, if $\tilde{U}_{j,re}^{h}$ is the solution of the equation
	\begin{equation*}
		\left\{ \begin{array}{r l}
			- \text{tr} \left( a_{h}(re + x') D^{2} \tilde{U}_{j,re}^{h} \right) = \langle a_{h}(re + x') \eta_{j}, \eta_{j} \rangle - \langle (a_{h})_{e}^{\perp}(re) \eta_{j}, \eta_{j} \rangle & \text{in} \, \, \mathbb{T}^{d-1}_{e}(0) \\
			\tilde{U}_{j,re}^{h}(0) = 0 
		\end{array} \right.
	\end{equation*} 
then there are functions $\{\tilde{U}_{n,re}, \tilde{U}_{m,re}\} \in C(\mathbb{T}^{d-1}_{e}(0))$ and a function $a_{*}^{\perp} \in C(\mathbb{T}^{d})$ varying only in the $e$ direction such that, for $j \in \{n,m\}$ and $r \in [0,r_{e})$,
	\begin{equation*}
		\tilde{U}_{j,re} = \lim_{h \to 0} \frac{\tilde{U}^{h}_{j,re} - \tilde{U}^{0}_{j,re}}{h} \quad \text{uniformly in} \, \, \mathbb{T}^{d-1}_{e}(0), \quad a_{*}^{\perp} = \lim_{h \to 0} \frac{(a_{h})_{e}^{\perp} - a_{e}^{\perp}}{h} \quad \text{uniformly in} \, \, \mathbb{T}^{d}.
	\end{equation*}
Furthermore, $\tilde{U}_{j,re}$ is a solution of the equation
	\begin{equation*}
		- \text{tr} \left( a(re + x') D^{2}_{e}\tilde{U}_{j,re}\right) = \text{tr} \left(a_{*}(re + x') D^{2}_{e} u_{j}(re + x') \right) - \langle a_{*}^{\perp}(re) \eta_{j}, \eta_{j} \rangle \quad \text{in} \, \, \mathbb{T}^{d-1}_{e}(0).
	\end{equation*} 
Notice that $a_{*}^{\perp}(y) = \int_{\mathbb{T}^{d}} \text{tr} \left(a_{*}(y') D^{2}_{e} u_{j}(y') \right) \mu_{e}^{\langle y, e \rangle}(dy')$.  

By the previous step, there is no loss in generality assuming that $D^{2}_{e}u_{n} \neq D^{2}_{e}u_{m}$.  In fact, since these functions are $C^{\alpha}$, we can fix $y_{1},y_{2} \in \mathbb{T}^{d}$ so that $\langle y_{1}, e \rangle \neq \langle y_{2}, e \rangle$ and $D^{2}_{e}u_{n}(y_{i}) \neq D^{2}_{e}u_{m}(y_{i})$ for $i \in \{1,2\}$.  Since $\mu_{e}^{s}$ is supported on $\mathbb{T}^{d-1}_{e}(s)$ for each $s \in [0,r_{e})$, we can we fix $a_{*} \in C^{\infty}(\mathbb{T}^{d}; \mathcal{S}_{d})$ so that
	\begin{align*}
		\int_{\mathbb{T}^{d}} \text{tr} \left(a_{*}(y') D^{2}_{e} u_{n}(y') \right) \mu_{e}^{\langle y_{1}, e \rangle}(dy') > 0 > \int_{\mathbb{T}^{d}} \text{tr} \left(a_{*}(y') D^{2}_{e}u_{m}(y') \right) \, \mu_{e}^{\langle y_{1},e\rangle}(dy'), \\
		\int_{\mathbb{T}^{d}} \text{tr} \left(a_{*}(y') D^{2}_{e} u_{n}(y') \right) \mu_{e}^{\langle y_{2}, e \rangle}(dy') < 0 < \int_{\mathbb{T}^{d}} \text{tr} \left(a_{*}(y') D^{2}_{e}u_{m}(y') \right) \, \mu_{e}^{\langle y_{2},e\rangle}(dy').
	\end{align*}
From this, we find that, for all $h > 0$ sufficiently small,
	\begin{equation*}
		\frac{\langle (a_{h})^{\perp}_{e}(y_{1}) \eta_{m}, \eta_{m} \rangle}{\langle (a_{h})^{\perp}_{e}(y_{1}) \eta_{n}, \eta_{n} \rangle} < C < \frac{\langle (a_{h})^{\perp}_{e}(y_{2}) \eta_{m}, \eta_{m} \rangle}{\langle (a_{h})^{\perp}_{e}(y_{2}) \eta_{n}, \eta_{n} \rangle}.
	\end{equation*}
Therefore, $\langle (a_{h})_{e}^{\perp} \eta_{n}, \eta_{n} \rangle$ is not a constant multiple of $\langle (a_{h})_{e}^{\perp} \eta_{m}, \eta_{m} \rangle$ even while $\|a - a_{h}\|_{C^{2,\alpha}(\mathbb{T}^{d})} \leq C h$.   

We conclude from the preceding that $\mathcal{U}_{e}(n,m)$ is dense in $C^{2,\alpha}(\mathbb{T}^{d}; \mathcal{S}_{d}(\lambda,\Lambda))$.  It remains to prove the same thing for $\mathcal{V}_{e}(n,m)$.  

\textbf{Step 3: Ensure $\tilde{a}^{\eta_{n}}_{e} \neq \tilde{a}^{\eta_{m}}_{e}$}  

We want to show that $\mathcal{V}_{e}(n,m)$ is dense.  Given that $\tilde{a}^{\eta}_{e} = \int_{\mathbb{T}^{d}} a(y) \, \tilde{\mu}^{\eta}_{e}(dy)$, this is intuitively clear given what we just proved.  

The previous arguments show that we can assume that $a \in \mathcal{U}_{e}(n,m) \setminus \mathcal{V}_{e}(n,m)$ to start with.  We will show that there is an $m \in C^{2,\alpha}(\mathbb{T}^{d})$ varying only in the $e$ direction such that the function $a_{h} = (1 + hm)^{-1} a$ is in $\mathcal{V}_{e}(n,m)$ for all $h \in \mathbb{R}$ sufficiently small.  

To start with, notice that, by arguing as in Section \ref{S: invariant measures}, we see that, for each $\eta \in S^{d-1} \cap \langle e \rangle^{\perp}$, the limiting coefficient $\tilde{a}_{h,e}^{\eta}$ associated with $a_{h}$ is given by
	\begin{equation} \label{E: final computation}
		\tilde{a}_{h,e}^{\eta} = \frac{\tilde{a}_{e}^{\eta}}{1 + h \int_{\mathbb{T}^{d}} m(y) \tilde{\mu}^{\eta}_{e}(dy)}.
	\end{equation}
Since $m$ only varies in the $e$ direction, the integral in the denominator becomes
	\begin{equation*}
		\int_{\mathbb{T}^{d}} m(y) \tilde{\mu}^{\eta}_{e}(dy) = \frac{\int_{0}^{r_{e}} m(se) \langle a_{e}^{\perp}(se) \eta, \eta \rangle^{-1} \, ds}{\int_{0}^{r_{e}} \langle a_{e}^{\perp}(se) \eta, \eta \rangle^{-1} \, ds}.
	\end{equation*}
Given that $a \in \mathcal{U}_{e}(n,m)$, we know that $\langle a_{e}^{\perp}(se) \eta_{n}, \eta_{n} \rangle$ does not equal a multiple of $\langle a_{e}^{\perp}(se) \eta_{m}, \eta_{m} \rangle$, and, therefore, we can choose $m$ so that $\int_{\mathbb{T}^{d}} m(\langle y,e \rangle e) \tilde{\mu}^{\eta_{n}}_{e}(dy) \neq \int_{\mathbb{T}^{d}} m(\langle y,e \rangle e) \tilde{\mu}^{\eta_{m}}_{e}(dy)$.  Therefore, since $\tilde{a}^{\eta_{n}}_{e} = \tilde{a}^{\eta_{m}}_{e}$, \eqref{E: final computation} implies $\tilde{a}^{\eta_{n}}_{h,e} \neq \tilde{a}^{\eta_{m}}_{h,e}$ for all $h$ small enough.  This proves $a$ is a limit point of $\mathcal{V}_{e}(n,m)$.

\textbf{Conclusion}  

We showed that $\mathcal{U}_{e}(n,m)$ and $\mathcal{V}_{e}(n,m)$ are both open and dense in $C^{2,\alpha}(\mathbb{T}^{d} ; \mathcal{S}_{d}(\lambda,\Lambda))$ in the $C^{2,\alpha}$ norm topology.  Define $\mathscr{C}_{d}$ by
	\begin{equation*}
		\mathscr{C}_{d} = \bigcap_{n \in \mathbb{N}} \bigcap_{m \in \mathbb{N} \setminus \{n\}} \mathcal{U}_{e}(n,m) \cap \mathcal{V}_{e}(n,m).
	\end{equation*}  
This set is residual, being a counta`ble intersection of open, dense sets.  Further, since $C^{2,\alpha}(\mathbb{T}^{d}; \mathcal{S}_{d}(\lambda,\Lambda))$ is an open subset of the Banach space $C^{2,\alpha}(\mathbb{T}^{d}; \mathcal{S}_{d})$, $\mathscr{C}_{d}$ is itself dense.
\end{proof}  

\section{Comparison Principle}  \label{S: comparison}

In this section, we prove the well-posedness of \eqref{E: effective equation}, completing the proof of Theorem \ref{T: quasilinear}.   

To clarify the exposition, we will consider the following general framework.  Given operators $G^{*}, G_{*} : \mathbb{R}^{d} \times \mathcal{S}_{d} \to \mathbb{R}$ and a countable set $\{e_{n}\}_{n \in \mathbb{N}} \subseteq S^{d-1}$ of ``bad" directions, we consider sub- and supersolutions of the viscosity inequalities
	\begin{equation} \label{E: sub and super solutions}
		u_{t} - G^{*}(Du,D^{2}u) \leq 0, \quad u_{t} - G_{*}(Du,D^{2}u) \geq 0 \quad \text{in} \, \, \mathbb{R}^{d} \times (0,T).
	\end{equation}
Here we assume that the operators $G^{*}, G_{*} : \mathbb{R}^{d} \times \mathcal{S}_{d} \to \mathbb{R}$ satisfy the following assumptions:
	\begin{itemize}
		\item[(i)](Geometric) If $G \in \{G^{*},G_{*}\}$, $(p,X) \in \mathbb{R}^{d} \times \mathcal{S}_{d}$, $\mu \in \mathbb{R}$, and $\kappa > 0$, then
			\begin{equation*}
				G(\kappa p,\kappa X + \mu p \otimes p) = \kappa G(p, X).
			\end{equation*}
		\item[(ii)](Strongly degenerate elliptic)  There are constants $\lambda,\Lambda > 0$ such that if $G \in \{G^{*},G_{*}\}$, $p \in \mathbb{R}^{d} \setminus \{0\}$, $X, Y \in \mathcal{S}_{d}$, and $Y \geq 0$, then 
			\begin{equation*}
				\lambda \|\tilde{Y}_{\hat{p}}\| \leq G(p,X + Y) - G(p,X) \leq \Lambda \|\tilde{Y}_{\hat{p}}\|.
			\end{equation*}
		\item[(iii)](Stationary planes) $G^{*}(e,0) = G_{*}(e,0) = 0$ for each $e \in S^{d-1}$.
		\item[(iv)](Semi-continuity) $G^{*}$ is upper semi-continuous, $G_{*}$ is lower semi-continuous, $(G^{*})_{*} = G_{*}$, and $(G_{*})^{*} = G^{*}$.
		\item[(v)](Continuity at ``good" directions)  If $(p,X) \in (\mathbb{R}^{d} \setminus \{0\}) \times \mathcal{S}_{d}$ and $\hat{p} \notin \{e_{n}\}_{n \in \mathbb{N}}$, then 
			\begin{equation*}
				G^{*}(p,X) = G_{*}(p,X).
			\end{equation*}
		\item[(vi)](Controlled oscillation)  The discontinuities of $G^{*}$ and $G_{*}$ can be controlled in the following manner:
			\begin{equation*} 
				\lim_{N \to \infty} \sup \left\{ \frac{G^{*}(e_{n},X) - G_{*}(e_{n},X)}{1 + \|X\|} \, \mid \, X \in \mathcal{S}_{d}, \, \, n \geq N \right\} = 0.
			\end{equation*} 
	\end{itemize}
	
In the next subsection, we will show that the assumptions above are satisfied when $G^{*} = \overline{F}^{*}$ and $G_{*} = \overline{F}_{*}$.  Thus, a comparison result for operators satisfying (i)-(vi) implies homogenization in Theorem \ref{T: quasilinear} just as the usual one did in Theorem \ref{T: level set PDE}.

The remainder of the section is devoted to the proof of just such a comparison principle, stated next:

	\begin{theorem} \label{T: comparison}  Assume that $u : \mathbb{R}^{d} \times (0,T) \to \mathbb{R}$ is a locally bounded, upper semi-continuous subsolution of \eqref{E: sub and super solutions} and $v : \mathbb{R}^{d} \times (0,T) \to \mathbb{R}$ is a locally bounded, lower semi-continuous supersolution.  If (i)-(vi) all hold and $u$ and $v$ satisfy the following condition
		\begin{equation*}
			\lim_{\delta \to 0^{+}} \sup \left\{ u^{*}(x,0) - v_{*}(y,0) \, \mid \, \|x - y\| < \delta \right\} \leq 0,
		\end{equation*}
	then $u \leq v$ in $\mathbb{R}^{d} \times (0,T)$.  \end{theorem}
	
The idea of the proof is this: assumption (v) implies that, for each $\beta > 0$, $F^{*}$ and $F_{*}$ \emph{almost} coincide (up to a $\beta$ error) except at finitely many rational directions.  The papers of Gurtin, Soner, and Souganidis \cite{gurtin soner souganidis}, Ohnuma and Sato \cite{ohnuma sato}, and Ishii \cite{ishii level set} show how to prove a comparison principle in the case when $G^{*}$ and $G_{*}$ coincide at all but finitely many directions.  Therefore, if we can manage the $\beta$ error, a comparison principle should hold in our setting as well.  

\subsection{Application to $\overline{F}$}  Let us verify that the semi-continuous envelopes $\overline{F}^{*}$ and $\overline{F}_{*}$ of the effective operator $\overline{F}$ of \eqref{E: homogenized operator} satisfy assumptions (i)-(vi) above with $\{e_{n}\}_{n \in \mathbb{N}}$ being any enumeration of $S^{d-1} \cap \mathbb{R} \mathbb{Z}^{d}$.  For the sake of completeness, we will then recall how Theorem \ref{T: comparison} can be combined with what we previously proved to give Theorem \ref{T: quasilinear}.  Finally, in a concluding remark, we note that the limiting behavior of the forced motion \eqref{E: quasilinear forced} cannot be described by a level set PDE with a comparison principle.

	Let us start with assumptions (i)-(v).

	\begin{lemma}  If $\{e_{n}\}_{n \in \mathbb{N}}$ is any enumeration of $S^{d-1} \cap \mathbb{R} \mathbb{Z}^{d}$, then the pair $(\overline{F}^{*},\overline{F}_{*})$ satisfies assumptions (i)-(v) with ``bad directions" $\{e_{n}\}_{n \in \mathbb{N}}$. \end{lemma}    
	
		\begin{proof}  (i) follows by the definition \eqref{E: homogenized operator} and (ii) and (iii) were proved in Corollary \ref{C: basic properties}.  (iv) is a direct consequence of the definition of $\overline{F}^{*}$ and $\overline{F}_{*}$ (see Section \ref{S: irrational directions}).
		
		It only remains to prove (v).  If $e \in S^{d-1} \setminus \mathbb{R} \mathbb{Z}^{d}$ and $(e_{n})_{n \in \mathbb{N}} \subseteq S^{d-1}$, then, for any $(s_{n})_{n \in \mathbb{N}} \subseteq (0,\infty)$, each accumulation point of the measures $(\mu^{s_{n}}_{e_{n}})_{n \in \mathbb{N}}$ is necessarily in $\mathscr{I}^{a}_{e}$.  (Here, if necessary, we set $\mu^{s}_{e'} = \bar{\mu}_{e'}$ when $e' \notin \mathbb{R} \mathbb{Z}^{d}$.)  Therefore, by Theorem \ref{T: invariant measures}, (i), we have $\mu^{s_{n}}_{e_{n}} \overset{*}{\rightharpoonup} \bar{\mu}_{e}$.  From this, we conclude that, for each $X \in \mathcal{S}_{d}$ and each sequence $(x_{n})_{n \in \mathbb{N}} \subseteq \mathbb{T}^{d}$, we must have $F^{\perp}_{e_{n}}(X,x_{n}) \to \overline{F}(e,X)$.  From this, a straightforward computation shows that $\overline{F}^{*}(e,\cdot) = \overline{F}(e,\cdot) = \overline{F}_{*}(e,\cdot)$ in $\mathcal{S}_{d}$.      \end{proof}  
		
	Finally, we treat (vi).  By analogy with similar results in homogenization and Aubry-Mather theory (cf.\ \cite[Lemma 3.1]{feldman} and \cite[Theorem 3]{senn differentiation}), we expect that there is a modulus $\omega : [0,\infty) \to [0,\infty)$ with $\lim_{\delta \to 0^{+}} \omega(\delta) = 0$ such that, for each $e \in S^{d-1} \cap \mathbb{R} \mathbb{Z}^{d}$, the following estimate holds:		
		\begin{equation} \label{E: oscillation estimate}
			\sup \left\{ \frac{\overline{F}^{*}(e,X) - \overline{F}_{*}(e,X)}{1 + \|X\|} \, \mid \, X \in \mathcal{S}_{d} \right\} \leq \omega(r_{e}).
		\end{equation}
	When $d = 2$, it is not hard to show that if $F$ satisfies \eqref{E: standard lipschitz}, then there is a constant $A > 0$ such that
		\begin{equation*}
			\sup \left\{ \frac{F^{\perp}_{e}(X,x) - F^{\perp}_{e}(X,y)}{1 + \|X\|} \, \mid \, X \in \mathcal{S}_{d}, \, \, x,y \in \mathbb{T}^{d}\right\} \leq A r_{e}.
		\end{equation*} 
	If such an estimate were to hold in higher dimensions (possibly with $A r_{e}$ replaced by $\omega(r_{e})$), then it would imply \eqref{E: oscillation estimate}.  However, this remains to be seen.
	
	Instead, we employ a soft argument pointed out by I.C. Kim:
	
	\begin{lemma}  The pair $(\overline{F}^{*},\overline{F}_{*})$ satisfies (vi) with $\{e_{n}\}_{n \in \mathbb{N}}$ any enumeration of $S^{d-1} \cap \mathbb{R} \mathbb{Z}^{d}$.  \end{lemma}  
	
		\begin{proof}  We claim that 
			\begin{equation*}
				\lim_{N \to \infty} \sup \left\{ \text{diam}(\mathscr{I}^{a}_{e_{n}}) \, \mid \, n \geq N \right\} = 0.
			\end{equation*}
		Here $\text{diam}(\mathscr{I}^{a}_{e})$ is the diameter of $\mathscr{I}^{a}_{e}$ with respect to $D$, the metric on $\mathscr{P}(\mathbb{T}^{d})$ chosen just prior to Theorem \ref{T: limiting measures}.  Notice that, in view of Theorem \ref{T: limiting measures} and the definition \eqref{E: homogenized operator} of $\overline{F}$, the claim implies (vi) holds.  
		
		To prove it, we argue by contradiction, exploiting the compactness of $S^{d-1}$.  If (vi) fails, then we can find $\zeta > 0$, $(e_{n})_{n \in \mathbb{N}} \subseteq S^{d-1}$, and sequences $(s_{n})_{n \in \mathbb{N}}, (t_{n})_{n \in \mathbb{N}} \subseteq \mathbb{R}$ such that
			\begin{equation} \label{E: big diameter}
				\inf \left\{ D(\mu^{s_{n}}_{e_{n}},\mu^{t_{n}}_{e_{n}}) \, \mid \, n \in \mathbb{N} \right\} \geq \zeta.
			\end{equation}
		To see this is impossible, note that, up to extraction, we can assume that there is an $e \in S^{d-1}$ such that $\lim_{n \to \infty} e_{n} = e$.  
		
		If $e \in S^{d-1} \setminus \mathbb{R} \mathbb{Z}^{d}$, then any accumulation point of $(\mu^{s_{n}}_{e_{n}})_{n \in \mathbb{N}}$ or $(\mu^{t_{n}}_{e_{n}})_{n \in \mathbb{N}}$ is in $\mathscr{I}^{a}_{e}$, hence must equal $\bar{\mu}_{e}$.  In particular, $\mu^{s_{n}}_{e_{n}} \overset{*}{\rightharpoonup} \bar{\mu}_{e}$ and $\mu^{t_{n}}_{e_{n}} \overset{*}{\rightharpoonup} \bar{\mu}_{e}$, contradicting \eqref{E: big diameter}.
		
		On the other hand, if $e \in S^{d-1} \setminus \mathbb{R} \mathbb{Z}^{d}$, then, passing to another subsequence if necessary, we can assume that there is an $\eta \in S^{d-1} \cap \langle e \rangle^{\perp}$ such that $\frac{e_{n} - e}{\|e_{n} - e\|} \to -\eta$ as $n \to \infty$.  By Theorem \ref{T: limiting measures}, this implies $\mu^{s_{n}}_{e_{n}} \overset{*}{\rightharpoonup} \tilde{\mu}^{\eta}_{e}$ and $\mu^{t_{n}}_{e_{n}} \overset{*}{\rightharpoonup} \tilde{\mu}^{\eta}_{e}$, another contradiction.
		\end{proof}  

Now, for the sake of completeness, observe that we can combineTheorem \ref{T: comparison} with the results of Section \ref{S: irrational directions} to prove Theorem \ref{T: quasilinear}:

	\begin{proof}[Proof of Theorem \ref{T: quasilinear}]  By Proposition \ref{P: homogenization irrational directions} and Theorem \ref{T: irrational directions theorem}, the half-relaxed limits $\bar{u}^{*}$ and $\bar{u}_{*}$ of $(u^{\epsilon})_{\epsilon > 0}$ are, respectively, sub- and supersolution of \eqref{E: effective equation} in the usual viscosity sense, and we can show that $\bar{u}^{*} \leq u_{0} \leq \bar{u}_{*}$ by constructing appropriate $\epsilon$-independent sub- and super-solutions of \eqref{E: quasilinear}.  Hence Theorem \ref{T: comparison} implies $\bar{u}^{*} \leq \bar{u}_{*}$, from which $\bar{u}^{*} = \bar{u}_{*}$ necessarily follows.  Denoting this function by $\bar{u}$, we see that it is a continuous viscosity solution of \eqref{E: effective equation}, necessarily unique, and $u^{\epsilon} \to \bar{u}$.      \end{proof}  
	
A last remark is in order:

	\begin{remark}  Since $\overline{m}_{\text{pl}}$ can be discontinuous, Theorem \ref{T: homogenization forced planes} shows that, in general, the limiting behavior of \eqref{E: quasilinear forced} with $\alpha \neq 0$ and arbitrary initial data $u_{0} \in UC(\mathbb{R}^{d})$ cannot be described by an effective equation with a comparison principle, unlike the $\alpha = 0$ case.  The reason is that comparison forces the solution map to be continuous with respect to the topology of local uniform convergence, whereas Theorem \ref{T: homogenization forced planes} shows that any solution map, if well-defined, has to act discontinuously on linear functions when $\overline{m}_{\text{pl}}$ is discontinuous.  It remains to be seen what can be said about these problems.  \end{remark}  

\subsection{Proof of Theorem \ref{T: comparison}}  As in \cite{ishii level set}, the proof of Theorem \ref{T: comparison} proceeds by replacing the Euclidean norm by some other Finsler norm in a variable doubling argument.  (Recall that $\psi : \mathbb{R}^{d} \to [0,\infty)$ is a Finsler norm if it is convex, positively one-homogeneous, and positive away from zero.)  To improve the result of \cite{ishii level set} from finitely many discontinuity points to our setting, we use the following fact: 

	\begin{prop} \label{P: norm}  There is a universal constant $c_{0} > 0$ such that if $\{e_{n}\}_{n \in \mathbb{N}} \subseteq S^{d-1}$, then, for each $N \in \mathbb{N}$, there is a Finsler norm $\psi_{N} \in C^{2}(\mathbb{R}^{d} \setminus \{0\})$ such that
		\begin{equation*}
			1 \leq \psi_{N}(e) \leq \frac{5}{4}, \quad D^{2} \psi_{N}(e) \leq c_{0} (\text{Id} - e \otimes e) \quad \text{if} \, \, e \in S^{d-1}
		\end{equation*}
	and the following property holds: given $p \in \mathbb{R}^{d} \setminus \{0\}$, if $\widehat{D\psi}_{N}(p) = e_{i}$ for some $i \in \{1,2,\dots,N\}$, then $D^{2} \psi_{N}(p) = 0$.  
	\end{prop}  
	
Before proving Proposition \ref{P: norm}, which follows from a relatively simple but tedious geometric construction, let us see how it implies Theorem \ref{T: comparison}.

	\begin{proof}[Proof of Theorem \ref{T: comparison}]  By (i), if $\varphi : \mathbb{R} \to \mathbb{R}$ is a smooth, non-decreasing function, then $\varphi(u)$ remains a subsolution and $\varphi(v)$, a supersolution.  Therefore, we can assume that $u$ and $v$ are bounded.
	
	We argue by contradiction, assuming that the following inequality holds:
		\begin{equation*}
			\sup \left\{ u(x,t) - v(x,t) \, \mid \, (x,t) \in \mathbb{R}^{d} \times (0,T] \right\} > 0.
		\end{equation*}
	It follows that we can fix $\sigma > 0$ small enough that 
		\begin{equation} \label{E: tip up}
			\sup \left\{ u(x,t) - v(x,t) - \sigma t \, \mid \, (x,t) \in \mathbb{R}^{d} \times (0,T] \right\} > 0.
		\end{equation}
		
	Let $\zeta, \beta, \gamma > 0$ be free variables.  By (v), we can fix an $N = N(\gamma) \in \mathbb{N}$ such that 
		\begin{equation*}
			\sup \left\{ G^{*}(e_{i}, X) - G_{*}(e_{i}, X) \, \mid \, i \in \mathbb{N} \setminus \{1,2,\dots,N\} \right\} \leq \gamma (1 + \|X\|) \quad \text{if} \, \, X \in \mathcal{S}_{d}.
		\end{equation*}  
		
	Letting $\psi_{N}$ be the Finsler norm of Proposition \ref{P: norm} with $\{e_{n}\}_{n \in \mathbb{N}}$ the set of ``bad" directions associated with the pair $(G^{*},G_{*})$, define $\Phi = \Phi_{\zeta,\beta}: \mathbb{R}^{d} \times \mathbb{R}^{d} \times [0,T] \to \mathbb{R}$ by 
		\begin{equation*}
			\Phi(x,y,t) = u(x,t) - v(y,t) - \frac{\psi_{N}(x - y)^{4}}{4 \zeta} - \frac{\beta}{2} \|y\|^{2} - \sigma t.
		\end{equation*}
	Since $u$ and $v$ are bounded, $\Phi$ is bounded above and attains its maximum in $\mathbb{R}^{d} \times \mathbb{R}^{d} \times [0,T]$.  Let $(\bar{x},\bar{y},\bar{t}) = (\bar{x}_{\zeta,\beta}, \bar{y}_{\zeta,\beta}, \bar{t}_{\zeta,\beta})$ be such a global maximum, that is,
		\begin{equation*}
			\Phi(\bar{x},\bar{y},\bar{t}) = \max \left\{ \Phi(x,y,t) \, \mid \, (x,y,t) \in \mathbb{R}^{d} \times \mathbb{R}^{d} \times [0,T] \right\}.
		\end{equation*}
	By the boundedness of $u$ and $v$ and the lower bound $\psi_{N} \geq \|\cdot\|$, there is a $\gamma$-independent constant $C > 0$ such that
		\begin{equation} \label{E: trivial bound}
			\sup \left\{ \frac{\beta \|\bar{y}_{\zeta,\beta}\|^{2}}{2} + \frac{\|\bar{x}_{\zeta,\beta} - \bar{y}_{\zeta,\beta}\|^{4}}{\zeta} \, \mid \, (\zeta,\beta) \in (0,\infty) \times (0,\infty) \right\} \leq C.
		\end{equation}
		
	In view of \eqref{E: tip up} and the assumptions on $u$ and $v$, there are constants $\zeta_{0}, \beta_{0} > 0$ such that $\bar{t}_{\zeta,\beta} > 0$ for all $(\zeta,\beta) \in (0,\zeta_{0}) \times (0,\beta_{0})$ and all $\gamma \in (0,1)$.  Henceforth let $(\zeta,\beta) \in (0,\zeta_{0}) \times (0,\beta_{0})$ and assume $\gamma < 1$.
	
	Since $\bar{t} > 0$, we can invoke the maximum principle for semi-continuous functions, \cite[Lemma 1]{ishii level set}, and the equations satisfied by $u$ and $v$.  This gives matrices $X, Y \in \mathcal{S}_{d}$ and numbers $a, b \in \mathbb{R}$ so that if $\bar{A} = \bar{A}(\bar{x} - \bar{y}) \in \mathcal{S}_{d}$ and $\bar{p} = \bar{p}(\bar{x} - \bar{y}) \in \mathbb{R}^{d}$ are defined by 
		\begin{align*}
			\bar{p}(\bar{x} - \bar{y}) &= \zeta^{-1} \psi_{N}(\bar{x} - \bar{y})^{3} D\psi_{N}(\bar{x} - \bar{y}), \\
			\bar{A}(\bar{x} - \bar{y}) &= \left\{ \begin{array}{r l}
				\zeta^{-1} \psi_{N}(\bar{x} - \bar{y})^{3} D^{2} \psi_{N}(\bar{x} - \bar{y}) + 3 \zeta^{-1} \psi_{N}(\bar{x} - \bar{y})^{2} D\psi_{N}(\bar{x} - \bar{y})^{\otimes 2}, & \text{if} \, \, \bar{x} \neq \bar{y}, \\
				0, & \text{otherwise},
			\end{array} \right.
		\end{align*}
	then
		\begin{align*}
			\sigma = a - b, \quad -3 \left( \begin{array}{c c}
										\bar{A} & 0 \\
										0 & \bar{A} 
									\end{array} \right) &\leq \left( \begin{array}{c c}
										X & 0 \\
										0 & -(Y + \beta \text{Id}) 
									\end{array} \right) \leq 3 \left( \begin{array}{c c}
	\bar{A} & -\bar{A} \\
	-\bar{A} & \bar{A} \end{array} \right), \\
			a - G^{*}(\bar{p}, X) \leq 0, &\quad b - G_{*}(\bar{p} - \beta \bar{y}, Y) \geq 0.
		\end{align*}  
	Note, in addition, that Proposition \ref{P: norm} and \eqref{E: trivial bound} yield the following $\beta$-independent estimates on $\|\bar{p}\|$ and $\|\bar{A}\|$:
		\begin{equation} \label{E: important estimate for uniqueness}
			\|\bar{p}\| \leq \frac{5}{4} C^{\frac{3}{4}} \zeta^{-\frac{1}{4}}, \quad \|\bar{A}\| \leq \sqrt{C} \zeta^{-\frac{1}{2}}.
		\end{equation}
	Hence, we can send $\beta \to 0^{+}$ and invoke \eqref{E: trivial bound} to obtain $\xi \in \mathbb{R}^{d}$,  $\tilde{p} = \bar{p}(\xi) \in \mathbb{R}^{d}$, $\tilde{A} = \bar{A}(\xi) \in \mathcal{S}_{d}$, and $\tilde{X}, \tilde{Y} \in \mathcal{S}_{d}$ such that
		\begin{equation*}
			\sigma + G_{*}(\tilde{p},\tilde{Y}) - G^{*}(\tilde{p},\tilde{X}) \leq 0, \quad -3 \left( \begin{array}{c c}
										\tilde{A} & 0 \\
										0 & \tilde{A} 
									\end{array} \right) \leq \left( \begin{array}{c c}
										\tilde{X} & 0 \\
										0 & -\tilde{Y} 
									\end{array} \right) \leq 3 \left( \begin{array}{c c}
	\tilde{A} & -\tilde{A} \\
	-\tilde{A} & \tilde{A} \end{array} \right).
		\end{equation*}

There are four cases left to consider: (i) $\xi = 0$, (ii) $\widehat{D\psi}_{N}(\xi) \in \{e_{1},\dots,e_{N}\}$, (iii) $\widehat{D\psi}_{N}(\xi) \in \{e_{N + 1},e_{N+ 2},\dots\}$, and (iv) $\widehat{D\psi}_{N}(\xi) \in S^{d-1} \setminus \mathbb{R} \mathbb{Z}^{d}$.  

\textbf{Case (i): $\xi = 0$}  

	In this case, we have $\tilde{p} = 0$ and $\tilde{A} = 0$, hence $\tilde{X} = 0$ and $\tilde{Y} = 0$.  This yields the estimate
		\begin{equation} \label{E: case i}
			\sigma \leq \sigma + G_{*}(0,0) - G^{*}(0,0) \leq 0.
		\end{equation}
		
	\textbf{Case (ii): $\widehat{D\psi}_{N}(\xi) \in \{e_{1},e_{2},\dots,e_{N}\}$}
	
		In this case, Proposition \ref{P: norm} implies that $D^{2} \psi_{N}(\xi) = 0$.  Thus, $\tilde{A} = c e_{i} \otimes e_{i}$ for some $c > 0$ and $\|\tilde{p}\|^{-1} \tilde{p} = e_{i}$ so (i) and (vi) give
			\begin{equation*}
				G^{*}(\tilde{p},X) \leq G^{*}(\tilde{p}, 3c e_{i} \otimes e_{i}) = 0, \quad G_{*}(\tilde{p},\tilde{Y}) \geq G_{*}(\tilde{p},-3ce_{i} \otimes e_{i}) = 0.
			\end{equation*}
		Combining these estimates, we obtain
			\begin{equation} \label{E: case ii}
				\sigma \leq \sigma + G_{*}(\tilde{p},\tilde{Y}) - G^{*}(\tilde{p},\tilde{X}) \leq 0.
			\end{equation}
			
	\textbf{Case (iii): $\widehat{D\psi}_{N}(\xi) \in \{e_{N+1},e_{N+2},\dots\}$}  
	
	By the choice of $N$,
		\begin{equation*}
			G_{*}(\tilde{p},\tilde{Y}) \geq G^{*}(\tilde{p},\tilde{Y}) - \gamma \|\tilde{Y}\| - \gamma \|\tilde{p}\| \geq G^{*}(\tilde{p},\tilde{Y}) - \gamma \|\tilde{p}\| - 3 \gamma \|\tilde{A}\|.
		\end{equation*}
	From this and the inequality $\tilde{X} \leq \tilde{Y}$, we find
		\begin{equation} \label{E: case iii}
			\sigma - \gamma \|\tilde{p}\| - 3 \gamma \|\tilde{A}\| \leq \sigma + G_{*}(\tilde{p},\tilde{Y}) - G^{*}(\tilde{p},\tilde{X}) \leq 0.
		\end{equation}
		
	\textbf{Case (iv): $\widehat{D \psi}_{N}(\xi) \in S^{d-1} \setminus \mathbb{R} \mathbb{Z}^{d}$}  
	
	Here $G^{*}(\tilde{p},\tilde{X}) = G_{*}(\tilde{p},\tilde{X})$ and, thus, 
		\begin{equation*}
			G_{*}(\tilde{p},\tilde{Y}) - G^{*}(\tilde{p},\tilde{X}) = G_{*}(\tilde{p},\tilde{Y}) - G_{*}(\tilde{p},\tilde{X}) \geq 0.
		\end{equation*}
	 This gives our last estimate:
	 	\begin{equation} \label{E: case iv}
			\sigma \leq \sigma + G_{*}(\tilde{p},\tilde{Y}) - G^{*}(\tilde{p},\tilde{X}) \leq 0.
		\end{equation}
		
	Combining \eqref{E: case i}, \eqref{E: case ii}, \eqref{E: case iii}, and \eqref{E: case iv}, we conclude that
		\begin{equation} \label{E: final estimate}
			\sigma \leq (\|\tilde{p}\| + 3 \|\tilde{A}\|) \gamma.
		\end{equation}
	However, in view of \eqref{E: important estimate for uniqueness}, this is a contradiction as soon as $\gamma$ is small enough compared to $\zeta$ and $\sigma$.
	\end{proof}

\section{Effective Mobility as Linear Response} \label{S: linear response}

In this section, we prove Theorems \ref{T: homogenization forced planes} and \ref{T: linear response formula} and their corollaries.  As already indicated in the introduction, the proofs are relatively routine when $e \notin \mathbb{R} \mathbb{Z}^{d}$.  We will start by motivating the modifications that are necessary in the case that $e \in \mathbb{R} \mathbb{Z}^{d}$.

\subsection{Strategy of proof}  As already mentioned in Section \ref{S: main results}, the analysis of the forced problem \eqref{E: quasilinear forced} requires an additional correction.  Fix $e \in \mathbb{R} \mathbb{Z}^{d}$, $\alpha \in \mathbb{R} \setminus \{0\}$, and let $(u^{\epsilon}_{e})_{\epsilon > 0}$ be the solutions of \eqref{E: quasilinear forced}.  Proceeding as in the proof of Theorems \ref{T: level set PDE} and \ref{T: quasilinear}, we seek to identify the limit of $(u^{\epsilon}_{e})_{\epsilon > 0}$ by showing that the functions $\bar{u}^{*}_{e}$ and $\bar{u}_{*,e}$ are sub- and supersolutions of an effective equation.  

Let us see what can be deduced na\"{i}vely about $\bar{u}^{*}_{e}$ using the perturbed test function method.  As in Section \ref{S: rational correctors}, we let $\tilde{V}_{e}$ be the solution of the following cell problem:
	\begin{equation} \label{E: oscillating cell problem}
		m(y,e) - m_{e}^{\perp}(y) - \text{tr} \left(A(y,e) D^{2}\tilde{V}_{e} \right) = 0 \quad \text{in} \, \, \mathbb{T}^{d}.
	\end{equation} 
Suppose that $(x_{0},t_{0}) \in \mathbb{R}^{d} \times (0,\infty)$, $\varphi$ is smooth, and $\bar{u}^{*}_{e} - \varphi$ has a strict local maximum at $(x_{0},t_{0})$.  Since we expect that $\bar{u}^{*}_{e}(x,t) = \langle x,e \rangle + \alpha \overline{m}_{pl}(e) t$ for some $\overline{m}_{pl}(e) > 0$, we may as well assume that $\widehat{D\varphi}(x_{0},t_{0}) = e$.  

With $\tilde{V}_{e}$ a solution of \eqref{E: oscillating cell problem}, define the perturbed test function $\varphi^{\epsilon}$ by 
	\begin{equation*}
		\varphi^{\epsilon}(x,t) = \varphi(x,t) + \epsilon^{2} \varphi_{t}(x_{0},t_{0}) \tilde{V}_{e}(\epsilon^{-1} x).
	\end{equation*}
Letting $(x_{\epsilon},t_{\epsilon})$ be a local maximum of $u^{\epsilon}_{e} - \varphi^{\epsilon}$ close enough to $(x_{0},t_{0})$, we know that $(x_{\epsilon},t_{\epsilon}) \to (x_{0},t_{0})$ (along a subsequence) and, thus, we can invoke the equation satisfied by $u^{\epsilon}_{e}$, which gives, after some simplification,
	\begin{equation*}
		\limsup_{\epsilon \to 0^{+}} \left[ m_{e}^{\perp}(\epsilon^{-1}\langle x_{\epsilon}, e \rangle) \varphi_{t}^{\epsilon}(x_{\epsilon},t_{\epsilon}) - \alpha \|D\varphi^{\epsilon}(x_{\epsilon},t_{\epsilon})\| \right] \leq 0.
	\end{equation*}
This shows that to homogenize $(u^{\epsilon}_{e})_{\epsilon > 0}$, it is not enough to treat the transversal fluctuations of the front using $\tilde{V}_{e}$; the position of the front along the $e$ axis continues to oscillate.  Nonetheless, we have lost $(d - 1)$-degrees of freedom by averaging, and the equation obtained so far suggests that we should be able to proceed by exploiting the asymptotic behavior of the averaged, one-dimension equation:
	\begin{equation} \label{E: averaged equation}
		\left\{ \begin{array}{r l}
			m_{e}^{\perp}(s e) \mathcal{U}_{t} - \alpha |\mathcal{U}_{s}| = 0 & \text{in} \, \, \mathbb{R} \times (0,\infty), \\
			\mathcal{U}(s,0) = s & \text{if} \, \, s \in \mathbb{R}.
		\end{array} \right.
	\end{equation}
This is precisely the approach taken in what follows.


\subsection{Intermediate results}  Let us state precisely some of the results used in the proof of Theorem \ref{T: homogenization forced planes}.  

First, we construct traveling wave sub- and supersolutions of \eqref{E: quasilinear forced} in rational directions:

\begin{prop} \label{P: rational traveling}  Under the same assumptions as in Theorems \ref{T: quasilinear} and \ref{T: homogenization forced planes}, if $e \in \mathbb{R} \mathbb{Z}^{d}$ and $\alpha \in \mathbb{R} \setminus \{0\}$, then:
		\begin{itemize}
			\item[(i)] The cell problem \eqref{E: oscillating cell problem} has a solution $\tilde{V}_{e} \in C^{2, \alpha}(\mathbb{T}^{d})$ and $m^{\perp}_{e} \in C^{2,\alpha}(\mathbb{T}^{d})$.
			\item[(ii)] There is a $C^{2}$, $r_{e}$-periodic function $P_{e} : \mathbb{R} \to \mathbb{R}$ and a constant $\overline{m}_{pl}(e)$ such that
				\begin{equation*}
					\overline{m}_{pl}(e)^{-1} m_{e}^{\perp}(se) - |1 + P_{e,s}| = 0 \quad \text{in} \, \, \mathbb{R}.
				\end{equation*}
			In particular, for each $\alpha \in \mathbb{R} \setminus \{0\}$, the function $(s,t) \mapsto s + P_{e}(s) +\alpha \overline{m}_{pl}(e)^{-1} t$ is a pulsating wave solution of \eqref{E: averaged equation}.  Moreover, $\overline{m}_{pl}(e)$ is given explicitly by \eqref{E: linear response}.  
			\item[(iii)]  For each $\alpha \in \mathbb{R} \setminus \{0\}$, there is an $\epsilon_{0}^{+} > 0$ and a family $(\alpha^{+}_{\epsilon})_{\epsilon \in (0,\epsilon_{0}^{+})}$ such that the functions $(u^{+,\epsilon})_{\epsilon \in (0,\epsilon_{0}^{+})}$ defined by
				\begin{equation*}
					u^{+,\epsilon}(x,t) = \langle x,e \rangle + \epsilon P_{e}(\epsilon^{-1} \langle x,e \rangle) + \alpha_{\epsilon}^{+} \overline{m}_{pl}(e)^{-1}  \left( \epsilon^{2} \tilde{V}_{e}(\epsilon^{-1} x) + t\right)
				\end{equation*}
			are super-solutions of \eqref{E: quasilinear forced}.  Furthermore, there is a constant $C_{+} > 0$ depending only on $\tilde{V}_{e}$ and $P_{e}$ such that, for each $(x,t) \in \mathbb{R}^{d} \times \mathbb{R}$,
				\begin{align*}
					|\alpha^{+}_{\epsilon} - \alpha| &\leq C_{+} \epsilon \\
					|u^{+,\epsilon}(x,t) - \langle x,e \rangle - \alpha \overline{m}(e)^{-1} t| &\leq C_{+} \epsilon(1 + |t|)
				\end{align*}
			\item[(iv)]  For each $\alpha \in \mathbb{R} \setminus \{0\}$, there is an $\epsilon_{0}^{-} > 0$ and a family $(\alpha^{-}_{\epsilon})_{\epsilon \in (0,\epsilon_{0}^{-})}$ such that the functions $(u^{-,\epsilon})_{\epsilon \in (0,\epsilon_{0}^{-})}$ defined by
				\begin{equation*}
					u^{-,\epsilon}(x,t) = \langle x,e \rangle + \epsilon P_{e}(\epsilon^{-1} \langle x,e \rangle) + \alpha_{\epsilon}^{-} \overline{m}_{pl}(e)^{-1} \left(\epsilon^{2}  \tilde{V}_{e}(\epsilon^{-1} x) + t \right)
				\end{equation*}
			are sub-solutions of \eqref{E: quasilinear forced} in $\mathbb{R}^{d} \times \mathbb{R}$.  Furthermore, there is a constant $C_{-} > 0$ depending only on $\tilde{V}_{e}$ and $P_{e}$ such that, for each $(x,t) \in \mathbb{R}^{d} \times \mathbb{R}$,
				\begin{align*}
					|\alpha^{-}_{\epsilon} - \alpha| &\leq C_{-} \epsilon \\
					|u^{-,\epsilon}(x,t) - \langle x,e \rangle - \alpha \overline{m}(e)^{-1} t| &\leq C_{-} \epsilon(1 + |t|)
				\end{align*}
		\end{itemize}
	\end{prop}  

A similar construction works when $e \notin \mathbb{R} \mathbb{Z}^{d}$ and \eqref{E: oscillating cell problem} has a $C^{2}$ solution.  Here is the result in that case:

	\begin{prop} \label{P: irrational traveling} If there is a solution $V_{e} \in C^{2}(\mathbb{T}^{d})$ of \eqref{E: oscillating cell problem} and $e \in S^{d-1} \setminus \mathbb{R} \mathbb{Z}^{d}$, then there is an $\epsilon_{0} > 0$ and families $(\alpha_{\epsilon}^{+})_{\epsilon \in (0,\epsilon_{0})}$ and $(\alpha_{\epsilon}^{-})_{\epsilon \in (0,\epsilon_{0})}$ such that the functions $(u^{+,\epsilon})_{\epsilon \in (0,\epsilon_{0})}$ and $(u^{-,\epsilon})_{\epsilon \in (0,\epsilon_{0})}$ defined by 
		\begin{equation*}
			u^{\pm,\epsilon}(x,t) = \langle x,e \rangle + \alpha_{\epsilon}^{\pm} \overline{m}(e)^{-1} \left( \epsilon^{2} V_{e}(\epsilon^{-1} x) + t \right)
		\end{equation*} 
	are, respectively, super- and sub-solutions of \eqref{E: quasilinear forced} in $\mathbb{R}^{d} \times \mathbb{R}$.  Furthermore, there is a $C > 0$ such that, for each $(x,t) \in \mathbb{R}^{d} \times \mathbb{R}$,
		\begin{align*}
			|\alpha^{+}_{\epsilon} - \alpha| + |\alpha^{-}_{\epsilon} - \alpha| &\leq C \epsilon \\
			|u^{+,\epsilon}(x,t) - \langle x,e \rangle - \alpha \overline{m}(e)^{-1} t| + |u^{-,\epsilon}(x,t) - \langle x,e \rangle - \alpha \overline{m}(e)^{-1} t| &\leq C \epsilon (1 + |t|)
		\end{align*}  \end{prop}
		
	Since the proof of Proposition \ref{P: irrational traveling} is effectively a simpler version of that of Proposition \ref{P: rational traveling}, we omit it.  In general, it is far from clear whether or not \eqref{E: oscillating cell problem} has a solution when $e \notin \mathbb{R} \mathbb{Z}^{d}$ due to the loss of compactness; cases where this is possible are discussed in Appendix \ref{A: technical lemmata}.   
	
	When applicable, Propositions \ref{P: rational traveling} and \ref{P: irrational traveling} allow us to quantify the convergence in Theorem \ref{T: homogenization forced planes}.  More precisely, we have
	
		\begin{prop} \label{P: good estimates} If $e \in S^{d-1} \cap \mathbb{R} \mathbb{Z}^{d}$ or there is a $V_{e} \in C^{2}(\mathbb{T}^{d})$ solving \eqref{E: cell problem}, then there is an $\epsilon_{0} \in (0,1)$ and a $C_{e} > 0$ such that if $\epsilon \in (0,\epsilon_{0})$ and $u_{e}^{\epsilon}$ solves \eqref{E: quasilinear forced} with $u_{e}^{\epsilon}(x,0) = \langle x,e \rangle$, then
		\begin{equation*}
			|u_{e}^{\epsilon}(x,t) - \langle x,e \rangle - \alpha \overline{m}_{pl}(e)^{-1} t| \leq C_{e} \epsilon (1 + t) \quad \text{if} \, \, (x,t) \in \mathbb{R}^{d} \times (0,\infty)
		\end{equation*}
	\end{prop}  
	
	Since the proof of Proposition \ref{P: good estimates} follows directly from Propositions \ref{P: rational traveling} and \ref{P: irrational traveling} and the comparison principle, we omit it.
	
	It only remains to consider the case when $e \notin \mathbb{R} \mathbb{Z}^{d}$ yet the hypotheses of Proposition \ref{P: irrational traveling} fail.  Using approximate correctors, it is still possible to construct traveling wave sub- and supersolutions and use these to give a qualitative proof of Theorem \ref{T: homogenization forced planes}.  This is briefly treated in Section \ref{S: forced planes weak proof} below.  
	
\subsection{Proof of Proposition \ref{P: rational traveling}}  Most of the work in proving the proposition lies in (iii) and (iv) since (i) is Proposition \ref{P: rational correctors} and (ii) is an exercise in ODE theory.  More precisely, (ii) is covered by the following lemma, the proof of which is left to the interested reader:

	\begin{lemma} \label{L: rotation number} Fix $r > 0$.  If $\tilde{m} : r \mathbb{T} \to (0,\infty)$ is $C^{1}$, then the solution $X^{s,\tilde{m}}$ of the one-dimensional ODE
		\begin{equation*}
			\left\{ \begin{array}{r l}
				\tilde{m}(X^{s,\tilde{m}}_{t}) \dot{X}^{s,\tilde{m}}_{t} = 1, \\
				X^{s,\tilde{m}}_{0} = s,
			\end{array} \right.
		\end{equation*}
	satisfies 
		\begin{equation*}
			\lim_{t \to \infty} \frac{X^{s,\tilde{m}}_{t} - s}{t} = \frac{1}{\fint_{0}^{r} \tilde{m}(u) \, du}.
		\end{equation*}
	Furthermore, the function $\tilde{P} : r \mathbb{T} \to \mathbb{R}$ given by 
		\begin{equation*}
			\tilde{P}(s) = \left( \fint_{0}^{r_{e}} \tilde{m}(u) \, du \right)^{-1} \int_{0}^{s} \tilde{m}(s) \, ds - 1
		\end{equation*}
	solves the pulsating wave equation
		\begin{equation*}
			\frac{\tilde{m}(s)}{\fint_{0}^{r} \tilde{m}(u) \, du} - |1 + P'(s)| = 0 \quad \text{in} \, \, r \mathbb{T}.
		\end{equation*} \end{lemma}    
		
	Now we proceed with the

	\begin{proof}[Proof of Proposition \ref{P: rational traveling}]  (i) follows from Proposition \ref{P: rational correctors} above, and (ii) is an application of Lemma \ref{L: rotation number}.  We will only prove (iii) since (iv) follows similarly.  
	
	Let $\epsilon > 0$ and $\alpha_{\epsilon}^{+} \in \mathbb{R}$ be free variables for the moment and define $u^{+,\epsilon} : \mathbb{R}^{d} \times \mathbb{R} \to \mathbb{R}$ by 
		\begin{equation} 
			u^{+,\epsilon}(x,t) = \langle x,e \rangle + \epsilon P_{e}(\epsilon^{-1} \langle x,e \rangle) + \alpha_{\epsilon} \overline{m}_{pl}(e)^{-1}\left( \epsilon^{2} \tilde{V}_{e}(\epsilon^{-1} x) + t\right).
		\end{equation}
	We will show that if $\epsilon > 0$ is small enough and $\alpha_{\epsilon}^{+} = \alpha + C_{0} \epsilon$ for some large enough constant $C_{0} > 0$, then $u^{+,\epsilon}$ is a super-solution of \eqref{E: quasilinear forced}.
	
	Let us study the equation for $u^{+,\epsilon}$ term by term.  First, the term with the time derivative.  To declutter the notation, we define $p_{\epsilon} : \mathbb{R}^{d} \times \mathbb{R} \to \mathbb{R}^{d}$ by 
		\begin{equation*}
			p_{\epsilon}(x,t) = Du^{+,\epsilon}(x,t) = (1 + P'_{e}(\epsilon^{-1} \langle x,e \rangle)) e + \epsilon \alpha_{\epsilon}^{+} \overline{m}_{pl}(e)^{-1} D\tilde{V}_{e}(\epsilon^{-1} x).
		\end{equation*}
	Observe that since $1 + P'_{e} > 0$, it follows that $\hat{p_{\epsilon}} = e + O(\epsilon)$.   Thus,
		\begin{align*}
			m(\epsilon^{-1} x, \widehat{Du^{+,\epsilon}}) u^{+,\epsilon}_{t} = \alpha_{\epsilon}^{+} \overline{m}_{pl}(e)^{-1} m(\epsilon^{-1} x, e) + O(\epsilon).
		\end{align*}
		
	Next, the curvature term.  First, notice that if we write $\mathcal{N}(p) = \hat{p} \otimes \hat{p}$, then, by Taylor expansion,
		\begin{align*}
			\text{tr} \left(A(\epsilon^{-1} x, \widehat{Du^{+,\epsilon}}) D^{2}u^{+,\epsilon} \right) &= \text{tr} \left(a(\epsilon^{-1} x, p_{\epsilon}) (\text{Id} - \hat{p}_{\epsilon} \otimes \hat{p}_{\epsilon}) D^{2} u^{+,\epsilon} (\text{Id} - \hat{p}_{\epsilon} \otimes \hat{p}_{\epsilon}) \right) \\
				&= \text{tr} \left( a(\epsilon^{-1} x, e) (\text{Id} - e \otimes e) D^{2}u^{+,\epsilon} (\text{Id} - e \otimes e) \right) \\
				&\quad + \text{tr} \left(D_{p}a(\epsilon^{-1}x,e)[p_{\epsilon} - e] (\text{Id} - e \otimes e) D^{2}u^{+,\epsilon} (\text{Id} - e \otimes e) \right) \\
				&\quad + \text{tr} \left( a(\epsilon^{-1} x,e) D_{p}\mathcal{N}(e)[p_{\epsilon} - e] D^{2}u^{+,\epsilon} (\text{Id} - e \otimes e)\right) \\
				&\quad + \text{tr} \left( a(\epsilon^{-1} x, e) (\text{Id} - e \otimes e) D^{2}u^{+,\epsilon} D_{p}\mathcal{N}(e)[p_{\epsilon} - e] \right) \\
				&\quad + \|D^{2}u^{+,\epsilon}\|_{L^{\infty}(\mathbb{T}^{d})} O(\|\hat{p}_{\epsilon} - e\|^{2})
		\end{align*}
	Next, observe that $D^{2}u^{+,\epsilon} = \epsilon^{-1} P_{e}'' e \otimes e + \alpha_{\epsilon} \overline{m}_{pl}(e)^{-1} D^{2} \tilde{V}_{e}$ and, thus, in each of the first four terms in the previous expressions, we can replace $D^{2}u^{+,\epsilon}$ by $\alpha_{\epsilon}^{+} \overline{m}_{pl}(e)^{-1} D^{2} \tilde{V}_{e}$.  This leads to
		\begin{align*}
			\text{tr} \left(A(\epsilon^{-1} x, \widehat{Du^{+,\epsilon}}) D^{2}u^{+,\epsilon} \right) &= \alpha_{\epsilon}^{+} \overline{m}_{pl}(e)^{-1} \text{tr} \left( A(\epsilon^{-1} x, e) D^{2} \tilde{V}_{e} \right)  + \|D^{2}\tilde{V}_{e}\|_{L^{\infty}(\mathbb{T}^{d})} O(\|\hat{p}_{\epsilon} - e\|) \\
				&\quad + \|D^{2}u^{+,\epsilon}\|_{L^{\infty}(\mathbb{T}^{d})} O(\|\hat{p}_{\epsilon} - e\|^{2})
		\end{align*}
	Since $\|D^{2}u^{+,\epsilon}\| \leq C\epsilon^{-1}$ and $\|\hat{p}_{\epsilon} - e\| = O(\epsilon)$, we conclude
		\begin{equation*}
			\text{tr} \left(A(\epsilon^{-1} x, \widehat{Du^{+,\epsilon}}) D^{2}u^{+,\epsilon} \right) = \alpha_{\epsilon}^{+} \overline{m}_{pl}(e)^{-1} \text{tr} \left( A(\epsilon^{-1} x, e) D^{2} V_{e} \right) + O(\epsilon).
		\end{equation*}
		
	Finally, we treat the first-order term.  Since $\|e\| = 1$, we have
		\begin{align*}
			\alpha \|Du^{+,\epsilon}\| &= \alpha |1 + P_{e}'| + O(\epsilon).
		\end{align*}
		
	Putting it all together and invoking the equations satisfied by $\tilde{V}_{\epsilon}$ and $P_{e}$, we find, for some $\overline{C} > 0$,
		\begin{align*}
			&m(\epsilon^{-1} x, \widehat{Du^{+,\epsilon}}) u^{+,\epsilon}_{t} - \text{tr} \left(A(\epsilon^{-1} x, \widehat{Du^{+,\epsilon}}) D^{2}u^{+,\epsilon} \right) - \alpha \|Du^{+,\epsilon}\| \geq \\
			&\qquad \alpha_{\epsilon}^{+}\overline{m}_{pl}(e)^{-1} \left( m(\epsilon^{-1} x, e) - \text{tr} \left( A(\epsilon^{-1} x, e) D^{2}\tilde{V}_{e} \right) \right) - \alpha |1 + P_{e}'| - (\overline{C} + o(1))\epsilon = \\
			&\qquad \alpha_{\epsilon}^{+} \overline{m}_{pl}(e)^{-1} m_{e}^{\perp}(\epsilon^{-1} \langle x, e \rangle) - \alpha \overline{m}_{pl}(e)^{-1} m_{e}^{\perp}(\epsilon^{-1} \langle x,e \rangle) - (\overline{C} + o(1)) \epsilon \\
			&\qquad =(C_{0} - C + o(1)) \epsilon \overline{m}_{pl}(e)^{-1} m_{e}^{\perp}(\epsilon^{-1} \langle x,e \rangle).
		\end{align*}
	Setting $C_{0} = C + 1$, we deduce that there is an $\epsilon_{0} > 0$ such that $u^{+,\epsilon}$ is a super-solution of \eqref{E: quasilinear forced} in $\mathbb{R}^{d} \times \mathbb{R}$.\end{proof}  
	
\subsection{Proofs of Theorems \ref{T: homogenization forced planes} and \ref{T: linear response formula}} \label{S: forced planes weak proof} Finally, we complete the proof of Theorem \ref{T: homogenization forced planes} and address Theorem \ref{T: linear response formula}.  

In view of Propositions \ref{P: rational traveling} and \ref{P: irrational traveling}, Theorem \ref{T: homogenization forced planes} is proved as soon as we analyze the case that $e \notin \mathbb{R} \mathbb{Z}^{d}$ and yet \eqref{E: oscillating cell problem} has no smooth solution.  Not surprisingly, we argue using approximate correctors; the argument is included only for the sake of completeness.

Theorem \ref{T: linear response formula} follows directly from Proposition \ref{P: rational traveling} when $e \in \mathbb{R} \mathbb{Z}^{d}$ and the proof of Theorem \ref{T: homogenization forced planes}, otherwise.

	\begin{proof}[Proof of Theorem \ref{T: homogenization forced planes}]  Assume $e \in S^{d-1} \setminus \mathbb{R} \mathbb{Z}^{d}$; otherwise, Proposition \ref{P: rational traveling} applies.  Given $\delta > 0$, let $V^{\delta}$ be the solution of \eqref{E: penalized cell problem linear} with $f(y) = -m(y,e)$.  By Schauder estimates, there is a constant $C > 0$ such that
		\begin{equation} \label{E: penalized corrector estimates}
			\delta \|V^{\delta}\|_{L^{\infty}(\mathbb{T}^{d})} + \delta^{2} \|DV^{\delta}\|_{L^{\infty}(\mathbb{T}^{d})} + \delta^{3} \|D^{2} V^{\delta}\|_{L^{\infty}(\mathbb{T}^{d})} \leq C.
		\end{equation}
		
	To start with, we claim that
		\begin{equation} \label{E: upper bound}
			\limsup \nolimits^{*} u_{e}^{\epsilon}(x,t) \leq \langle x,e \rangle + \alpha \overline{m}(e)^{-1} t.
		\end{equation}
	
	To see this, fix $\beta \in (\alpha,\infty)$, let $\epsilon \in (0,1)$, set $\delta(\epsilon) = \epsilon^{\frac{1}{4}}$, and define $v^{\epsilon} : \mathbb{R}^{d} \times \mathbb{R} \to \mathbb{R}$ by 
			\begin{equation*}
				v^{\epsilon}(x,t) = \langle x,e \rangle + \beta \overline{m}(e)^{-1} \left(t + \epsilon^{2} V^{\delta(\epsilon)}_{e} (\epsilon^{-1} x) \right).
			\end{equation*}
		We claim there is an $\epsilon_{0} > 0$ depending only on $m$ and $\beta$ such that $v^{\epsilon}$ is a super-solution of \eqref{E: quasilinear forced} if $\epsilon \in (0,\epsilon_{0})$.  
		
		Indeed, invoking \eqref{E: penalized corrector estimates}, we find
			\begin{align*}
				m(\epsilon^{-1} x, \widehat{Dv^{\epsilon}}) & v^{\epsilon}_{t} - \text{tr} \left( A(\epsilon^{-1}x,\widehat{Dv^{\epsilon}}) D^{2}v^{\epsilon} \right) - \alpha \|Dv^{\epsilon}\| \\
				&\quad = (\beta - \alpha) + \beta \overline{m}(e)^{-1}(-\delta V^{\delta}_{e}(\epsilon^{-1}x) - \overline{m}(e)) + O(\epsilon^{\frac{1}{2}}) \\
				&\quad = \beta - \alpha + o(1).
			\end{align*}
		Thus, there is an $\epsilon_{0} \in (0,1)$ such that if $\epsilon \in (0,\epsilon_{0})$, then
			\begin{equation*}
				m(\epsilon^{-1} x, \widehat{Dv^{\epsilon}})  v^{\epsilon}_{t} - \text{tr} \left( A(\epsilon^{-1}x,\widehat{Dv^{\epsilon}}) D^{2}v^{\epsilon} \right) - \alpha \|Dv^{\epsilon}\| \geq \frac{\beta - \alpha}{2} \quad \text{in} \, \, \mathbb{R}^{d} \times \mathbb{R}.
			\end{equation*}
			
		Now we prove \eqref{E: upper bound}.  First, notice that, by the choice of $\delta(\epsilon)$, 
			\begin{equation*}
				\langle x,e \rangle \leq v^{\epsilon}(x,0) + \beta \overline{m}(e)^{-1} \|\delta V^{\delta}_{e}\|_{L^{\infty}(\mathbb{T}^{d})} \epsilon^{\frac{7}{4}}
			\end{equation*}
		Thus, the comparison principle implies
			\begin{equation*}
				u^{\epsilon}(x,t) \leq v^{\epsilon}(x,t) + \beta \overline{m}(e)^{-1} \|m\|_{L^{\infty}(\mathbb{T}^{d})} \epsilon^{\frac{7}{4}}.
			\end{equation*}
		Sending $\epsilon \to 0^{+}$, we deduce that
			\begin{equation*}
				\limsup \nolimits^{*} u^{\epsilon}(x,t) \leq \langle x,e \rangle + \beta \overline{m}(e)^{-1} t.
			\end{equation*}
		At the same time, $\beta$ was an arbitrary number in $(\alpha,\infty)$.  Therefore, sending $\beta \to \alpha^{+}$, we recover \eqref{E: upper bound}.  
		
		Replacing $\beta \in (\alpha, \infty)$ by $\beta \in (-\infty,\alpha)$, we similarly prove that 
			\begin{equation*}
				\liminf \nolimits_{*} u^{\epsilon}(x,t) \geq \langle x,e \rangle + \alpha \overline{m}(e)^{-1} t.
			\end{equation*}  

	\end{proof} 
	
\subsection{Derivatives of Front Speeds}  This section treats the proof of Corollary \ref{C: derivatives}.  In the proof, we use the fact that there is a pulsating wave solution of \eqref{E: hyperbolic problem}.  This is constructed using a viscosity solution $P_{e,\alpha} \in C(\mathbb{T}^{d})$ of the equation
	\begin{equation*}
		\lambda_{e}(\alpha) m(y,\widehat{(e + DP_{e,\alpha})}) - \text{tr} \left(A(y,\widehat{(e + DP_{e,\alpha})}) D^{2}P_{e,\alpha}\right) - \alpha \|e + DP_{e,\alpha}\| = 0 \quad \text{in} \, \, \mathbb{T}^{d}.
	\end{equation*}
Since $d = 2$, the existence of such a function follows from \cite{caffarelli monneau}.  

		\begin{proof}[Proof of Corollary \ref{C: derivatives}]  As in the previous section, the proof is neater depending on whether or not \eqref{E: oscillating cell problem} has a smooth solution.  We only give the arguments in the case when $e \in S^{d-1} \setminus \mathbb{R} \mathbb{Z}^{d}$, but \eqref{E: oscillating cell problem} does not have a smooth solution.
		
		Fix $\beta \in (1,\infty)$.  Let $v^{\epsilon}(x,t) = \langle x,e \rangle + \beta \overline{m}(e)^{-1}(\epsilon^{2} V^{\delta(\epsilon)}_{e}(\epsilon^{-1} x) + t)$ for $\delta(\epsilon) = \epsilon^{\frac{1}{4}}$.  Arguing as in the proof of Theorem \ref{T: homogenization forced planes}, we see that there is an $\epsilon_{0} \in (0,1)$ such that $v^{\epsilon}$ is a super-solution of \eqref{E: quasilinear forced} with $\alpha = 1$ if $\epsilon \in (0,\epsilon_{0})$.
		
		Fix $P_{e,\epsilon}$ as above and set $u^{\epsilon}(y,t) = \langle y,e \rangle + P_{e,\epsilon}(y) + \lambda_{e}(\epsilon) t$.  Notice that this is a viscosity solution of \eqref{E: hyperbolic problem}.  Hence $(x,t) \mapsto \epsilon u^{\epsilon}(\epsilon^{-1} x,\epsilon^{-2} t)$ is a viscosity solution of the first equation in \eqref{E: quasilinear forced} with $\alpha = 1$.  Since $P_{e,\epsilon}$ is bounded,
			\begin{equation*}
				\epsilon u^{\epsilon}(\epsilon^{-1} x,0) \leq v^{\epsilon}(x,0) + \epsilon \|P_{e,\epsilon}\|_{L^{\infty}(\mathbb{T}^{d})}
			\end{equation*}
		and, thus, the comparison principle implies that, for each $(x,t) \in \mathbb{R}^{d} \times (0,\infty)$,
			\begin{align*}
				\langle x,e \rangle + \epsilon P_{e,\epsilon}(\epsilon^{-1} x) + \epsilon^{-1} \lambda(\epsilon) t &= \epsilon u^{\epsilon}(\epsilon^{-1} x, \epsilon^{-2} t) \leq v^{\epsilon}(x,t) + \epsilon \|P_{e,\epsilon}\|_{L^{\infty}(\mathbb{T}^{d})} \\
				&= \langle x,e \rangle + \beta \overline{m}(e)^{-1} \epsilon^{2}V_{e}^{\delta(\epsilon)}(\epsilon^{-1}x) + \epsilon \|P_{e,\epsilon}\|_{L^{\infty}(\mathbb{T}^{d})} \\
				&\quad + \beta \overline{m}(e)^{-1} t.
			\end{align*}
		Since $P_{e,\epsilon}$ is bounded and $t > 0$ is arbitrary, we deduce that $\epsilon^{-1} \lambda_{e}(\epsilon) \leq \beta \overline{m}(e)^{-1}$.
		Sending first $\epsilon \to 0^{+}$ and then $\beta \to 1^{+}$, we conclude
			\begin{equation*}
				\limsup_{\epsilon \to 0^{+}} \epsilon^{-1} \lambda_{e}(\epsilon) \leq \overline{m}(e)^{-1}.
			\end{equation*}
			
		Arguing using subsolutions instead of supersolutions, we find
			\begin{equation*}
				\liminf_{\epsilon \to 0^{+}} \epsilon^{-1} \lambda_{e}(\epsilon) \geq \overline{m}(e)^{-1}.
			\end{equation*}
		
		It remains to show that $\tilde{\alpha}^{-1} \lambda_{e}(\tilde{\alpha}) \to \overline{m}(e)^{-1}$ as $\tilde{\alpha} \to 0^{-}$.  Here we repeat the previous proof, replacing $\alpha = 1$ by $\alpha = -1$.  \end{proof}

\begin{remark}  When \eqref{E: oscillating cell problem} has a smooth solution, it is possible to prove the following rate: $|\epsilon^{-1} \lambda_{e}(\epsilon) - \overline{m}(e)^{-1}| \leq C_{e} \epsilon$ for some $C_{e} > 0$ depending on bounds on the derivatives of the corrector. \end{remark} 

\appendix

\section{Technical Lemmata}  \label{A: technical lemmata}

This appendix covers some of the technical results that were needed above.  

\subsection{Comparison and Well-posedness} \label{A: wellposed}  In this section, we briefly review some of the comparison principles and well-posedness results that were used above.  To start with, following \cite{caffarelli monneau}, we make assumption (v) of Theorem \ref{T: level set PDE} precise:

\begin{itemize}
	\item[(v)] (comparison)  There is a $K \geq 9$ and a modulus $\omega_{K} : [0,\infty) \to [0,\infty)$ satisfying $\lim_{\delta \to 0^{+}} \omega_{K}(\delta) = 0$ such that if $\alpha \geq 0$, $X,Y \in \mathcal{S}_{d}$, and $x,y \in \mathbb{R}^{d}$ satisfy
		\begin{equation*}
			- K \alpha \left( \begin{array}{c c}
								\text{Id} & 0 \\
								0 & \text{Id} 
						\end{array} \right) \leq \left( \begin{array}{c c}
												X & 0 \\
												0 & Y 
											\end{array} \right) \leq K \alpha \left( \begin{array}{c c}
				\text{Id} & -\text{Id} \\
				- \text{Id} & \text{Id} 
	\end{array} \right)
		\end{equation*}
	then
		\begin{equation*}
			F^{*}(\alpha (x - y), X, x) - F_{*}(\alpha(x - y), Y, y) \leq \omega_{K}(\|x - y\| (1 + \alpha \|x -y\|))
		\end{equation*}
	with the understanding that $\alpha = 0$ if $x = y$.
\end{itemize}  

\begin{prop} \label{P: level set wellposed} If $F$ satisfies the assumptions of Theorem \ref{T: level set PDE} or $F$ is given by 
	\begin{equation*}
		F(p,X,y) = m(y,\hat{p})^{-1} \text{tr}(a(y,\hat{p}) \tilde{X}_{\hat{p}})
	\end{equation*} 
with $a$ and $m$ satisfying the assumptions of Theorem \ref{T: quasilinear}, then, for each $u_{0} \in UC(\mathbb{R}^{d})$ and $\epsilon > 0$, there is a unique $u^{\epsilon}$ solving \eqref{E: level set PDE} or \eqref{E: quasilinear}, respectively. \end{prop}  

	\begin{proof}  In the case of \eqref{E: level set PDE} and Theorem \ref{T: level set PDE}, the assumptions of the theorem imply that \eqref{E: level set PDE} satisfies a comparison principle.  See \cite[Theorem 3.3]{caffarelli monneau}.  This guarantees uniqueness of the solution, and then existence follows through Perron's Method.
	
	The same idea pertains to \eqref{E: quasilinear} and Theorem \ref{T: quasilinear}.  The key point is the matrix $\sqrt{a}$ is uniformly Lipschitz continuous in the spatial variable and, thus, one can obtain the inequality in (v) by arguing as in \cite[Example 3.6]{users guide}.    \end{proof}  
	
Next, we turn to well-posedness of the penalized cell problem studied in Section \ref{S: approximate correctors}:

	\begin{prop} \label{P: cell problem wellposed} If $F$ satisfies the assumptions of Theorem \ref{T: level set PDE} or 
		\begin{equation*}
			F(p,X,y) = m(y,\hat{p})^{-1} \text{tr}(a(y,\hat{p})\tilde{X}_{\hat{p}})
		\end{equation*} 
	with $a$ and $m$ satisfying the assumptions of Theorem \ref{T: quasilinear}, then, for each $\delta > 0$, $e \in S^{d-1}$, and $X \in \mathcal{S}_{d}$, the penalized cell problem \eqref{E: penalized cell problem second instance} has a unique solution $V^{\delta} \in C(\mathbb{T}^{d})$.  \end{prop}  
	
		\begin{proof}  If $V_{1}$ and $V_{2}$ are respectively bounded sub- and super-solutions of $\delta V_{1} - F(e,X + D^{2}_{e} V_{1},y) \leq 0$ and $\delta V_{2} - F(e,X + D^{2}_{e} V_{2},y) \geq 0$ in $\mathbb{R}^{d}$, then the function $\{\tilde{V}_{1,y}\}_{y \in \mathbb{R}^{d}}$ and $\{\tilde{V}_{2,y}\}_{y \in \mathbb{T}^{d}}$ defined by analogy with \eqref{E: slices} lead to sub- and supersolutions of the equations
			\begin{equation*}
				\delta \tilde{V}_{1,y} - F(e,\tilde{X}_{e} + D^{2}_{e}\tilde{V}_{1,y},y + x') \leq 0, \quad \delta \tilde{V}_{2,y} - F(e,\tilde{X}_{e} + D^{2}_{e} \tilde{V}_{2,y},y + x') \geq 0 \quad \text{in} \, \, \langle e \rangle^{\perp}.
			\end{equation*}
		Thus, the functions $\{\tilde{V}_{0,y}\}_{y \in \mathbb{T}^{d}}$ given by $\tilde{V}_{0,y} = \tilde{V}_{1,y} - \tilde{V}_{2,y}$ satisfy
			\begin{equation*}
				\delta \tilde{V}_{0,y} - \mathscr{P}^{+}_{e,\frac{\lambda}{d -1}, \Lambda}(D^{2}_{e} \tilde{V}_{0,y}) \leq 0 \quad \text{in} \, \, \langle e \rangle^{\perp}.
			\end{equation*}
		Here the Pucci maximal operator $\mathscr{P}^{+}_{e,\frac{\lambda}{d-1},\Lambda}$ is given by 
			\begin{equation*}
				\mathscr{P}^{+}_{e,\frac{\lambda}{d - 1},\Lambda}(X) = \sup \left\{ \text{tr}(A \tilde{X}_{e}) \, \mid \, A \in \mathcal{S}_{d}, \, \, (d-1)^{-1} \lambda \text{Id} \leq A \leq \Lambda \text{Id} \right\}.
			\end{equation*}
		Since we are working in $\langle e \rangle^{\perp}$, this is a uniformly elliptic, translationally invariant operator.  Therefore, a standard comparison argument shows that $\tilde{V}_{0,y} \leq 0$ for all $y \in \mathbb{R}^{d}$.  In particular, $V_{1} \leq V_{2}$ in $\mathbb{R}^{d}$.  
		
	Applying Perron's Method, we find a continuous $V$ satisfying $\delta V - F(e,X + D^{2}_{e} V,y) = 0$ in $\mathbb{R}^{d}$.  By uniqueness, $V$ descends to a function in $C(\mathbb{T}^{d})$.       \end{proof}  

\subsection{Comparison Principle for Ergodic Constants}  Here we state and prove the comparison principles for ergodic constants that were used in Sections \ref{S: approximate correctors} and \ref{S: discontinuity}.  The idea goes back at least as far as \cite{evans}.  

	\begin{lemma} \label{L: comparison ergodic constants}  Assume that $F$ satisfies the assumptions of Theorem \ref{T: level set PDE} or $F(p,X,y) = m(y,\hat{p})^{-1} \text{tr}(a(y,\hat{p}) \tilde{X}_{\hat{p}})$ with $a$ and $m$ satisfying the assumptions of Theorem \ref{T: quasilinear}.  Fix $e \in S^{d-1}$, $g : \langle e \rangle^{\perp} \to \mathbb{R}$ is a bounded, uniformly continuous function, $X \in \mathcal{S}_{d}$, and $C_{1},C_{2} \in \mathbb{R}$.  If $V_{1} : \langle e \rangle^{\perp} \to \mathbb{R}$ is a bounded, upper semi-continuous subsolution of the differential inequality
		\begin{equation*}
			C_{1} + g(x') - F(e, \tilde{X}_{e} + D^{2}_{e} V_{1},x') \leq 0 \quad \text{in} \, \, \langle e \rangle^{\perp}
		\end{equation*}
	and $V_{2} : \langle e \rangle^{\perp} \to \mathbb{R}$ is a bounded, lower semi-continuous supersolution of the differential inequality
		\begin{equation*}
			C_{2} + g(x') - F(e,\tilde{X}_{e} + D^{2}_{e} V_{2},x') \geq 0 \quad \text{in} \, \, \langle e \rangle^{\perp},
		\end{equation*}
	then $C_{1} \leq C_{2}$.  \end{lemma}  
	
		\begin{proof}  The assumptions of Theorem \ref{T: level set PDE} and \ref{T: quasilinear} imply that the difference $V_{0} = V_{1} - V_{2}$ is a bounded, upper semi-continuous subsolution of the differential inequality
			\begin{equation*}
				(C_{1} - C_{2}) - \mathscr{P}^{+}_{e,\frac{\lambda}{d-1},\Lambda}(D^{2}_{e}V_{0}) \leq 0 \quad \text{in} \, \, \langle e \rangle^{\perp},
			\end{equation*}
		where the Pucci maximal operator $\mathscr{P}^{+}_{e,\frac{\lambda}{d-1},\Lambda}$ is as in the previous proof.  Therefore, another standard comparison argument shows that $C_{1} \leq C_{2}$.  
		\end{proof}  
		
\subsection{Equi-distribution of codimension one sub-tori}  We will be interested in certain probability measures supported on the sub-tori $(\mathbb{T}^{d-1}_{e}(s))_{s \in [0,r_{e})}$ defined by \eqref{E: subtori}.  Toward that end, the result that follows is fundamental.

In this section, if $\mu$ is a finite measure, we write $\fint_{A} (\cdot) \, \mu(dy) = \frac{1}{\mu(A)} \int_{A} (\cdot) \, \mu(dy)$.  

	\begin{lemma} \label{L: equidistribution} If $(e_{n})_{n \in \mathbb{N}} \subseteq S^{d-1} \cap \mathbb{R} \mathbb{Z}^{d}$ is any infinite sequence and $(s_{n})_{n \in \mathbb{N}} \subseteq [0,\infty)$ satisfies $s_{n} \in [0,r_{e_{n}})$ for each $n \in \mathbb{N}$, then 
		\begin{equation} \label{E: equidistribution}
			\mathcal{H}^{d-1}(\mathbb{T}^{d-1}_{e_{n}}(s_{n}))^{-1} \mathcal{H}^{d-1} \restriction_{\mathbb{T}^{d-1}_{e}(s_{n})} \overset{*}{\rightharpoonup} \mathcal{L}^{d}
		\end{equation}
	Furthermore, if $(f_{n})_{n \in \mathbb{N}}$ are $\mathcal{H}^{d-1}$-measurable functions in $\mathbb{T}^{d}$, $p \in [1,\infty]$, and $C > 0$ is such that
		\begin{equation*}
			\fint_{\mathbb{T}^{d-1}_{e_{n}}(s_{n})} |f_{n}(\xi)|^{p} \, \mathcal{H}^{d-1}(d \xi)  \leq C^{p}
		\end{equation*} 
	and if we define measures $(\mu_{n})_{n \in \mathbb{N}}$ on $\mathbb{T}^{d}$ by 
		\begin{equation*}
			\int_{\mathbb{T}^{d}} g(y) \, \mu_{n}(dy) = \fint_{\mathbb{T}^{d - 1}_{e_{n}}(s_{n})} g(\xi) f_{n}(\xi) \, \mathcal{H}^{d-1}(d \xi),
		\end{equation*}
	then there is a subsequence $(n_{j})_{j \in \mathbb{N}} \subseteq \mathbb{N}$ and a measure $\tilde{\mu}$ such that $\tilde{\mu} = \lim_{j \to \infty} \mu_{n_{j}}$ weakly-$*$, $\tilde{\mu} \ll \mathcal{L}^{d}$, and 
		\begin{equation} \label{E: lebesgue bounds}
			\left\| \frac{d \tilde{\mu}}{d \mathcal{L}^{d}} \right\|_{L^{p}(\mathbb{T}^{d})} \leq C.
		\end{equation}
	\end{lemma}  
	
		\begin{proof}  First, we prove that the normalized surface measures converge to $\mathcal{L}^{d}$.  Assume that $g \in C(\mathbb{T}^{d})$ satisfies $\sum_{k \in \mathbb{Z}^{d}} |\hat{g}(k)| < \infty$.
		An exercise in Fourier analysis shows that if $k \in \mathbb{Z}^{d}$, then
			\begin{equation*}
				\fint_{\mathbb{T}^{d-1}_{e_{n}}(s_{n})} e^{i 2 \pi \langle k, \xi \rangle} \mathcal{H}^{d-1}(d \xi) = \left\{ \begin{array}{r l}
							e^{i 2 \pi \langle k, e_{n} \rangle s_{n}}, & \text{if} \, \, k \in \langle e_{n} \rangle, \\
							0, & \text{otherwise}.
						\end{array} \right.
			\end{equation*}
		Therefore, 
			\begin{align*}
				\left| \fint_{\mathbb{T}^{d-1}_{e_{n}}(s_{n})} g(\xi) \, \mathcal{H}^{d-1}(d\xi) - \int_{\mathbb{T}^{d}} g(y) \, dy \right| &= \left| \sum_{k \in \mathbb{Z}^{d} \setminus \{0\}} \hat{g}(k) \fint_{\mathbb{T}^{d-1}_{e_{n}}(s_{n})} e^{i 2 \pi \langle k, \xi \rangle} \mathcal{H}^{d-1}(d \xi) \right| \\
					&\leq \sum_{k \in \mathbb{Z}^{d} \cap \langle e_{n} \rangle \setminus \{0\}} |\hat{g}(k)|.
			\end{align*}
		Since $(e_{n})_{n \in \mathbb{N}}$ is infinite, it follows that for each $R > 0$, there is an $N \in \mathbb{N}$ such that if $n \geq N$, then
			\begin{equation*}
				\mathbb{Z}^{d} \cap \langle e_{n} \rangle \setminus \{0\} \subseteq \mathbb{R}^{d} \setminus B(0,R).
			\end{equation*}
		Thus,
			\begin{align*}
				\limsup_{n \to \infty} \left| \fint_{\mathbb{T}^{d-1}_{e_{n}}(s_{n})} g(\xi) \, \mathcal{H}^{d-1}(d\xi) - \int_{\mathbb{T}^{d}} g(y) \, dy \right| &\leq \lim_{R \to \infty} \sum_{k \in \mathbb{Z}^{d} \setminus B(0,R)} |\hat{g}(k)| = 0.
			\end{align*}
			
		Recalling that functions with summable Fourier coefficients are dense in $C(\mathbb{T}^{d})$, we conclude that \eqref{E: equidistribution} holds as claimed.
		
	Next, we prove the claim concerning $(\mu_{n})_{n \in \mathbb{N}}$.  First, observe that H\"{o}lder's inequality implies 
		\begin{equation*}
			\|\mu_{n}\|(\mathbb{T}^{d}) = \fint_{\mathbb{T}^{d-1}_{e_{n}}(s_{n})} |f_{n}(\xi)| \, \mathcal{H}^{d-1}(d \xi) \leq C \quad \text{if} \, \, n \in \mathbb{N}.
		\end{equation*}
	Thus, $(\mu_{n})_{n \in \mathbb{N}}$ is pre-compact in $C(\mathbb{T}^{d})^{*}$, which gives the desired sub-sequence $(n_{j})_{j \in \mathbb{N}}$ and limit point $\tilde{\mu}$.
	
	Note that if $g \in C(\mathbb{T}^{d})$ and $q \in [1,\infty]$ is the conjugate exponent of $p$ (i.e.\ the $q$ so that $p^{-1} + q^{-1} = 1$), then 
		\begin{equation*}
			\left| \fint_{\mathbb{T}^{d-1}_{e_{n}}(s_{n})} g(\xi) f_{n}(\xi) \, \mathcal{H}^{d-1}(d\xi) \right| \leq C \left( \fint_{\mathbb{T}^{d-1}_{e_{n}}(s_{n})} |g(\xi)|^{q} \, \mathcal{H}^{d-1}(d \xi) \right)^{\frac{1}{q}}.
		\end{equation*}
	Assume first that $q < \infty$.  Since $|g|^{q} \in C(\mathbb{T}^{d})$, we have
		\begin{equation*}
			\left| \int_{\mathbb{T}^{d}} g(y) \, \tilde{\mu}(d y) \right| \leq C \lim_{n \to \infty} \left( \fint_{\mathbb{T}^{d-1}_{e_{n}}(s_{n})} |g(\xi)|^{q} \, \mathcal{H}^{d-1}(d \xi) \right)^{\frac{1}{q}} = C \|g\|_{L^{q}(\mathbb{T}^{d})}.
		\end{equation*}
	This proves $\tilde{\mu} \ll \mathcal{L}^{d}$ and \eqref{E: lebesgue bounds} holds when $q < \infty$.  When $q = \infty$, a similar argument applies.       \end{proof}  
		
\subsection{Smooth Correctors in a Particular Case}  \label{S: diophantine}

In this section, we prove that \eqref{E: cell problem} does have solutions in the setting of Theorem \ref{T: quasilinear} when $a$ is constant, $m$ is sufficiently regular, and $e$ is well-chosen. 
		
\begin{prop} \label{P: smooth correctors}  Fix $e \in S^{d-1}$.  If $m(\cdot,e) \in H^{s}(\mathbb{T}^{d})$ for some $s > \frac{d}{2} + \frac{1}{d- 1} + 2$ and there is a $C_{e} \in (0, 1)$ and $\frac{1}{d - 1} < \tau < s - \frac{d}{2} -2$ such that
		\begin{equation} \label{E: diophantine}
		\|k - \langle k,e \rangle e\| \geq C_{e}\|k\|^{-\tau} \quad (\forall k \in \mathbb{Z}^{d} \setminus \{0\}),
		\end{equation}
	then there is a solution $V_{e} \in C^{2}(\mathbb{T}^{d})$ of the equation
		\begin{equation} \label{E: smooth corrector}
			 m(\cdot,e) - \text{tr}((\text{Id} - e \otimes e) D^{2}V_{e}) = \overline{m}(e).
		\end{equation}
	Furthermore, $V_{e}$ is the unique such solution among all functions $U \in L^{2}(\mathbb{T}^{d})$ with $\int_{\mathbb{T}^{d}} U(y) \, dy = \int_{\mathbb{T}^{d}} V_{e}(y) \, dy$.
\end{prop}  

Concerning the generality of the assumption \eqref{E: diophantine}, see \cite{homogenization and boundary layers}, where it is shown that $\mathcal{H}^{d-1}$-a.e.\ $e \in S^{d-1}$ satisfies such an estimate for any given $\tau > \frac{1}{d -1}$.  More precisely, if $A(C_{e},\tau)$ is the set of all such $e$, then there is a constant $B(d,\tau) > 0$ such that
	\begin{equation*}
		\mathcal{H}^{d-1}(S^{d-1} \setminus A(C_{e},\tau)) \leq B(d,\tau) C_{e}^{d-1}.
	\end{equation*}

	\begin{proof}  Define $\hat{V}_{e} : \mathbb{Z}^{d} \to \mathbb{C}$ by $\hat{V}_{e}(0) = 0$ and
		\begin{equation*} 
			\hat{V}_{e}(k) = - \frac{\hat{m}(k)}{4 \pi^{2} \|k - \langle k, e \rangle e\|^{2}}.
		\end{equation*}
	Since $m(\cdot,e) \in H^{s}(\mathbb{T}^{d})$ and $s > \tau + \frac{d}{2} + 2$, for each $i \in \{0,1,2,\}$, we have
		\begin{equation*}
			\sum_{k \in \mathbb{Z}^{d}} \|k\|^{i} |\hat{V}_{e}(k)| \leq C_{e}^{-2} \left(\sum_{k \in \mathbb{Z}^{d}}  \|k\|^{2(\tau + i - s)} \right)^{\frac{1}{2}} \left(\sum_{k \in \mathbb{Z}^{d}} \|k\|^{2s} |\hat{m}(k,e)|^{2} \right)^{\frac{1}{2}} < \infty.
		\end{equation*}
	Thus, we can define $V_{e} \in C^{2}(\mathbb{T}^{d})$ by 
		\begin{equation*}
			V_{e}(y) = \sum_{k \in \mathbb{Z}^{d} \setminus \{0\}} \hat{V}_{e}(k) e^{i2 \pi \langle k ,y \rangle}
		\end{equation*}  
	and then 
		\begin{equation*}
			(\text{Id} - e \otimes e) D^{2}V_{e}(y) = - \sum_{k \in \mathbb{Z}^{d}} 4 \pi^{2}( k \otimes k - \langle k, e \rangle e \otimes k) \hat{V}_{e}(k) e^{i2 \pi \langle k, y \rangle}.
		\end{equation*}
	In particular, by construction, $V_{e}$ is a solution of \eqref{E: smooth corrector}.  Since $k - \langle k,e \rangle e \neq 0$ for each $k \in \mathbb{Z}^{d}$, a straightforward argument shows that $V$ is unique up to the addition of a constant.
	\end{proof}  
	

\section{Construction of Some Finsler Norms} \label{A: geometric construction}

To construct the Finsler norms of Proposition \ref{P: norm}, we proceed by constructing appropriate convex perturbations of the unit ball $B(0,1)$.

Recall that if $\mathcal{O}$ is a convex, open subset of $\mathbb{R}^{d}$ and $0 \in \mathcal{O}$, then there is a unique Finsler norm $\psi_{\mathcal{O}}$ such that $\{\psi_{\mathcal{O}} < 1\} = \mathcal{O}$.  We will prove Proposition \ref{P: norm} by building a sequence $\{\mathcal{O}_{N}\}$ of such sets, each of which is obtained from $B(0,1)$ by flattening its boundary at the points $\{e_{1},\dots,e_{N}\}$.  Once this is done, it will be easy to see that the corresponding Finsler norms $\{\psi_{N}\}_{N \in \mathbb{N}}$ have the desired properties.

In what follows, we define $\varphi : [-1,1] \to [0,1]$ by $\varphi(x) = \sqrt{1 - x^{2}}$.  We will take advantage of the fact that if $e \in S^{d-1}$, then the map $B(0,1) \cap \langle e \rangle^{\perp} \ni x' \mapsto x' + \varphi(x')e$ parametrizes a neighborhood of $e$ in $S^{d-1}$.  In particular, we build $\mathcal{O}_{N}$ by perturbing these maps.

\subsection{Graphs}  To start with, we modify $\varphi$ so that it is flat at the points $\{e_{1},\dots,e_{N}\}$.  Recall that $\varphi'$ and $\varphi''$ are given by
	\begin{equation*}
		\varphi'(x) = -\frac{x}{\sqrt{1 - x^{2}}}, \quad \varphi''(x) = - \frac{1}{\sqrt{1 - x^{2}}} + \frac{x^{2}}{(1 - x^{2})^{\frac{3}{2}}}.
	\end{equation*}  

Let $\eta : \mathbb{R} \to [0,1]$ be a smooth function satisfying the following:
	\begin{equation*}
		\eta(x) = 1 \quad \text{if} \, \, |x| \leq 1, \quad \eta(x) = 0 \quad \text{if} \, \, |x| \geq 2, \quad \eta(-x) = \eta(x), \quad \eta' > 0 \quad \text{in} \, \, (-2,-1).
	\end{equation*}
	
Given $\epsilon > 0$, define the modified parametrization $\varphi_{\epsilon} : [-1,1] \to \mathbb{R}$ by 
	\begin{equation*}
		\varphi_{\epsilon}(x) = \int_{-1}^{x} \left(1 - \eta(\epsilon^{-1}y)\right) \varphi'(y) \, dy.
	\end{equation*}

	\begin{prop} \label{P: main construction}  $\varphi_{\epsilon}$ is $C^{2}$ in $[-1,1]$ and it satisfies the following conditions:
		\begin{itemize}
			\item[(i)] $0 \leq \varphi_{\epsilon} \leq 1$ in $[-1,1]$,
			\item[(ii)] $\varphi_{\epsilon}''$ is given by $\varphi_{\epsilon}''(x) = (1 - \eta(\epsilon^{-1}x)) \varphi''(x) - \epsilon^{-1} \eta'(\epsilon^{-1} x) \varphi'(x)$ and, for $\epsilon < 1/4$, we have
				\begin{equation*}
					\|\varphi_{\epsilon}''\|_{L^{\infty}([-1,1])} \leq \|\varphi''\|_{L^{\infty}([-1,1])} + \frac{4}{\sqrt{3}} \|\eta'\|_{L^{\infty}(\mathbb{R})}.
				\end{equation*}
			\item[(iii)] For each $x \in [-2\epsilon,2\epsilon]$, $\varphi_{\epsilon}(x) \geq \sqrt{1 - 4 \epsilon^{2}}$.
		\end{itemize}
	\end{prop}  
	
		\begin{proof}  First, notice that $\varphi_{\epsilon}(-1) = 0$.  Furthermore, since $\varphi_{\epsilon}'(-x) = -\varphi_{\epsilon}'(x)$, it follows that $\varphi_{\epsilon}(x) = \varphi_{\epsilon}(-x)$, which gives $\varphi_{\epsilon}(1) = 0$.  
		
		Next, observe that $\varphi_{\epsilon}'(x) = 0$ if and only if $|x| \leq \epsilon$.  Thus, the maximum of $\varphi_{\epsilon}$ is attained in $[-\epsilon,\epsilon]$.  Furthermore, 
			\begin{equation*}
				\|\varphi_{\epsilon}\|_{L^{\infty}([-1,1])} = \varphi_{\epsilon}(0) \leq \int_{-1}^{0} \varphi'(y) \, dy = 1.
			\end{equation*}
			
		Concerning (ii), the equation for $\varphi_{\epsilon}''$ is a direct consequence of the chain rule, and the second term can be estimated as follows:
			\begin{align*}
				\epsilon^{-1} |\eta'(\epsilon^{-1}x) \varphi'(x)| &\leq \|\eta'\|_{L^{\infty}(\mathbb{R})} \frac{2}{\sqrt{1 - x^{2}}} \chi_{[-2\epsilon,2\epsilon]}(x) \leq \frac{4}{\sqrt{3}} \|\eta'\|_{L^{\infty}(\mathbb{R})}.
			\end{align*}
		
		Finally, if $x \in [-2\epsilon,0]$, then
			\begin{equation*}
				\varphi_{\epsilon}(x) \geq \int_{-1}^{-2\epsilon} \varphi_{\epsilon}'(y) \, dy = \varphi(-2 \epsilon) = \sqrt{1 - 4 \epsilon^{2}}.
			\end{equation*}
		Since $\varphi_{\epsilon}(-x) = \varphi_{\epsilon}(x)$, this proves (iii).
	\end{proof}

\subsection{Construction of $\mathcal{O}_{N}$}  Let $\epsilon_{N} \in (0,1/4)$ be a free variable.  

We define the convex set $\mathcal{O}_{N} \subseteq B(0,1)$ as follows: given $j \in \{1,2,\dots,N\}$, define $\varphi^{e_{j}}_{N} : B(0,1) \cap \langle e_{j} \rangle^{\perp} \to \mathbb{R}$ by
	\begin{equation*}
		\varphi^{e_{j}}_{N}(x') = \varphi_{\epsilon_{N}}(\|x'\|).
	\end{equation*}Notice that $\varphi^{e_{j}}_{\epsilon}(x') e_{j} + x' = \sqrt{1 - \|x'\|^{2}} e_{j} + x'$ if $\|x'\| \geq 2 \epsilon_{N}$.  Thus, it follows that there is an $\epsilon_{N}' \in (0,1/4)$ such that if $\epsilon_{N} \in (0,\epsilon_{N}')$ and we define $\mathcal{G}_{N}$, $\mathcal{B}_{N}$, and $\mathcal{S}_{N}$ by 
	\begin{align*}
		\mathcal{G}_{N} &= \bigcup_{j = 1}^{N} \left\{\varphi^{e_{j}}_{N}(x') e_{j} + x' \, \mid \, x' \in \overline{B(0,2\epsilon_{N})} \cap \langle e_{j} \rangle^{\perp} \right\}, \\
		\mathcal{B}_{N} &= \bigcup_{j = 1}^{N} \left\{ \varphi(\|x'\|)e_{j} + x' \, \mid \, x' \in \overline{B(0,2\epsilon_{N})} \cap \langle e_{j} \rangle^{\perp} \right\}, \\
		\mathcal{S}_{N} &= \mathcal{G}_{N} \cup (S^{d-1} \setminus \mathcal{B}_{N}),
	\end{align*}
then $\mathcal{S}_{N}$ is a compact, connected $C^{2}$ hypersurface in $\mathbb{R}^{d}$.  Furthermore, for each $j \in \{1,2,\dots,N\}$, the function $x' \mapsto x' + \varphi_{N}^{e_{j}}(x') e_{j}$ is a parametrization of $\mathcal{S}_{N}$.  

Being connected, $\mathcal{S}_{N} = \partial \mathcal{O}_{N}$ for some bounded, connected open set $\mathcal{O}_{N}$ (see \cite[Section 2.5]{guillemin pollack}, \cite{mapping degree theory}, or argue directly using the fact that $\mathcal{S}_{N}$ is obtained by perturbing parametrizations of $S^{d-1}$).  Since $\varphi^{e_{j}}_{N}$ is concave for each $j$, it follows that, for each $x \in \overline{\mathcal{O}_{N}}$, there is an $r > 0$ such that $\overline{\mathcal{O}_{N}} \cap B(x,r)$ is convex.  Therefore, by the Tietze-Nakajima Theorem (cf.\ \cite[Section 2]{tietze nakajima exposition}), $\mathcal{O}_{N}$ is a convex subset of $\mathbb{R}^{d}$.

Next, we note that, by making $\epsilon_{N} > 0$ smaller if necessary, we can ensure that the outward normal $\nu$ to $\mathcal{S}_{N}$ points in the $e_{j}$ direction for some $j \in \{1,2,\dots,N\}$ if and only if the corresponding point is in the flat part of $\varphi^{e_{j}}_{\epsilon_{N}}$.  This is the content of the next result:

	\begin{prop} \label{P: key proposition for flat parts}  Let $\nu$ denote the outward normal vector to $\partial \mathcal{O}_{N}$.  Making $\epsilon_{N}$ smaller if necessary, given $p \in \mathcal{S}_{N}$, we have $\nu(p) = e_{j}$ for some $j \in \{1,2,\dots,N\}$ if and only if $p = x' + \varphi^{e_{j}}_{N}(x') e_{j}$ for some $x' \in \overline{B(0,\epsilon_{N})} \cap \langle e_{j} \rangle^{\perp}$.  \end{prop}
	
		\begin{proof}  Fix $k \in \{1,2,\dots,N\}$.  If $x' \in B(0,1) \cap \langle e_{k} \rangle^{\perp}$, $\|x'\| \geq 2 \epsilon_{N}$, and $\tilde{p} := \varphi^{e_{k}}_{N}(x') e_{k} + x' \in \mathcal{S}_{N}$, then $\nu(\tilde{p}) = \tilde{p}$.  Thus, by construction, $\nu(p) = e_{k}$ only if $p = \varphi^{e_{k}}_{N}(x') e_{k} + x'$ for some $x' \in B(0,2\epsilon_{N}) \cap \langle e_{k} \rangle^{\perp}$.  
		
		It follows that $\nu(p)$ is given by 
			\begin{equation*}
				\nu(p) = \frac{e_{k} - D\varphi^{e_{k}}_{N}(x')}{\sqrt{1 + \|D\varphi^{e_{k}}_{N}(x')\|^{2}}}.
			\end{equation*}
		However, if $\epsilon_{N}$ is small enough, then the relation $e_{\ell} \neq e_{j}$ for $\ell \neq j$ forces $k = j$.  Indeed,
			\begin{align*}
				\left\| \nu(p) - e_{k} \right\| &\leq (\sqrt{1 + \|D\varphi^{e_{k}}_{\epsilon_{N}}(x')\|^{2}} - 1) + \|D\varphi^{e_{k}}_{\epsilon_{N}}(x')\|.
			\end{align*}
		Since there is a $C > 0$ independent of $\epsilon_{N}$ and $k$ such that $\|D\varphi^{e_{k}}_{N}(x')\| \leq C \epsilon_{N}$ for $x' \in B(0,2 \epsilon_{N}) \cap \langle e_{k} \rangle^{\perp}$, it follows that there is an $\epsilon_{N}'' \in (0,1/2)$ such that if $\epsilon_{N} < \epsilon_{N}''$, then $\nu(p) = e_{j}$ only if $k = j$.  
		
		Finally, we claim that $\|x'\| \leq \epsilon_{N}$.  Indeed, if $\epsilon_{N} < \|x'\| \leq 2\epsilon_{N}$, then $\nu(p)$ is given by 
			\begin{equation*}
				\nu(p) = \frac{e_{j} + (1 - \eta(\epsilon_{N}^{-1} x'))(1 - \|x'\|^{2})^{-\frac{1}{2}} x'}{\sqrt{1 + \|D\varphi^{e_{k}}_{N}(x')\|^{2}}}
			\end{equation*}
		and this gives $(\text{Id} - e_{j} \otimes e_{j}) \nu(p) \neq 0$.  Therefore, $\nu(p) = e_{j}$ only if $\|x'\| \leq \epsilon_{N}$.  
		\end{proof}    

%
%
%
%
\subsection{Proof of Proposition \ref{P: norm}}  Let $\psi_{N} : \mathbb{R}^{d} \to [0,\infty)$ be the Minkowski gauge associated with $\mathcal{O}_{N}$, that is,
	\begin{equation*}
		\psi_{N}(p) = \inf \left\{ \alpha > 0 \, \mid \, \alpha^{-1} p \in \mathcal{O}_{N} \right\}.
	\end{equation*}
$\psi_{N}$ is the unique Finsler norm with $\{\psi_{N} < 1\} = \mathcal{O}_{N}$.  It has the following properties:

	\begin{prop}  (i)  For each $e \in S^{d-1}$, $1 \leq \psi_{N}(e) \leq (1 - 4 \epsilon_{N}^{2})^{-\frac{1}{2}}$.  
	
	(ii)  Given $p \in \mathbb{R}^{d} \setminus \{0\}$, if $\widehat{D\psi}_{N}(p) = e_{i}$ for some $i \in \{1,2,\dots,N\}$, then 
		\begin{equation*}
			D^{2}\psi_{N}(p) = 0.
		\end{equation*}
	
	(iii)  There is a $c_{0} > 0$ such that $D^{2}\psi_{N}(e) \leq c_{0} (\text{Id} - e \otimes e)$ for each $e \in S^{d-1}$.    
	\end{prop}  
	
	Notice that this implies Proposition \ref{P: norm}.  
	
		\begin{proof}  First, we prove (i).  Assume that $e \in S^{d-1}$.  Choose $\kappa > 0$ such that $\kappa e \in \partial \mathcal{O}_{N}$.  If $\kappa e = \varphi_{N}^{e_{i}}(x')e_{i} + x'$ for some $i \in \{1,2,\dots,N\}$ and $x' \in B(0,1) \cap \langle e_{i} \rangle^{\perp}$, then Proposition \ref{P: main construction}, (iii) implies
			\begin{equation*}
				1 \geq \|\kappa e\|^{2} = \|x'\|^{2} + \varphi_{\epsilon_{N}}^{e_{i}}(x')^{2} \geq 1 - 4 \epsilon^{2}.
			\end{equation*}
		From this, we find 
			\begin{equation*}
				1 \leq \kappa^{-1} = \psi_{N}(e) \leq \frac{1}{\sqrt{1 - 4 \epsilon^{2}}}.
			\end{equation*}
		Otherwise, if $\kappa e \in S^{d-1} \cap \partial \mathcal{O}_{N}$, then $\kappa = 1$  and $\psi_{N}(e) = \psi_{N}(\kappa e) = 1$.
			
		Next, we tackle (ii).  Suppose that $p \in \mathbb{R}^{d}$ and $\widehat{D \psi}_{N}(p) = e_{i}$ for some $i \in \{1,2,\dots,N\}$.  By homogeneity, there is a $\gamma > 0$ such that $D\psi_{N}(p) = \gamma \nu(\psi_{N}(p)^{-1}p)$.  Hence $\nu(\psi_{N}(p)^{-1} p) = e_{i}$ and we can invoke Proposition \ref{P: key proposition for flat parts} to find that $p = x' + \varphi^{e_{i}}_{N}(x') e_{i}$ for some $x' \in \overline{B(0,\epsilon_{N})} \cap \langle e_{i} \rangle^{\perp}$.  From this, the flatness of $\varphi_{N}^{e_{i}}$ in $B(0, \epsilon_{N}) \cap \langle e_{i} \rangle^{\perp}$ implies that $D^{2}\psi_{N}(\psi_{N}(p)^{-1} p) = 0$, and then homogeneity implies $D^{2}\psi_{N}(p) = 0$.  
		
		Finally, concerning (iii), we note that the $\epsilon$-independent bounds on $\varphi''_{\epsilon}$ in Proposition \ref{P: main construction} give corresponding $N$-independent bounds on $D^{2}\varphi^{e_{j}}_{N}$, and then this readily shows that $D^{2}\psi_{N}$ is bounded on $\{\psi_{N} = 1\}$ independently of $N$.  Using homogeneity and (i), this gives an $N$-independent bound on $D^{2}\psi_{N}$ in $S^{d-1}$.  
		\end{proof}

\section*{Acknowledgements}  

It is a pleasure to acknowledge P.E.\ Souganidis for helpful discussions and considerable patience.  The author would also like to thank W.M. Feldman and I.C. Kim, whose correspondence and suggestions led to a number of significant improvements to the paper, ultimately culminating in the proof of Theorem \ref{T: comparison}.

\end{document}